\documentclass[11pt]{amsart}
\usepackage[utf8]{inputenc}
\usepackage[T1]{fontenc}
\usepackage{bm,amsmath,amsthm,amssymb}
\usepackage{geometry}
\usepackage{graphicx}
\usepackage{tikz}
\usepackage{amsfonts}
\usepackage{array}
\usepackage{enumerate}
\usepackage{float}
\usepackage{lscape}
\usepackage[english]{babel}
\usepackage[all]{xy}
\usepackage{authblk}
\usepackage{hyperref}

\theoremstyle{plain}
\newtheorem{fact}{Fact}[section]
\newtheorem{theo}[fact]{Theorem}
\newtheorem{lem}[fact]{Lemma}
\newtheorem{defi}[fact]{Definition}

\newtheorem{prop}[fact]{Proposition}
\newtheorem{rmk}[fact]{Remark}
\newtheorem{coro}[fact]{Corollary}
\newcommand{\cor}[1]{\textcolor{black}{#1}}

\title{A bijection between $m$-cluster-tilting objects and $(m+2)$-angulations in $m$-cluster categories}
\author{Lucie JACQUET-MALO\\
lucie.jacquet.malo@u-picardie.fr\\
LAMFA\\
33 rue Saint Leu\\
80 000 AMIENS\\
FRANCE}

\begin{document}

\maketitle

\begin{center}
\large{\sc{Lucie Jacquet-Malo}} \\
lucie.jacquet.malo@u-picardie.fr\\
\end{center}

\begin{abstract}
In this article, we study the geometric realizations of $m$-cluster categories of Dynkin types A, D, $\tilde{A}$ and $\tilde{D}$. We show, in those four cases, that there is a bijection between $(m+2)$-angulations and isoclasses of basic $m$-cluster tilting objects. Under these bijections, flips of $(m+2)$-angulations correspond to mutations of $m$-cluster tilting objects. Our strategy consists in showing that certain Iyama-Yoshino reductions of the $m$-cluster categories under consideration can be described in terms of cutting along an arc the corresponding geometric realizations. This allows to infer results from small cases to the general ones.
\end{abstract}

\textbf{
Keywords: Cluster algebras, $m$-cluster categories, tame quivers, $\tilde{D_n}$.}

\textbf{MSC classification: Primary: 18E30 ; Secondary: 13F60, 05C62
}


\section*{Introduction}

In the early 2000's, Fomin and Zelevinsky in \cite{FZ} invented cluster algebras in order to give a combinatorial framework to the study of canonical bases. Later, it has been proved that cluster algebras have many connections with Calabu-Yau algebras, integrable systems, Poisson geometry and quiver representations. In order to categorify this notion, Buan, Marsh, Reineke, Reiten and Todorov in \cite{BMRRT} (and Caldero, Chapoton, Schiffler in case $A_n$ in \cite{CCS}) invented cluster categories. This allowed to categorify mutations in cluster algebras by using tilting theory. For a gentle introduction to cluster categories, see the article of Keller, \cite{Kel01}.

\vspace{20pt}

The cluster category is defined as follows. Let $K$ be a field, $Q$ be an acyclic quiver, and $\mathcal{D}^b(KQ)$ the bounded derived category of $Q$. The cluster category is the orbit of $\mathcal{D}^b(KQ)$ under the functor $\tau^{-1}[1]$, where $\tau$ is the Auslander-Reiten translation and $[1]$ is the shift. Keller showed in \cite{Kel03} that the cluster category is triangulated, with shift functor $[1]$. In fact, he proved that, for nice enough endofunctors, the orbit category of a derived category was triangulated. This led to the higher cluster category: the category \[\mathcal{D}^b(KQ)/\tau^{-1}[m].\]
Thomas in \cite{Tho} defined them properly, and showed that they played the same role as cluster categories, but with respect to $m$-clusters (defined by Fomin and Reading in \cite{FR}. Later, Wraalsen and Zhou/Zhu in \cite{W} and \cite{ZZ} showed that many properties of cluster categories could be generalized to higher cluster categories. For example, they showed that any $m$-rigid object $X$ having $n-1$ nonisomorphic indecomposable summands has exactly $m+1$ complements (it means nonisomorphic indecomposable objects $Y$ such that $X \bigoplus Y$ is an $m$-cluster-tilting object).

For some specific classes of quivers, it is sometimes possible to construct geometric realizations of (higher) cluster categories. This was done by Caldero, Chapoton and Schiffler in case $A_n$ for cluster categories. Schiffler in \cite{Sch} found it for case $D_n$ for cluster categories, and Baur and Marsh generalized both results to higher cluster categories in \cite{BM01} and \cite{BM02}. In these cases, the Auslander-Reiten quiver of the higher cluster category can be realized as a connected component of a category geometrically built.

Unfortunately, this cannot happen in Euclidean cases, which are representation-infinite. It means that the Auslander-Reiten quiver of the higher cluster category is infinite, and composed of three main parts, which repeat many times. Torkildsen in \cite{Tor} treated case $\tilde{A}$, and Baur and Torkildsen in \cite{BauTor} gave a complete geometric realization of case $\tilde{A}$.

This paper is aimed to give further results on all realization of higher cluster categories. Indeed, we are going to show that there is a bijection between well-defined $(m+2)$-angulations, and the $m$-cluster-tilting objects in the higher cluster category, using the following strategy:

First, we recall the combinatorial framework for types $A$, $D$, $\tilde{A}$ and $\tilde{D}$. We define the underlying surface, the concepts of admissible arcs, $m$-diagonal, and $(m+2)$-angulation associated. Second, with each $(m+2)$-angulation in the geometric realization, we define the associated colored quiver. We also state the compatibility Theorem \ref{theo:corresp} between the flip of an $(m+2)$-angulation and the mutation of the associated colored quiver for all types.

In Theorem \ref{th:ra}, we state that, for all types, there is a bijective correspondence between the $m$-rigid indecomposable objects of the $m$-cluster category, and the $m$-diagonals in the underlying geometric model.

Next, in Theorem \ref{th:cross}, we show that the vanishing of the first $m$ positive extension groups of a pair of $m$-rigid indecomposable objects implies that the corresponding $m$-diagonals in the geometric realization do not cross. In order to show this result, we state that if we cut along a special arc $\alpha$ (called $m$-ears, which means, roughly speaking, when cut-off, produces a piece of the original polygon $P$, that has the same shape as $P$, but whose $(m+2)$-angulations have one less $m$-diagonal than those of $P$), then the Iyama-Yoshino reduction with respect to $\alpha$ is equivalent to the $m$-cluster categoryof the same type $A$, $D$, $\tilde{A}$, or $\tilde{D}$, but with one less vertex.

This Theorem \ref{th:ra} serves to associate with each $m$-cluster-tilting object an $(m+2)$-angulation in $P$, and we then show in Theorem \ref{th:comp} that this assignement makes the mutations of $m$-cluster-tilting objects and the flips of $(m+2)$-angulations correspond mutually to each other.

Finally, we establish the converse of Theorem \ref{th:ra}, namely, that if two $m$-diagonals do nos cross, then the first $m$-positive extension groups between their corresponding $m$-rigid objects vanish. This permits us to show the final Theorem \ref{th:bij}.

This paper is organized as follows.

In section $1$, we recall some important notions on higher cluster categories, mutation of $m$-rigid objects and colored quivers.

Section $2$ is a survey of all the geometric descriptions of types $A$, $D$, $\tilde{A}$ and $\tilde{D}$, with a slight modification on type $D$. We also see the bijection between $m$-rigid objects and $m$-diagonals.

In section $3$, we show in each type of quiver that, if two arcs cross each other, then there exists a nonzero extension between the associated $m$-rigid objects.

Finally in section $4$ we show the compatibility between mutations of $m$-cluster-tilting objects and flips of $(m+2)$-angulations. \cor{The results in section $3$ allows us to define a function which sends an $(m+2)$-angulation to the correspondent $m$-cluster-tilting object. We also show that this function is in fact a bijection.}

\subsection*{Acknowledgements}
This article is part of my PhD thesis under the supervision of Yann Palu and Alexander Zimmermann. I would like to thank Yann Palu warmly for introducing me to the subject of cluster categories, and for his patience and kindness. I also would like to thank warmly the anonymous referee for his kind comments and advises.

\section{Preliminaires}

Notations:

Throughout this paper, we fix a field K and an acyclic finite quiver $Q$.
In the remaining of the paper, $n$ and $m$ are integers, where $n$ is the number of vertices of $Q$, $n \geq 4$. \cor{We note that all the results apply to the cases $A_1$, $A_2$, and $A_3$ using exactly the same arguments.}

If $A$ is an object in a category ${\mathcal{C}}$, $A^\perp$ is the class of all objects $X$ such that \[{\mathrm{Ext}}_\mathcal{C}^i(A,X)=0 \text{ for all } i \in \{1,\cdots,m \}.\]

The category ${\mathrm{mod}}(KQ)$ is the category of finitely generated right modules over the path algebra $KQ$. The letter $\tau$ stands for the Auslander-Reiten translation. We write $[1]$ for the shift functor in the bounded derived category $\mathcal{D}^b(KQ)$. For any further information about representation theory of associative algebras, see the book written by Assem, Simson and Skowronski, \cite{ASS}.

\subsection{Higher cluster categories}
In 2006, in order to categorify the notions of clusters and mutations in cluster algebras, Buan, Marsh, Reineke, Reiten and Todorov in \cite{BMRRT} defined the cluster category of an acyclic quiver in the following way:

If $Q$ is an acyclic quiver, let ${\mathcal{D}}^b(KQ)$ be the bounded derived category of the category ${\mathrm{mod}}~KQ$. The category ${\mathcal{C}}_Q$ is the orbit category (in the sense of Keller in \cite{Kel03}) of the derived category under the functor $\tau^{-1}[1]$.

Cluster categories give a categorification of clusters in a cluster algebra in terms of cluster-tilting objects. To be precise, the cluster variables of the cluster algebra are in $1-1$ correspondence with the rigid indecomposable objects in ${\mathcal{C}}_Q$ \cor{(recalling that an object is rigid if it has no self-extensions)}, and the clusters are in $1-1$ correspondence with the isoclasses of basic cluster-tilting objects in ${\mathcal{C}}_Q$.

It is known from Buan, Marsh, Reineke, Reiten and Todorov in \cite{BMRRT} that ${\mathcal{C}}_Q$ is Krull-Schmidt. Since $\tau$ and $[1]$ become isomorphic in $\mathcal{C}_Q$, we have that ${\mathcal{C}}_Q$ is $2$-Calabi-Yau, and Keller in \cite{Kel03} has shown that it was a triangulated category.

For a positive integer $m$, following Thomas in \cite{Tho} and Keller in \cite{Kel03}, we can also define the higher cluster category
\[ {\mathcal{C}}^m_Q={\mathcal{D}}^b(KQ)/\tau^{-1}[m] \] where $[m]$ is the shift $[1]$ repeated $m$ times.

Again, the higher cluster category is Krull-Schmidt, $(m+1)$-Calabi-Yau, and triangulated.

\begin{defi}
Let $T$ be an object in the category ${\mathcal{C}}^m_Q$. Then $T$ is said to be $m$-rigid if, for any $i \in \{1,\cdots,m \}$, we have
\[ {\mathrm{Ext}}_\mathcal{C}^i(T,T)=0. \]
\end{defi}

\begin{defi}\cite{KR}
\cor{Let $T$ be an object in the category ${\mathcal{C}}^m_Q$. Then $T$ is $m$-cluster-tilting if, for any object $X$ of the category ${\mathcal{C}}^m_Q$, we have the following equivalence: \[ X \text{ is in } {\mathrm{add}}~T \iff ~\forall i \in \{ 1,\cdots,m \},~{\mathrm{Ext}}^i_{{\mathcal{C}_Q^{m}}}(T,X)=0, \] where $\mathrm{add}~T$ is the smallest additive subcategory of ${\mathcal{C}}^m_Q$ containing the object $T$.}
\end{defi}

\cor{Under the same notations, it is known from Zhu in \cite{Z}, that $T$ is an $m$-cluster-tilting object if and only if $T$ has $n$ indecomposable direct summands (up to isomorphism) and is $m$-rigid. So, let \[T=\bigoplus_{i=1}^n T_i\] be an $m$-cluster-tilting object, where each $T_i$ is indecomposable for any $i \in \{1,\cdots,n\}$. Let us define an almost complete $m$-rigid object. Let $k \in \{1,\cdots,n\}$. Recall that $T_k$ is an indecomposable summand of $T$. Then the object ${\overline{T}}=T/T_k$ is called an almost $m$-cluster-tilting object.}

It is shown by Wraalsen on the one hand in \cite{W} and by Zhou, Zhu on the other hand in \cite{ZZ}, that there exist, up to isomorphism, $m+1$ objects, denoted $T_k^{(c)}$ (for $c \in \{0,\cdots,m\}$), and called complements, such that \[\overline{T} \oplus T_k^{(c)}\] is an $m$-cluster-tilting object.

\begin{defi}
Let $\mathcal{T}$ be a triangulated category, with shift functor $[1]$, and $\mathcal{X}, \mathcal{Y}, \mathcal{D}$ be subcategories of $\mathcal{T}$. Let $\mu^{-1}(\mathcal{X},\mathcal{D})$ be the full subcategory made of objects $T \in \mathcal{T}$ such that there exists a triangle \[\xymatrix{X \ar^a[r] & D \ar[r] & T \ar[r] & X[1]},\] with $X \in \mathcal{X}$ and $a$ a left-$\mathcal{D}$-approximation. Dually, we introduce $\mu(\mathcal{Y},\mathcal{D})$ as the full subcategory made of objects $T \in \mathcal{T}$ such that there exists a triangle \[\xymatrix{T \ar[r] & D \ar^b[r] & Y \ar[r] & X[1]},\] with $Y \in \mathcal{Y}$ and $b$ a right-$\mathcal{D}$-approximation.

We say that $(\mathcal{X},\mathcal{Y})$ form a $\mathcal{D}$-mutation pair if we have the following inclusions:
\[\mathcal{D} \subseteq \mathcal{Y} \subseteq \mu^{-1}(\mathcal{X},\mathcal{D}) \text{ and } \mathcal{D} \subseteq \mathcal{X} \subseteq \mu(\mathcal{Y},\mathcal{D}).\] 
\end{defi}

About these complements, Iyama and Yoshino in \cite{IY} showed the following theorem:

\begin{theo}\label{th:iy}
Let $\mathcal{T}$ be a triangulated category with shift functor $[1]$, and let $\mathcal{Z}$ and $\mathcal{D}$ be two subcategories of $\mathcal{T}$ such that $\mathcal{D} \subseteq \mathcal{Z}$. Suppose that $\mathcal{Z}$ is extension closed, and $(\mathcal{Z},\mathcal{Z})$ form a $\mathcal{D}$-mutation pair (see \cite{IY} for a precise definition). Let $\mathcal{U}=\mathcal{Z}/\mathcal{D}$. Then there exists an equivalence $\langle 1 \rangle:\mathcal{U} \to \mathcal{U}$, and an object $D_x \in \mathcal{D}$ such that 
\[X \to D_X \to X\langle 1 \rangle \to X [1]\] is a triangle.
Then $\mathcal{U}$ forms a triangulated category, with shift functor $\langle 1 \rangle$.
\end{theo}

\begin{coro}
For any $k \leq n$, there are $m+1$ exchange triangles (for $c \in \{0,\cdots,m\}$):
\[
\xymatrix{
T_k^{(c)} \ar^{f_k^{(c)}}[r] & B_k^{(c)} \ar^{g_k^{(c+1)}}[r] & T_k^{(c+1)} \ar^{h_k^{(c+1)}}[r] & T_k^{(c)}[1] }
\]
Where the objects $B_k^{(c)}$ are in ${\mathrm{add}}~\overline{T}$, the morphisms $f_k^{(c)}$ (respectively $g_k^{(c+1)}$) are minimal left (respectively right) ${\mathrm{add}}~\overline{T}$-approximations, hence, not split monomorphisms nor split epimorphisms.

\cor{The new object $\overline{T} \oplus T_k^{(1)}$ is the mutation $T'=\mu_k(T)$ of $T$.}
\end{coro}

\subsection{Mutation of colored quivers}\label{sec:colq}
\cor{In this section, we define a colored quiver in the sense of Buan and Thomas. We also define the mutation of colored quivers and show how they help the understanding of Iyama-Yoshino mutation of $m$-cluster-tilting objects.}

To focus, we let $T$ be an $m$-cluster-tilting object in $\mathcal{C}_Q^m$, and we let $T'$ be an $m$-cluster-tilting object which is obtained by mutation of $T$ (in the sense of Iyama and Yoshino at Theorem \ref{th:iy}). Unfortunately, if $Q_T$ is the Gabriel quiver associated with $T$, there does not exist any quiver mutation $\mu$ such that $Q_{T'}=\mu(Q_T)$. Then, to remedy this lack, Buan and Thomas in \cite{BT} built a new quiver from $T$, which is called as the colored quiver associated with $T$.

\begin{defi}\cite{BT}
Given two positive integers $n$ and $m$, a colored quiver consists of the data of a quiver $Q=(Q_0,Q_1,s,t)$ with $n$ vertices, and of a function $\mathfrak{c}:Q_1 \to \{0,1,\cdots,m \}$ which associates with an arrow its color.
Let $q_{ij}^{(c)}$ be the number of arrows from $i$ to $j$ of color $c$. If there is an arrow from $i$ to $j$ of color $c$, then we write $i \xrightarrow{(c)} j$.
\end{defi}

\cor{For any $i$ and $j$ two vertices of the quiver $Q$, as Buan and Thomas in \cite{BT}, we only consider colored quivers that satisfy the following conditions:}

\begin{enumerate}
\item $q_{ii}^{(c)}=0$ for any $c \in \{0,\cdots,m\}$.
\item monochromaticity: if $q_{ij}^{(c)} \neq 0$ then $q_{ij}^{(c')}=0$ for all $c' \neq c \in \{0,\cdots,m\}$.
\item symmetry: $q_{ij}^{(c)}=q_{ji}^{(m-c)}$ for any $c \in \{0,\cdots,m\}$.
\end{enumerate}

\cor{Then, from now, each time we build a quiver, we ensure that it satisfies these three conditions.}

The operation we are about to define is an involution called the mutation of a colored quiver at a vertex.

\cor{From now and all throughout the paper, we consider colors modulo $m+1$. For instance, if we add one to the color $m$, then it becomes $0$.}

\begin{defi}\cite{BT}
Let $Q$ be a colored quiver, and let $k$ be a vertex of $Q$. We define the new quiver $\mu_k(Q)$ with the same vertices as the ones of $Q$, and the new number of arrows $\tilde{q}_{ij}^{(c)}$ given by:
\[ 
\tilde{q}_{ij}^{(c)} = \left\{
    \begin{array}{l}
        {q}_{ij}^{(c+1)} \mbox{ if } j = k \\
        {q}_{ij}^{(c-1)} \mbox{ if } i = k \\
        {\mathrm{max}} \{0,q_{ij}^{(c)}-\sum_{t \neq c}{q}_{ij}^{(t)} + ({q}_{ik}^{(c)}-{q}_{ik}^{(c-1)}){q}_{kj}^{(0)} + {q}_{ik}^{(m)}({q}_{kj}^{(c)}-{q}_{kj}^{(c+1)}) \} \mbox{ otherwise.}
    \end{array}
    \right.
\]
\end{defi}

The authors Buan and Thomas showed in \cite{BT} that mutating a colored quiver in this way is equivalent to the following procedure ($i,j,k$ being three vertices of $Q$):

\begin{enumerate}
\item For any $\xymatrix@1{i\ar[r]^{(c)} & k\ar[r]^{(0)} & j}$, if $i \neq j$ and $c$ is an integer in $\{0,\cdots,m\}$, then draw an arrow $\xymatrix@1{i\ar[r]^{(c)} & j}$ and an arrow $\xymatrix@1{j\ar[r]^{(m-c)} & i}$.
\item If condition 2 of monochromaticity in the restriction of colored quivers is not satisfied anymore from one vertex $i$ to one vertex $j$, then remove the same number of arrows of each color, in order to restore the condition.
\item For any arrow $\xymatrix@1{i \ar[r]^{(c)} & k}$, add $1$ to the color $c$, and for any arrow $\xymatrix@1{k \ar[r]^{(c)} & j}$, subtract $1$ to the color $c$.
\end{enumerate}

\cor{Let $T=\bigoplus_{i=1}^nT_i$ be an $m$-cluster-tilting object, and let $k \in \{1,\cdots,n\}$}. We recall that there are exchange triangles (for $c \in \{0,\cdots,m\}$):
\begin{equation}
\xymatrix{
T_k^{(c)} \ar^{f_k^{(c)}}[r] & B_k^{(c)} \ar^{g_k^{(c+1)}}[r] & T_k^{(c+1)} \ar^{h_k^{(c+1)}}[r] & T_k^{(c)}[1] }
\label{eq:1}
\end{equation}

With the $m$-cluster-tilting object $T$ in the $m$-cluster category, we associate a corresponding colored quiver $Q_T$ as follows:

\begin{enumerate}
\item The vertices of $Q_T$ are the integers from $1$ to $n$ where $n$ is the number of indecomposable summands of $T$.
\item The number $q_{ij}^{(c)}$ is the multiplicity of $T_j$ in $B_i^{(c)}$ in the exchange triangle (\ref{eq:1}).
\end{enumerate}

We now state the main theorem about colored quivers and $m$-cluster-tilting objects a proof of which can be found in \cite{BT}:

\begin{theo}\cite[Theorem 2.1]{BT}\label{th:mut}
\cor{Let $T$ be an $m$-cluster-tilting object, say \[T=\bigoplus_{i=1}^{n} T_i.\] Let $k \in \{1,\cdots,n\}$. Let \[T'=T/T_k \bigoplus T_k^{(1)}\] be the mutation of $T$ at $k$, where there is an exchange triangle \[T_k \to B_k^{(0)} \to T_k^{(1)} \to T_k[1]\] from Iyama and Yoshino. Then we have the following:
\[Q_{T'}=\mu_k(Q_T).\]}
\end{theo}

In particular, the colored quiver $Q_{T'}$ only depends on the colored quiver $Q_T$.

\subsection{The theorem of Keller and Reiten}
Here we only cite a beautiful theorem of Keller and Reiten in \cite{KR}, we will use all throughout the paper:

\begin{theo}\cite[Theorem 4.2]{KR}\label{th:kr}
Let $\mathcal{C}$ be a Hom-finite algebraic $(m+1)$-Calabi-Yau category. Let $T$ be an $m$-cluster-tilting object in $\mathcal{C}$, such that, for any $i \in \{ 1,\cdots,m \}$, we have \[\mathrm{Ext}_\mathcal{C}^{-i}(T,T)=0.\] Suppose that there exists a quiver $Q^{(T)}$ such that \cor{\[\mathrm{End}_{\mathcal{C}}(T) \simeq KQ^{(T)}.\]}

Then there exists an equivalence between the category $\mathcal{C}$ and the category $\mathcal{C}^{m}_{Q^{(T)}}$.
\end{theo}

\begin{rmk}
\cor{In our paper, if $Q$ is a quiver of a certain type ($A$, $D$, $\tilde{A}$, $\tilde{D}$) and $T$ an $m$-cluster-tilting object in the higher cluster category $\mathcal{C}^{m}_Q$, then we will apply this Theorem to the Iyama-Yoshino reduction of $\mathcal{C}^{m}_Q$ along $T$. In addition, the quiver $Q^{(T)}$ is no one but $Q_T$ defined just above. All these triangulated categories are algebraic (see the article of Buan, Iyama, Reiten and Scott \cite[Theorem I.1.8]{BIRS} for example).}

\cor{Moreover, in his paper \cite{Kel03}, Keller has shown that orbit categories are also algebraic triangulated categories. Then, all throughout the paper, we work with algebraic triangulated categories.}
\end{rmk}

\section{Geometric realizations}

\cor{This section is a survey of the geometric realizations of types $A$ (due to Baur and Marsh in \cite{BM01}), $D$ (due to Baur and Marsh \cite{BM02}), $\tilde{A}$ (due to Torkildsen in \cite{T}) and $\tilde{D}$ (due to the author in \cite{JM}). For each case, we define properly what are $m$-diagonals. After all this, we introduce some results common to these four cases. To be precise, we define $(m+2)$-angulations, the flip of an $(m+2)$-angulation. We also introduce the quiver associated with an $(m+2)$-angulation, and establish the compatibility between mutation of colored quivers, and flip of $(m+2)$-angulations. We finish with a bijection between $m$-diagonals and $m$-rigid objects in the higher cluster category.}

\cor{All throughout the section, and more generally the paper, we always consider paths up to homotopy.}

\subsection{Case $A$}\cite{BM01}

In this subsection, we recall the geometric realization of the $m$-cluster category of a quiver of type $A_n$, for a positive integer $n$, as did Baur and Marsh in their paper \cite{BM01}.

Let $Q$ be a quiver of type $A_n$, with $n$ vertices, and let $\mathcal{C}^{m}_Q$ be the $m$-cluster category associated with $Q$ (as defined in the preliminaries). Let $P$ be a polygon with $nm+2$ sides, with vertices numbered clockwise.

\begin{defi}
Let $i$ and $j$ be two different vertices of $P$. An $m$-diagonal $\alpha$ from the vertex $i$ to $j$ in $P$ is a path lying inside of $P$ linking $i$ and $j$ such that $\alpha$ cuts the figure into two polygons, one with $km+2$ sides, for some $k \in \mathbb{N}$ and one with $lm+2$ sides, for some $l \in \mathbb{N}$.
\end{defi}

In figure \ref{fig:adarc} we draw an example and a counter-example of an $m$-diagonal, for $m=2$, and $n=5$.

\begin{figure}[!h]
\centering
\begin{tikzpicture}[scale=0.7]
\fill[fill=black,fill opacity=0.1] (3.8,4.86) -- (2.62,4.2) -- (1.93,3.04) -- (1.91,1.69) -- (2.57,0.51) -- (3.73,-0.19) -- (5.08,-0.2) -- (6.26,0.46) -- (6.96,1.62) -- (6.97,2.97) -- (6.31,4.15) -- (5.15,4.84) -- cycle;
\draw (3.8,4.86)-- (2.62,4.2);
\draw (2.62,4.2)-- (1.93,3.04);
\draw (1.93,3.04)-- (1.91,1.69);
\draw (1.91,1.69)-- (2.57,0.51);
\draw (2.57,0.51)-- (3.73,-0.19);
\draw (3.73,-0.19)-- (5.08,-0.2);
\draw (5.08,-0.2)-- (6.26,0.46);
\draw (6.26,0.46)-- (6.96,1.62);
\draw (6.96,1.62)-- (6.97,2.97);
\draw (6.97,2.97)-- (6.31,4.15);
\draw (6.31,4.15)-- (5.15,4.84);
\draw (5.15,4.84)-- (3.8,4.86);
\draw (1.93,3.04)-- (3.73,-0.19);
\end{tikzpicture}
\begin{tikzpicture}[scale=0.7]
\fill[fill=black,fill opacity=0.1] (3.8,4.86) -- (2.62,4.2) -- (1.93,3.04) -- (1.91,1.69) -- (2.57,0.51) -- (3.73,-0.19) -- (5.08,-0.2) -- (6.26,0.46) -- (6.96,1.62) -- (6.97,2.97) -- (6.31,4.15) -- (5.15,4.84) -- cycle;
\draw (3.8,4.86)-- (2.62,4.2);
\draw (2.62,4.2)-- (1.93,3.04);
\draw (1.93,3.04)-- (1.91,1.69);
\draw (1.91,1.69)-- (2.57,0.51);
\draw (2.57,0.51)-- (3.73,-0.19);
\draw (3.73,-0.19)-- (5.08,-0.2);
\draw (5.08,-0.2)-- (6.26,0.46);
\draw (6.26,0.46)-- (6.96,1.62);
\draw (6.96,1.62)-- (6.97,2.97);
\draw (6.97,2.97)-- (6.31,4.15);
\draw (6.31,4.15)-- (5.15,4.84);
\draw (5.15,4.84)-- (3.8,4.86);
\draw (3.8,4.86)-- (5.08,-0.2);
\end{tikzpicture}
\caption{The first is a $2$-diagonal, the second is not}
\label{fig:adarc}
\end{figure}

\subsection{Case $D$}\cite{BM02}

This case was treated by Baur and Marsh in \cite{BM02}. However, we have to use a slightly different geometric realization, in order to make the notion of flip of an $(m+2)$-angulation compatible with colored quiver mutation (shown at the paragraph of common results). Baur and Marsh use a polygon with $nm-m+1$ sides with a puncture inside of it. We replace the puncture by an $(m-1)$-gon with, on each vertex of it, a disk. Considering arcs up to Dehn twist, as we will see just below, we have an obvious $1-1$ correspondence between the $m$-diagonals of Baur and Marsh, and the ones we are about to define.

\cor{Let $n \geq 4$ and $m$ be a positive integer. Suppose that $m>1$. Let $P$ be an $nm-m+1$-gon with $R$, an $(m-1)$-gon inside of it. We replace each vertex of $R$ by a disk, which we henceforth call a thick vertex. If $m=1$, the surface $P$ considered is a disk with $n$ marked points on its border, with no inner boundary components and with one puncture in its interior.}

\begin{defi}
Consider one moment a path starting at an arbitrary vertex of $P$, and ending at a thick vertex of $R$. This path is called to be left tangent at $R$ if it is $\mathcal{C}^\infty$, tangent to a thick vertex of $R$, and if there exists a neighborhood of this thick vertex such that the path is situated at the right of the vertex. We similarly define right tangency (see figure ...)
\end{defi}

\begin{figure}[h!]
\centering
\begin{tikzpicture}[scale=0.4]

	\draw [line width=2pt,fill=black,fill opacity=1] (-1.74,4.99) circle (0.9879524280045071cm);

	\fill[line width=2pt,fill=black,fill opacity=0.1] (-4.22,-1.35) -- (0.735,-1.355) -- (4.242249634685275,2.14517856687341) -- (4.247249634685276,7.100178566873409) -- (0.7470710678118662,10.607428201558685) -- (-4.207928932188134,10.612428201558686) -- (-7.7151785668734085,7.112249634685276) -- (-7.72017856687341,2.1572496346852774) -- cycle;

	\draw [line width=2pt] (-4.22,-1.35)-- (0.735,-1.355);

	\draw [line width=2pt] (0.735,-1.355)-- (4.242249634685275,2.14517856687341);

	\draw [line width=2pt] (4.242249634685275,2.14517856687341)-- (4.247249634685276,7.100178566873409);

	\draw [line width=2pt] (4.247249634685276,7.100178566873409)-- (0.7470710678118662,10.607428201558685);

	\draw [line width=2pt] (0.7470710678118662,10.607428201558685)-- (-4.207928932188134,10.612428201558686);

	\draw [line width=2pt] (-4.207928932188134,10.612428201558686)-- (-7.7151785668734085,7.112249634685276);

	\draw [line width=2pt] (-7.7151785668734085,7.112249634685276)-- (-7.72017856687341,2.1572496346852774);

	\draw [line width=2pt] (-7.72017856687341,2.1572496346852774)-- (-4.22,-1.35);

	\draw [shift={(-6.1072365457236115,0.6598225179735304)},line width=2pt]  plot[domain=-0.2863742472137245:0.7657250095781128,variable=\t]({1*7.132721130576252*cos(\t r)+0*7.132721130576252*sin(\t r)},{0*7.132721130576252*cos(\t r)+1*7.132721130576252*sin(\t r)});

	\begin{scriptsize}

		\draw [fill=black] (-4.22,-1.35) circle (2.5pt);

		\draw [fill=black] (0.735,-1.355) circle (2.5pt);

		\draw [fill=black] (4.242249634685275,2.14517856687341) circle (2.5pt);

		\draw [fill=black] (4.247249634685276,7.100178566873409) circle (2.5pt);

		\draw [fill=black] (0.7470710678118662,10.607428201558685) circle (2.5pt);

		\draw [fill=black] (-4.207928932188134,10.612428201558686) circle (2.5pt);

		\draw [fill=black] (-7.7151785668734085,7.112249634685276) circle (2.5pt);

		\draw [fill=black] (-7.72017856687341,2.1572496346852774) circle (2.5pt);

	\end{scriptsize}

\end{tikzpicture}
\caption{A path which is right tangent, with $m=1$ (indeed, we can see that there is only one thick vertex)}
\label{fig:tangent}
\end{figure}

\begin{defi}
Let us number the vertices of the polygon $P$ from $1$ to $nm-m+1$ clockwise. Then, an admissible arc between $i$ and $j$ is defined in the following way:
\begin{enumerate}
\item If $i \neq j$, then an admissible arc is an oriented path from $P_i$ to $P_j$, homotopic to the boundary path, lying inside of $P$, and which does not cross $R$.
\item If $i=j$, then an admissible arc is a path ending in $i$, and the other end of the path is tangent to one of the thick vertices of $R$. There are two more admissible arcs starting and ending at $i$, going around $R$: the left loop and the right loop. They look the same on the picture, but they are labelled differently, we will see later why.
\end{enumerate}
\end{defi}

\vspace{10pt}

Note that we only consider unoriented arcs, the order of $i$ and $j$ does not matter. For convenience, we will nevertheless use the terminology "from $i$ to $j$".

\begin{defi}
\cor{We call a Dehn twist the action of rotating $R$ by stretching the arc. It means that if we consider an arc $\alpha$ ending at a thick vertex of $R$, applying a Dehn twist of $R$ makes $\alpha$ roll around $R$. We can define Dehn twist in both clockwise and counterclockwise directions.}
\end{defi}

We now define $m$-diagonals.

\begin{defi}\label{def:diagoD}
\cor{An $m$-diagonal is an equivalence class of admissible arcs under the Dehn twist equivalence relation.}
\end{defi}

\cor{Whenever it is clear, we indifferently identify  the term of equivalence class and well-chosen representative.}

\subsection{Case $\tilde{A}$}\cite{Tor}

The geometric description of case $\tilde{A}$ has been completely treated by Torkildsen in \cite{Tor}. We recall part of his description.

Let $m$ be a positive integer. Let $Q$ be a quiver of type $\tilde{A_n}$, with $p$ arrows going one direction, and $q$ arrows going the other. Let $O$ be a regular $mp$-gon, with $I$, a regular $mq$-gon inside of it. In the following, we give the example with $p=4$ and $q=3$, for $m=2$. We number the vertices of the outer polygon $O_1,\cdots,O_{mp-1}$ and the vertices of the inner polygon $I_1,\cdots,I_{mq-1}$, counted in opposite directions.

\begin{defi}

We define an $m$-diagonal to be a path satisfying one of the three cases below:

\begin{itemize}
\item A path from $O_i$ to $I_j$ where $i$ and $j$ are congruent modulo $m$
\item A path from $O_i$ to $O_{i+km+1}$, where $k \geq 1$, is counted modulo $pm+1$ and $i \in \{1,\cdots,pm+1\}$ homotopic to the boundary path of the outer polygon $P$.
\item A path from $I_i$ to $I_{i+km+1}$ for some $i$ and some $k \geq 1$ homotopic to the boundary path of the inner polygon $I$.
\end{itemize}
\end{defi}

\begin{figure}[!h]
\centering
\begin{tikzpicture}[scale=0.7]
\fill[fill=black,fill opacity=0.1] (3.36,2.26) -- (1.04,2.26) -- (-0.6,0.62) -- (-0.6,-1.7) -- (1.04,-3.34) -- (3.36,-3.34) -- (5,-1.7) -- (5,0.62) -- cycle;
\fill[fill=black] (2.72,0.42) -- (1.6,0.42) -- (1.04,-0.55) -- (1.6,-1.52) -- (2.72,-1.52) -- (3.28,-0.55) -- cycle;
\draw (3.36,2.26)-- (1.04,2.26);
\draw (1.04,2.26)-- (-0.6,0.62);
\draw (-0.6,0.62)-- (-0.6,-1.7);
\draw (-0.6,-1.7)-- (1.04,-3.34);
\draw (1.04,-3.34)-- (3.36,-3.34);
\draw (3.36,-3.34)-- (5,-1.7);
\draw (5,-1.7)-- (5,0.62);
\draw (5,0.62)-- (3.36,2.26);
\draw (2.72,0.42)-- (1.6,0.42);
\draw (1.6,0.42)-- (1.04,-0.55);
\draw (1.04,-0.55)-- (1.6,-1.52);
\draw (1.6,-1.52)-- (2.72,-1.52);
\draw (2.72,-1.52)-- (3.28,-0.55);
\draw (3.28,-0.55)-- (2.72,0.42);
\draw [shift={(2.72,1.77)}] plot[domain=2.85:4.02,variable=\t]({1*1.75*cos(\t r)+0*1.75*sin(\t r)},{0*1.75*cos(\t r)+1*1.75*sin(\t r)});
\end{tikzpicture}
\hspace{10pt}
\begin{tikzpicture}[scale=0.7]
\fill[fill=black,fill opacity=0.1] (3.36,2.26) -- (1.04,2.26) -- (-0.6,0.62) -- (-0.6,-1.7) -- (1.04,-3.34) -- (3.36,-3.34) -- (5,-1.7) -- (5,0.62) -- cycle;
\fill[fill=black] (2.72,0.42) -- (1.6,0.42) -- (1.04,-0.55) -- (1.6,-1.52) -- (2.72,-1.52) -- (3.28,-0.55) -- cycle;
\draw (3.36,2.26)-- (1.04,2.26);
\draw (1.04,2.26)-- (-0.6,0.62);
\draw (-0.6,0.62)-- (-0.6,-1.7);
\draw (-0.6,-1.7)-- (1.04,-3.34);
\draw (1.04,-3.34)-- (3.36,-3.34);
\draw (3.36,-3.34)-- (5,-1.7);
\draw (5,-1.7)-- (5,0.62);
\draw (5,0.62)-- (3.36,2.26);
\draw (2.72,0.42)-- (1.6,0.42);
\draw (1.6,0.42)-- (1.04,-0.55);
\draw (1.04,-0.55)-- (1.6,-1.52);
\draw (1.6,-1.52)-- (2.72,-1.52);
\draw (2.72,-1.52)-- (3.28,-0.55);
\draw (3.28,-0.55)-- (2.72,0.42);
\draw [shift={(0.62,-2.12)}] plot[domain=0.1:1.47,variable=\t]({1*4.4*cos(\t r)+0*4.4*sin(\t r)},{0*4.4*cos(\t r)+1*4.4*sin(\t r)});
\end{tikzpicture}
\caption{Examples of $m$-diagonals for case $p=4$, $q=3$, $m=2$}
\end{figure}

\subsection{Case $\tilde{D}$}\cite{JM}

Let $n \geq 4$ and $m$ be a positive integer. Let $P$ be an $(n-2)m$-gon with two central $(m-1)$-gons $R$ and $S$ inside of it (cf figure \ref{fig:figure1}). We replace each vertex of $R$ and $S$ by a disk, which we henceforth call a thick vertex. \cor{If $m=1$, then the surface considered is a disk with $n-2$ marked points on its border, with no inner boundary components and with two thick vertices in its interior.}

\cor{We could have used tagged arcs, and a polygon with two punctures, as Baur and Marsh did for type $D$. However, as the adaptation we did for Dynkin type $D$, this sort of realization will imply that the flip of an $(m+2)$-angulation is not compatible with the mutation of the colored quiver associated with the $(m+2)$-angulation.}

\begin{figure}[!h]
\centering
\begin{tikzpicture}[scale=0.25]
\fill[fill=black,fill opacity=0.15] (0,4) -- (0,-4) -- (6.93,-8) -- (13.86,-4) -- (13.86,4) -- (6.93,8) -- cycle;
\draw [fill=black,fill opacity=1.0] (3.46,2) circle (0.5cm);
\draw [fill=black,fill opacity=1.0] (3.46,-2) circle (0.5cm);
\draw [fill=black,fill opacity=1.0] (10.39,2) circle (0.5cm);`
\draw [fill=black,fill opacity=1.0] (10.39,-2) circle (0.5cm);
\draw (0,4)-- (0,-4);
\draw (0,-4)-- (6.93,-8);
\draw (6.93,-8)-- (13.86,-4);
\draw (13.86,-4)-- (13.86,4);
\draw (13.86,4)-- (6.93,8);
\draw (6.93,8)-- (0,4);
\draw [shift={(1.86,0)}] plot[domain=-0.89:0.89,variable=\t]({1*2.57*cos(\t r)+0*2.57*sin(\t r)},{0*2.57*cos(\t r)+1*2.57*sin(\t r)});
\draw [shift={(5.07,0)}] plot[domain=2.25:4.04,variable=\t]({1*2.57*cos(\t r)+0*2.57*sin(\t r)},{0*2.57*cos(\t r)+1*2.57*sin(\t r)});
\draw [shift={(12,0)}] plot[domain=2.25:4.04,variable=\t]({1*2.57*cos(\t r)+0*2.57*sin(\t r)},{0*2.57*cos(\t r)+1*2.57*sin(\t r)});
\draw [shift={(8.79,0)}] plot[domain=-0.89:0.89,variable=\t]({1*2.57*cos(\t r)+0*2.57*sin(\t r)},{0*2.57*cos(\t r)+1*2.57*sin(\t r)});
\end{tikzpicture}
\caption{The $(n-2)m$-gon with two digons. Here $m=3$ and $n=4$.}
\label{fig:figure1}
\end{figure}

\begin{defi}\label{def:arcdtilde}
Let us number the vertices of the polygon $P$ from $1$ to $(n-2)m$ clockwise. Then, an admissible arc between $i$ and $j$ is defined in the following way:
\begin{enumerate}
\item If $i \neq j$, then an admissible arc is an oriented path from $i$ to $j$, lying inside of $P$, which does not cross any of the two inner polygons, satisfying one of the following conditions:
\begin{itemize}
\item Either the arc crosses the space between both central polygons and cuts the figure into a $(km+1)$-gon and a $(k'm+1)$-gon, for some $k$ ($k'$ is entirely determined by $k$). This arc is of type $1$.
\item Either, the arc is homotopic to a boundary path, and cuts the figure into a $km$-gon with both central polygons inside of it and a $k'm+2$-gon (where $k'$ is still entirely determined by $k$). This arc is of type $2$.
\end{itemize}
\item If $i=j$, there are four types of admissible arcs :
\begin{itemize}
	\item A path ending in $i$, and the other end of the path tangent to one of the thick vertices placed around $R$.
	\item A path ending in $i$, and the other end of the path tangent to one of the thick vertices placed around $S$.
	\item A path starting and ending at $i$, going around $R$, called a loop.
	\item A path starting and ending at $i$, going around $R$, also called a loop.
\end{itemize}
\cor{All these arcs are of type 3.}
\item Any arc being tangent to two disks, one arising from $R$, and one from $S$ is admissible. \cor{This arc is of type 4.}
\end{enumerate}
\end{defi}

\cor{Note that any admissible arc can have some self-crossings. However, the $m$-diagonals we are about to define do not have any of them. We will see at the end of the section that the $m$-diagonals are in $1-1$ correspondence with the $m$-rigid indecomposable objects of the associated higher cluster category. Nevertheless, we can figure in \cite{JM} that there are some admissible arcs with self-crossings that can be identified to objects which are not $m$-rigid (in the tubes of the Auslander-Reiten quiver of $\mathcal{C}_Q^m$ for instance).}

\begin{defi}
Consider one moment an arc starting at an arbitrary vertex of $P$, and ending at a thick vertex $R$ or $S$. This arc is called to be left tangent at $R$ if it is $\mathcal{C}^\infty$, tangent to a thick vertex, and if there exists a neighborhood of this thick vertex such that this one is situated at the right of the arc. We similarly define right tangency.
\end{defi}

\vspace{10pt}

Note that we only consider unoriented arcs, the order of $i$ and $j$ does not matter. For convenience, we will nevertheless use the terminology "from $i$ to $j$".

\begin{defi}
\cor{We call a Dehn twist the action of rotating $R$ (respectively $S$). It means that if we consider an arc $\alpha$ hung to $R$ (respectively $S$), applying a Dehn twist of $R$ (respectively $S$) makes $\alpha$ roll around $R$ (respectively $S$) only (see figure \ref{fig:roll}). We can define Dehn twist in both clockwise and counterclockwise directions.}
\end{defi}

\begin{figure}[!h]
	\centering
	\begin{tikzpicture}[scale=0.25]
	\fill[fill=black,fill opacity=0.15] (0,4) -- (0,-4) -- (6.93,-8) -- (13.86,-4) -- (13.86,4) -- (6.93,8) -- cycle;
	\draw [fill=black,fill opacity=1.0] (3.46,2) circle (0.4cm);
	\draw [fill=black,fill opacity=1.0] (3.46,-2) circle (0.4cm);
	\draw [fill=black,fill opacity=1.0] (10.39,2) circle (0.4cm);
	\draw [fill=black,fill opacity=1.0] (10.39,-2) circle (0.4cm);
	\draw (0,4)-- (0,-4);
	\draw (0,-4)-- (6.93,-8);
	\draw (6.93,-8)-- (13.86,-4);
	\draw (13.86,-4)-- (13.86,4);
	\draw (13.86,4)-- (6.93,8);
	\draw (6.93,8)-- (0,4);
	\draw [shift={(1.86,0)}] plot[domain=-0.89:0.89,variable=\t]({1*2.57*cos(\t r)+0*2.57*sin(\t r)},{0*2.57*cos(\t r)+1*2.57*sin(\t r)});
	\draw [shift={(5.07,0)}] plot[domain=2.25:4.04,variable=\t]({1*2.57*cos(\t r)+0*2.57*sin(\t r)},{0*2.57*cos(\t r)+1*2.57*sin(\t r)});
	\draw [shift={(12,0)}] plot[domain=2.25:4.04,variable=\t]({1*2.57*cos(\t r)+0*2.57*sin(\t r)},{0*2.57*cos(\t r)+1*2.57*sin(\t r)});
	\draw [shift={(8.79,0)}] plot[domain=-0.89:0.89,variable=\t]({1*2.57*cos(\t r)+0*2.57*sin(\t r)},{0*2.57*cos(\t r)+1*2.57*sin(\t r)});
	\draw [shift={(1.59,-3.73)}] plot[domain=-1.26:0.41,variable=\t]({0.94*11.55*cos(\t r)+-0.35*6.17*sin(\t r)},{0.35*11.55*cos(\t r)+0.94*6.17*sin(\t r)});
	\end{tikzpicture}
	\hspace{20pt}
	\begin{tikzpicture}[scale=0.25]
	\fill[fill=black,fill opacity=0.15] (0,4) -- (0,-4) -- (6.93,-8) -- (13.86,-4) -- (13.86,4) -- (6.93,8) -- cycle;
	\draw [fill=black,fill opacity=1.0] (3.46,2) circle (0.4cm);
	\draw [fill=black,fill opacity=1.0] (3.46,-2) circle (0.4cm);
	\draw [fill=black,fill opacity=1.0] (10.39,2) circle (0.4cm);
	\draw [fill=black,fill opacity=1.0] (10.39,-2) circle (0.4cm);
	\draw (0,4)-- (0,-4);
	\draw (0,-4)-- (6.93,-8);
	\draw (6.93,-8)-- (13.86,-4);
	\draw (13.86,-4)-- (13.86,4);
	\draw (13.86,4)-- (6.93,8);
	\draw (6.93,8)-- (0,4);
	\draw [shift={(1.86,0)}] plot[domain=-0.89:0.89,variable=\t]({1*2.57*cos(\t r)+0*2.57*sin(\t r)},{0*2.57*cos(\t r)+1*2.57*sin(\t r)});
	\draw [shift={(5.07,0)}] plot[domain=2.25:4.04,variable=\t]({1*2.57*cos(\t r)+0*2.57*sin(\t r)},{0*2.57*cos(\t r)+1*2.57*sin(\t r)});
	\draw [shift={(12,0)}] plot[domain=2.25:4.04,variable=\t]({1*2.57*cos(\t r)+0*2.57*sin(\t r)},{0*2.57*cos(\t r)+1*2.57*sin(\t r)});
	\draw [shift={(8.79,0)}] plot[domain=-0.89:0.89,variable=\t]({1*2.57*cos(\t r)+0*2.57*sin(\t r)},{0*2.57*cos(\t r)+1*2.57*sin(\t r)});
	\draw [shift={(8.18,-2.24)}] plot[domain=0.72:4.18,variable=\t]({-0.53*7.35*cos(\t r)+0.85*3.06*sin(\t r)},{-0.85*7.35*cos(\t r)+-0.53*3.06*sin(\t r)});
	\draw [shift={(9.47,0.18)}] plot[domain=-1.84:1.09,variable=\t]({-0.57*3.36*cos(\t r)+0.82*2.57*sin(\t r)},{-0.82*3.36*cos(\t r)+-0.57*2.57*sin(\t r)});
	\end{tikzpicture}
	\caption{These two arcs represent the same $m$-diagonal.}
	\label{fig:roll}
\end{figure}

We now define $m$-diagonals.

\begin{defi}
Let $\alpha$ and $\beta$ be two admissible arcs of the same type. We say that these arcs are equivalent when:
\begin{itemize}
\item \cor{If $\alpha$ and $\beta$ are of type $1$, $2$ or $4$, then they are said to be equivalent if they are homotopic.
\item If $\alpha$ and $\beta$ are of type $3$ and hang to the same vertex and to the same inner polygon (say $R$ for instance, but the case of $S$ is similar), then they are said to be equivalent if they are homotopic, or, if there exists a Dehn twist $t$ such that $\alpha=t(\beta)$. In this case, we add in the class of equivalence a loop, which we call a left loop (respectively right loop) if the arcs in the equivalence class are left tangent (respectively right tangent), drawn in figure \ref{fig:loop}, around $R$, ending at the same vertex as $\alpha$.}
\end{itemize}
\end{defi}

\begin{figure}[!h]
\centering
\begin{tikzpicture}[scale=0.5]
\fill[fill=black,fill opacity=0.1] (0,4) -- (-1.66,3.46) -- (-2.69,2.05) -- (-2.68,0.3) -- (-1.66,-1.11) -- (0,-1.65) -- (1.66,-1.11) -- (2.69,0.3) -- (2.69,2.05) -- (1.66,3.46) -- cycle;
\draw [fill=black,fill opacity=1.0] (-1.66,1.57) circle (0.2cm);
\draw [fill=black,fill opacity=1.0] (1.66,1.58) circle (0.2cm);
\draw (0,4)-- (-1.66,3.46);
\draw (-1.66,3.46)-- (-2.69,2.05);
\draw (-2.69,2.05)-- (-2.68,0.3);
\draw (-2.68,0.3)-- (-1.66,-1.11);
\draw (-1.66,-1.11)-- (0,-1.65);
\draw (0,-1.65)-- (1.66,-1.11);
\draw (1.66,-1.11)-- (2.69,0.3);
\draw (2.69,0.3)-- (2.69,2.05);
\draw (2.69,2.05)-- (1.66,3.46);
\draw (1.66,3.46)-- (0,4);
\draw(-1.66,1.17) circle (0.4cm);
\draw(1.66,1.18) circle (0.4cm);
\draw [shift={(0.3,0.15)}] plot[domain=0.96:3.44,variable=\t]({-0.77*2.41*cos(\t r)+0.64*1.46*sin(\t r)},{-0.64*2.41*cos(\t r)+-0.77*1.46*sin(\t r)});
\draw [shift={(0.99,-1.47)}] plot[domain=3.09:4.82,variable=\t]({-0.21*3.53*cos(\t r)+0.98*0.93*sin(\t r)},{-0.98*3.53*cos(\t r)+-0.21*0.93*sin(\t r)});
\end{tikzpicture}
\caption{Loop drawn at a vertex, around $R$. The same can be done for $S$}
\label{fig:loop}
\end{figure}

\begin{defi}\label{def:diago}
\cor{An $m$-diagonal is an equivalence class of admissible arcs.}
\end{defi}

\cor{Whenever it is clear, we indifferently identify  the term of equivalence class and well-chosen representative.}

\subsection{Common results on all types}

\cor{Now that we have introduced the four types of geometric realizations, we can set some results which apply to $A$, $D$, $\tilde{A}$, $\tilde{D}$. Most of the results are already shown in other papers (in this case, we tried to give precise references).}

\begin{defi}
Two arcs are said to be noncrossing if their class under homotopy contain representatives which do not cross.

An $(m+2)$-angulation of $P$ is a maximal set of noncrossing $m$-diagonals.
\end{defi}

\begin{defi}
We call by $t$-angle, a figure with $t$ sides, made of:
\begin{itemize}
\item Sides of $P$ (for all cases) and $T$ (for case $\tilde{A}$)
\item Sides of $R$ (for cases $D$ and $\tilde{D}$) or $S$ (for case $\tilde{D}$)
\item $m$-diagonals (for all cases).
\end{itemize}
\end{defi}

\begin{rmk}
We note that the definition of an $(m+2)$-angulation is equivalent to the following one : an $(m+2)$-angulation is a set of $m$-diagonals cutting the polygon into $(m+2)$-angles.
\end{rmk}

\begin{figure}[!h]
\centering
\begin{tikzpicture}[scale=0.9]
\fill[fill=black,fill opacity=0.1] (3.8,4.86) -- (2.62,4.2) -- (1.93,3.04) -- (1.91,1.69) -- (2.57,0.51) -- (3.73,-0.19) -- (5.08,-0.2) -- (6.26,0.46) -- (6.96,1.62) -- (6.97,2.97) -- (6.31,4.15) -- (5.15,4.84) -- cycle;
\draw (3.8,4.86)-- (2.62,4.2);
\draw (2.62,4.2)-- (1.93,3.04);
\draw (1.93,3.04)-- (1.91,1.69);
\draw (1.91,1.69)-- (2.57,0.51);
\draw (2.57,0.51)-- (3.73,-0.19);
\draw (3.73,-0.19)-- (5.08,-0.2);
\draw (5.08,-0.2)-- (6.26,0.46);
\draw (6.26,0.46)-- (6.96,1.62);
\draw (6.96,1.62)-- (6.97,2.97);
\draw (6.97,2.97)-- (6.31,4.15);
\draw (6.31,4.15)-- (5.15,4.84);
\draw (5.15,4.84)-- (3.8,4.86);
\draw (1.93,3.04)-- (3.73,-0.19);
\draw (2.62,4.2)-- (5.08,-0.2);
\draw (5.08,-0.2)-- (5.15,4.84);
\draw (5.08,-0.2)-- (6.97,2.97);
\end{tikzpicture}
\hspace{20pt}
\begin{tikzpicture}[scale=0.85]
\fill[fill=black,fill opacity=0.1] (0,4) -- (-1.66,3.46) -- (-2.69,2.05) -- (-2.68,0.3) -- (-1.66,-1.11) -- (0,-1.65) -- (1.66,-1.11) -- (2.69,0.3) -- (2.69,2.05) -- (1.66,3.46) -- cycle;
\draw [fill=black,fill opacity=1.0] (-1.66,1.57) circle (0.2cm);
\draw [fill=black,fill opacity=1.0] (1.66,1.58) circle (0.2cm);
\draw (0,4)-- (-1.66,3.46);
\draw (-1.66,3.46)-- (-2.69,2.05);
\draw (-2.69,2.05)-- (-2.68,0.3);
\draw (-2.68,0.3)-- (-1.66,-1.11);
\draw (-1.66,-1.11)-- (0,-1.65);
\draw (0,-1.65)-- (1.66,-1.11);
\draw (1.66,-1.11)-- (2.69,0.3);
\draw (2.69,0.3)-- (2.69,2.05);
\draw (2.69,2.05)-- (1.66,3.46);
\draw (1.66,3.46)-- (0,4);
\draw(-1.66,1.17) circle (0.4cm);
\draw(1.66,1.18) circle (0.4cm);
\draw [shift={(1.95,-3.35)}] plot[domain=-0.11:1.3,variable=\t]({-0.07*6.83*cos(\t r)+-1*1.89*sin(\t r)},{1*6.83*cos(\t r)+-0.07*1.89*sin(\t r)});
\draw [shift={(1.44,-3.05)}] plot[domain=2.2:4.47,variable=\t]({-0.01*5.75*cos(\t r)+1*1.5*sin(\t r)},{-1*5.75*cos(\t r)+-0.01*1.5*sin(\t r)});
\draw [shift={(-0.65,-0.55)}] plot[domain=1.52:3.46,variable=\t]({-0.79*3.42*cos(\t r)+0.62*1.27*sin(\t r)},{-0.62*3.42*cos(\t r)+-0.79*1.27*sin(\t r)});
\draw [shift={(0.3,0.15)}] plot[domain=0.96:3.44,variable=\t]({-0.77*2.41*cos(\t r)+0.64*1.46*sin(\t r)},{-0.64*2.41*cos(\t r)+-0.77*1.46*sin(\t r)});
\draw [shift={(0.99,-1.47)}] plot[domain=3.09:4.82,variable=\t]({-0.21*3.53*cos(\t r)+0.98*0.93*sin(\t r)},{-0.98*3.53*cos(\t r)+-0.21*0.93*sin(\t r)});
\draw [shift={(-1.94,-3.36)}] plot[domain=-0.11:1.3,variable=\t]({0.07*6.83*cos(\t r)+1*1.89*sin(\t r)},{1*6.83*cos(\t r)+-0.07*1.89*sin(\t r)});
\draw [shift={(-1.43,-3.05)}] plot[domain=2.2:4.47,variable=\t]({0.01*5.75*cos(\t r)+-1*1.5*sin(\t r)},{-1*5.75*cos(\t r)+-0.01*1.5*sin(\t r)});
\draw [shift={(-0.98,-1.47)}] plot[domain=3.09:4.82,variable=\t]({0.22*3.53*cos(\t r)+-0.98*0.93*sin(\t r)},{-0.98*3.53*cos(\t r)+-0.22*0.93*sin(\t r)});
\draw [shift={(-0.29,0.15)}] plot[domain=0.96:3.44,variable=\t]({0.77*2.41*cos(\t r)+-0.64*1.46*sin(\t r)},{-0.64*2.41*cos(\t r)+-0.77*1.46*sin(\t r)});
\draw [shift={(-7.41,-2.97)}] plot[domain=0.18:0.67,variable=\t]({1*7.53*cos(\t r)+0*7.53*sin(\t r)},{0*7.53*cos(\t r)+1*7.53*sin(\t r)});
\end{tikzpicture}
\caption{This is an example of a $4$-angulation for $A$ and $\tilde{D}$}
\label{fig:mang}
\end{figure}

We can define the twist of an $m$-diagonal and the flip of an $(m+2)$-angulation, as Buan and Thomas did in \cite{BT} for case $A$.

\begin{defi}
Let $\Delta$ be an $(m+2)$-angulation. Let $\alpha$ be an $m$-diagonal of $\Delta$. Let $a$ and $b$, be the ends of $\alpha$ ($a$ and $b$ can be vertices of $P$, or $T$, or also thick vertices of $R$ or $S$). The twist of $\alpha$ in $\Delta$ is defined as follows:

Let $(a',a)$ (respectively $(b,b')$) be the side of the $(m+2)$-angle ending at $a$ (respectively at $b'$) preceding $a$ clockwise (respectively preceding $b$). Then the twist of $\alpha$, namely $\kappa_{\Delta}(\alpha)$ is the $m$-diagonal $(a',b')$.

\cor{To be said in another way, the twist consists in rotating $\alpha$ clockwise along the sides of the polygon.}
\end{defi}

\begin{defi}
Consider $\Delta$ an $(m+2)$-angulation. Let $\alpha$ be an arc in $\Delta$. The flip of the $(m+2)$-angulation at $\alpha$ is defined by $\mu_\alpha\Delta=\Delta \setminus \{ \alpha \} \cup \{ \kappa_{\Delta}(\alpha) \}$.
\end{defi}

\begin{rmk}
\begin{enumerate}
\item Note that the twist has an inverse, which consists in moving the $m$-diagonal counterclockwise. Then the flip is also invertible.
\item A flip does not change the number of $m$-diagonals in the $(m+2)$-angulation.
\end{enumerate}
\end{rmk}

In figure \ref{fig:flip}, we can see two examples of flips.

\begin{figure}[!h]
\centering
\begin{tikzpicture}[scale=0.7]
\fill[fill=black,fill opacity=0.1] (3.8,4.86) -- (2.62,4.2) -- (1.93,3.04) -- (1.91,1.69) -- (2.57,0.51) -- (3.73,-0.19) -- (5.08,-0.2) -- (6.26,0.46) -- (6.96,1.62) -- (6.97,2.97) -- (6.31,4.15) -- (5.15,4.84) -- cycle;
\draw (3.8,4.86)-- (2.62,4.2);
\draw (2.62,4.2)-- (1.93,3.04);
\draw (1.93,3.04)-- (1.91,1.69);
\draw (1.91,1.69)-- (2.57,0.51);
\draw (2.57,0.51)-- (3.73,-0.19);
\draw (3.73,-0.19)-- (5.08,-0.2);
\draw (5.08,-0.2)-- (6.26,0.46);
\draw (6.26,0.46)-- (6.96,1.62);
\draw (6.96,1.62)-- (6.97,2.97);
\draw (6.97,2.97)-- (6.31,4.15);
\draw (6.31,4.15)-- (5.15,4.84);
\draw (5.15,4.84)-- (3.8,4.86);
\draw (5.08,-0.2)-- (5.15,4.84);
\draw (5.08,-0.2)-- (6.97,2.97);
\draw (5.08,-0.2)-- (1.91,1.69);
\draw [color=red] (1.93,3.04)-- (5.15,4.84);
\end{tikzpicture}
$\rightarrow$
\begin{tikzpicture}[scale=0.7]
\fill[fill=black,fill opacity=0.1] (3.8,4.86) -- (2.62,4.2) -- (1.93,3.04) -- (1.91,1.69) -- (2.57,0.51) -- (3.73,-0.19) -- (5.08,-0.2) -- (6.26,0.46) -- (6.96,1.62) -- (6.97,2.97) -- (6.31,4.15) -- (5.15,4.84) -- cycle;
\draw (3.8,4.86)-- (2.62,4.2);
\draw (2.62,4.2)-- (1.93,3.04);
\draw (1.93,3.04)-- (1.91,1.69);
\draw (1.91,1.69)-- (2.57,0.51);
\draw (2.57,0.51)-- (3.73,-0.19);
\draw (3.73,-0.19)-- (5.08,-0.2);
\draw (5.08,-0.2)-- (6.26,0.46);
\draw (6.26,0.46)-- (6.96,1.62);
\draw (6.96,1.62)-- (6.97,2.97);
\draw (6.97,2.97)-- (6.31,4.15);
\draw (6.31,4.15)-- (5.15,4.84);
\draw (5.15,4.84)-- (3.8,4.86);
\draw (5.08,-0.2)-- (5.15,4.84);
\draw (5.08,-0.2)-- (6.97,2.97);
\draw (5.08,-0.2)-- (1.91,1.69);
\draw [color=red] (2.62,4.2)-- (5.08,-0.2);
\end{tikzpicture}

\hspace{10pt}

\begin{tikzpicture}[scale=0.65]
\fill[fill=black,fill opacity=0.1] (0,4) -- (-1.66,3.46) -- (-2.69,2.05) -- (-2.68,0.3) -- (-1.66,-1.11) -- (0,-1.65) -- (1.66,-1.11) -- (2.69,0.3) -- (2.69,2.05) -- (1.66,3.46) -- cycle;
\draw [fill=black,fill opacity=1.0] (-1.66,1.57) circle (0.2cm);
\draw [fill=black,fill opacity=1.0] (1.66,1.58) circle (0.2cm);
\draw (0,4)-- (-1.66,3.46);
\draw (-1.66,3.46)-- (-2.69,2.05);
\draw (-2.69,2.05)-- (-2.68,0.3);
\draw (-2.68,0.3)-- (-1.66,-1.11);
\draw (-1.66,-1.11)-- (0,-1.65);
\draw (0,-1.65)-- (1.66,-1.11);
\draw (1.66,-1.11)-- (2.69,0.3);
\draw (2.69,0.3)-- (2.69,2.05);
\draw (2.69,2.05)-- (1.66,3.46);
\draw (1.66,3.46)-- (0,4);
\draw(-1.66,1.17) circle (0.4cm);
\draw(1.66,1.18) circle (0.4cm);
\draw [shift={(-7.95,-1.41)}] plot[domain=-0.03:0.66,variable=\t]({1*7.96*cos(\t r)+0*7.96*sin(\t r)},{0*7.96*cos(\t r)+1*7.96*sin(\t r)});
\draw [shift={(8.42,-1.56)},color=red]  plot[domain=2.5:3.15,variable=\t]({1*8.42*cos(\t r)+0*8.42*sin(\t r)},{0*8.42*cos(\t r)+1*8.42*sin(\t r)});
\draw [shift={(-1.34,-1.17)}] plot[domain=-1.77:1.07,variable=\t]({-0.15*3.42*cos(\t r)+-0.99*1.28*sin(\t r)},{0.99*3.42*cos(\t r)+-0.15*1.28*sin(\t r)});
\draw [shift={(1.34,-1.17)}] plot[domain=-1.77:1.07,variable=\t]({0.15*3.42*cos(\t r)+0.99*1.28*sin(\t r)},{0.99*3.42*cos(\t r)+-0.15*1.28*sin(\t r)});
\draw [shift={(1.47,-0.36)}] plot[domain=-0.32:1.37,variable=\t]({-0.91*3.99*cos(\t r)+-0.42*1.82*sin(\t r)},{0.42*3.99*cos(\t r)+-0.91*1.82*sin(\t r)});
\draw [shift={(-1.47,-0.36)}] plot[domain=-0.32:1.37,variable=\t]({0.9*3.99*cos(\t r)+0.43*1.82*sin(\t r)},{0.43*3.99*cos(\t r)+-0.9*1.82*sin(\t r)});
\draw [shift={(-0.37,0.1)}] plot[domain=2.86:5.31,variable=\t]({0.78*2.44*cos(\t r)+0.63*1.36*sin(\t r)},{-0.63*2.44*cos(\t r)+0.78*1.36*sin(\t r)});
\draw [shift={(-1.26,-1.78)}] plot[domain=-1.59:0.07,variable=\t]({-0.16*3.73*cos(\t r)+-0.99*1.26*sin(\t r)},{0.99*3.73*cos(\t r)+-0.16*1.26*sin(\t r)});
\draw [shift={(1.26,-1.78)}] plot[domain=-1.59:0.07,variable=\t]({0.16*3.73*cos(\t r)+0.99*1.26*sin(\t r)},{0.99*3.73*cos(\t r)+-0.16*1.26*sin(\t r)});
\draw [shift={(0.37,0.1)}] plot[domain=2.86:5.31,variable=\t]({-0.78*2.44*cos(\t r)+-0.63*1.36*sin(\t r)},{-0.63*2.44*cos(\t r)+0.78*1.36*sin(\t r)});
\end{tikzpicture}
$\to$
\begin{tikzpicture}[scale=0.65]
\fill[fill=black,fill opacity=0.1] (0,4) -- (-1.66,3.46) -- (-2.69,2.05) -- (-2.68,0.3) -- (-1.66,-1.11) -- (0,-1.65) -- (1.66,-1.11) -- (2.69,0.3) -- (2.69,2.05) -- (1.66,3.46) -- cycle;
\draw [fill=black,fill opacity=1.0] (-1.66,1.57) circle (0.2cm);
\draw [fill=black,fill opacity=1.0] (1.66,1.58) circle (0.2cm);
\draw (0,4)-- (-1.66,3.46);
\draw (-1.66,3.46)-- (-2.69,2.05);
\draw (-2.69,2.05)-- (-2.68,0.3);
\draw (-2.68,0.3)-- (-1.66,-1.11);
\draw (-1.66,-1.11)-- (0,-1.65);
\draw (0,-1.65)-- (1.66,-1.11);
\draw (1.66,-1.11)-- (2.69,0.3);
\draw (2.69,0.3)-- (2.69,2.05);
\draw (2.69,2.05)-- (1.66,3.46);
\draw (1.66,3.46)-- (0,4);
\draw(-1.66,1.17) circle (0.4cm);
\draw(1.66,1.18) circle (0.4cm);
\draw [shift={(-7.95,-1.41)}] plot[domain=-0.03:0.66,variable=\t]({1*7.96*cos(\t r)+0*7.96*sin(\t r)},{0*7.96*cos(\t r)+1*7.96*sin(\t r)});
\draw [shift={(-1.34,-1.17)}] plot[domain=-1.77:1.07,variable=\t]({-0.15*3.42*cos(\t r)+-0.99*1.28*sin(\t r)},{0.99*3.42*cos(\t r)+-0.15*1.28*sin(\t r)});
\draw [shift={(1.34,-1.17)}] plot[domain=-1.77:1.07,variable=\t]({0.15*3.42*cos(\t r)+0.99*1.28*sin(\t r)},{0.99*3.42*cos(\t r)+-0.15*1.28*sin(\t r)});
\draw [shift={(1.47,-0.36)}] plot[domain=-0.32:1.37,variable=\t]({-0.91*3.99*cos(\t r)+-0.42*1.82*sin(\t r)},{0.42*3.99*cos(\t r)+-0.91*1.82*sin(\t r)});
\draw [shift={(-1.47,-0.36)}] plot[domain=-0.32:1.37,variable=\t]({0.9*3.99*cos(\t r)+0.43*1.82*sin(\t r)},{0.43*3.99*cos(\t r)+-0.9*1.82*sin(\t r)});
\draw [shift={(-0.37,0.1)}] plot[domain=2.86:5.31,variable=\t]({0.78*2.44*cos(\t r)+0.63*1.36*sin(\t r)},{-0.63*2.44*cos(\t r)+0.78*1.36*sin(\t r)});
\draw [shift={(-1.26,-1.78)}] plot[domain=-1.59:0.07,variable=\t]({-0.16*3.73*cos(\t r)+-0.99*1.26*sin(\t r)},{0.99*3.73*cos(\t r)+-0.16*1.26*sin(\t r)});
\draw [shift={(1.26,-1.78)}] plot[domain=-1.59:0.07,variable=\t]({0.16*3.73*cos(\t r)+0.99*1.26*sin(\t r)},{0.99*3.73*cos(\t r)+-0.16*1.26*sin(\t r)});
\draw [shift={(0.37,0.1)}] plot[domain=2.86:5.31,variable=\t]({-0.78*2.44*cos(\t r)+-0.63*1.36*sin(\t r)},{-0.63*2.44*cos(\t r)+0.78*1.36*sin(\t r)});
\draw [shift={(-0.54,-0.5)},color=red]  plot[domain=0.67:1.85,variable=\t]({1*4.12*cos(\t r)+0*4.12*sin(\t r)},{0*4.12*cos(\t r)+1*4.12*sin(\t r)});
\end{tikzpicture}
\caption{Example of a flip for cases $A$ and $\tilde{D}$}
\label{fig:flip}
\end{figure}

\begin{rmk}
\cor{We set a figure which is the geometric realization of type $A$, $D$, $\tilde{A}$, or $\tilde{D}$ as defined above. Let $\Delta$ be an $(m+2)$-angulation. Let $\Delta'$ be another $(m+2)$-angulation of this figure. Then there exists a finite sequence of flips $\mu_{\alpha_n} \circ \cdots \circ \mu_{\alpha_n}$ such that \[\Delta'=\mu_{\alpha_n} \circ \cdots \circ \mu_{\alpha_n}(\Delta)\] where, for $i \in \{1,\cdots,n\}$, $\alpha_i$ is some $m$-diagonal of $\mu_{\alpha_{i-1}} \circ \cdots \circ \mu_{\alpha_1}(\Delta)$, using the convention that $\mu_{\alpha_0}(\Delta)=\Delta$.}
\end{rmk}

\cor{This is shown by Torkildsen in \cite{T} for case $\tilde{A}$ and by the author in \cite{JM} for case $\tilde{D}$. The proof is similar and left to the reader for cases $A$ and $D$.}

With this lemma and the fact that the flip does not change the number of $m$-diagonals in an $(m+2)$-angulation, we notice that all the $(m+2)$-angulations contain exactly $n$ $m$-diagonals.

For case $A$, we can cite the paper of Tzanaki, in \cite{Tz}.

Note that if $\Lambda$ is a set of noncrossing $m$-diagonals, it can be completed with $m$-diagonals in order to form an $(m+2)$-angulation.

We now associate a colored quiver with an $(m+2)$-angulation.

\begin{defi}\label{def:colquiver}
Let $\Delta$ be an $(m+2)$-angulation. We define the colored quiver $Q_\Delta$ associated with $\Delta$ in the following way:

\begin{enumerate}
\item The vertices of $Q_\Delta$ are in bijection with the $m$-diagonals of $\Delta$.
\item If $i$ and $j$ form two sides of some $(m+2)$-angle in $\Delta$, then we draw an arrow from $i$ to $j$ and an arrow from $j$ to $i$. The color of the corresponding arrow is the number of edges between both $m$-diagonals, counted counterclockwise from $i$ (respectively from $j$).
\end{enumerate}
\end{defi}

\begin{prop}
There is an equivalent definition: the vertices are similarly defined, and for $i$ and $j$ two vertices, and $c$ an integer,
\[ 
q_{ij}^{(c)} = \left\{
    \begin{array}{ll}
        1 & \mbox{if there is a counterclockwise oriented angle from } \kappa_\Delta^c(i) \mbox{ to } j \\
        0 & \mbox{otherwise.}
    \end{array}
    \right.
\]
\end{prop}

\begin{proof}
To understand this proof, let us recall that the colors are counted modulo $m+1$.

We only have to show that the arrows are the same. If $i$ and $j$ form two sides of the polygon, with a color $c$, it means that if we apply the twist to $i$, then there will be $c-1$ edges from $\kappa_\Delta(i)$ to $j$. Then if we apply the twist $c$ times, there will be no edge from $ \kappa_\Delta^c(i)$ to $j$, and they will share an oriented angle \cor{(note that, here, the number $c$ is a power, not an index).}

On the other hand, if $ \kappa_\Delta^c(i)$ and $j$ share an oriented angle, it suffices to apply the inverse of the twist $c$ times to make sure that $i$ and $j$ form two sides of a polygon, and that there are $c$ edges between $i$ and $j$.
\end{proof}

\begin{lem}
The quiver fulfills the conditions asked for colored quivers in the article of Buan and Thomas \cite{BT}. In particular it is symmetric.
\end{lem}

\begin{proof}
By definition, the quiver contains no loops (it means, no arrows from $i$ to $i$).

The condition of monochromaticity is respected since two arcs can only share one polygon.

If there is an arrow from $i$ to $j$ of color $c$, it means that $i$ and $j$ are sides of the same $(m+2)$-angle in the $(m+2)$-angulation. If we count from $i$ to $j$, there are $c$ edges between them. But if we count from $j$ to $i$, as we deal with $(m+2)$-angles, it means that from $j$ to $i$ there are $m-c$ edges. So there is an arrow from $j$ to $i$ of color $m-c$. Then the symmetry is respected.
\end{proof}

We remark that we have the compatibility between the mutation of a colored quiver in the sense of Buan and Thomas, and the flip of an $(m+2)$-angulation in all cases.

\begin{theo}\label{theo:corresp}
Let $\Delta$ be any $(m+2)$-angulation. Let $Q_\Delta$ be the colored quiver associated with the $(m+2)$-angulation $\Delta$. Let us introduce $\Delta_k=\mu_k\Delta$, which means that $\Delta_k$ is obtained from $\Delta$ by mutation at $k$, then the colored quiver $Q_{\Delta_k}$ associated with $\Delta_k$ is the mutation at vertex $k$ of the colored quiver $Q_\Delta$.

To be clear, \[Q_{\mu_k\Delta}=\mu_k(Q_\Delta)\]
where $\mu_k$ denotes both flip of the $(m+2)$-angulation (at the left of the equation) and colored quiver mutation (at the right of the equation).
\end{theo}

\begin{proof}
This is shown by Buan and Thomas in case $A$ (\cite{BT}), Torkildsen in case $\tilde{A}$ (\cite{Tor}, Proposition 5.1), and by the author for case $\tilde{D}$ (\cite{JM}, Theorem 2.28). Case $D$, with the adaptation of geometric realization from the one of Baur and Marsh is similar to that of case $\tilde{D}$ and the proof is left to the reader.
\end{proof}

We now arrive to the main theorem of the section:

\begin{theo}[\cite{BM01}, Proposition 5.4]\label{th:ra}
Let $Q$ be a quiver of type $A$, $D$, $\tilde{A}$ or $\tilde{D}$. Let $\mathcal{C}^{m}_Q$ be the $m$-cluster category associated with $Q$.

There is a $1-1$ correspondence between the $m$-diagonals of the geometric realization, and the $m$-rigid indecomposable objects of $\mathcal{C}^{m}_Q$.
\end{theo}

\cor{This Theorem is claimed and shown for type $A$ (respectively $D$) by Baur and Marsh in \cite{BM01}, Proposition 5.4 (respectively \cite{BM02}, consequence of Theorem 3.6), for type $\tilde{A}$ by Torkildsen (\cite{Tor}, consequence of Theorem 7.3), and for type $\tilde{D}$ by the author (\cite{JM}, Lemma 5.1).}

\cor{For cases $A$, $D$, and $\tilde{A}$, this bijection is found in the following way: the authors build a quiver from $m$-diagonals, which is aimed to be isomorphic to the Auslander-Reiten quiver of $Q$. This Theorem is a consequence to the isomorphism of these quivers.}

\section{Noncrossing arcs and extensions}

\cor{All throughout the section, we identify the $m$-diagonals with the vertices of the Auslander-Reiten quiver. Indeed, thanks to Theorem \ref{th:ra}, we can identify each $m$-diagonal of the geometric realization with an $m$-rigid indecomposable object in the associated higher cluster category.}

In this section, we are going to show in types $A$, $D$, $\tilde{A}$ and $\tilde{D}$ the following theorem :

\begin{theo}\label{th:cross}
Let $\alpha$ and $\beta$ be two arcs in the polygon $P$. Let $X_\alpha$ and $X_\beta$ be the associated $m$-rigid objects. If $\forall i \in \{1,\cdots, m\}, \mathrm{Ext}^i_\mathcal{C}(X_\alpha,X_\beta)=0$, then $\alpha$ and $\beta$ do not cross each other.
\end{theo}

\begin{rmk}
The result in cases $A$ and $D$ has already been shown by Thomas in \cite{Tho} and by Baur and Marsh in \cite{BM01} for case $A$ and \cite{BM02} for case $D$.

We nonetheless include a proof as it illustrates the method that will be applied in types $\tilde{A}$ and $\tilde{D}$.
\end{rmk}

\cor{Our strategy to prove this is based on the work of Marsh and Palu in \cite{MP}. It consists in showing that cutting along a well-chosen $m$-diagonal $\alpha$ (we will define in this section whats means "cutting along") corresponds in the higher cluster category, to applying the Iyama-Yoshino reduction. To be precise, we are going to show that, under hypotheses, there is an equivalence of categories between on the one hand, the higher cluster category associated with a quiver of the geometric realization where we have forgotten $\alpha$, and on the other hand, the Iyama Yoshino reduction of the higher cluster category under the $m$-rigid indecomposable object associated with $\alpha$. But first, let us show a useful lemma:}

\begin{lem}\label{lem:shift}
Let $\mathcal{C}$ be a Hom-finite triangulated category with a Serre functor. Let $X \in \mathcal{C}$ be an $m$-rigid indecomposable object. Let $Y$ be an object of $\mathcal{C}$ which belongs to $X^\perp$. Suppose that $\mathcal{C}(X,Y)=0$ and for all $i \in \{1,\cdots,k\}$, where $k \leq m$, we have $\mathrm{Ext}^{-i}_\mathcal{C}(X,Y)=0$, then, we have the following isomorphism, in the Iyama-Yoshino reduction:

\[ \forall i \in \{1,\cdots, k \} Y \langle -i \rangle \simeq Y[-i] \]

where $[1]$ denotes the shift in $\mathcal{C}$ and $\langle 1 \rangle $ is the shift in the Iyama-Yoshino reduction $X^\perp/(X)$.
\end{lem}

\begin{proof}
We are going to use Theorem \ref{th:iy}. We show by induction on $i$ that $Y \langle -i \rangle \simeq Y[-i]$. It is defined on objects as follows: Let \[
\xymatrix{
T_k^{(c)} \ar^{f_k^{(c)}}[r] & B_k^{(c)} \ar^{g_k^{(c+1)}}[r] & T_k^{(c+1)} \ar^{h_k^{(c+1)}}[r] & T_k^{(c)}[1] }
\]
be the exchange triangles as seen in Theorem \ref{th:iy}. If we use the notations of the Theorem, we have that $Y=T_k^{(c)}$. Then $Y\langle1\rangle$ is in fact the object $T_k^{(c+1)}$ in the exchange triangle.

First, $Y \langle -1 \rangle \simeq Y[-1]$. Indeed, $\mathcal{C}(X,Y)=0$. \cor{Then, the category $\mathcal{C}$ being triangulated, there is a triangle $0 \to Y \to Y\langle 1 \rangle \to Y[1]$. Then we have the following triangle:}

\[ \xymatrix@1{
Y[-1] \ar[r] & Y \langle -1 \rangle \ar[r] & 0 \ar[r]^0 & Y}. \]

As the right morphism is zero, we have that $Y \langle -1 \rangle \simeq Y[-1]$.

Suppose $Y \langle -i+1 \rangle ~\simeq ~ Y[-i+1]$ for an $i \in \{1,\cdots,k\}$. Then \[ \mathcal{C}(X,Y \langle -i+1 \rangle )=\mathcal{C}(X,Y[-i+1])=0 \]

Let us once again take an $\mathrm{add}_X$-approximation of $Y \langle -i+1 \rangle$. Then we have \[ \xymatrix@1{
Y[-i] \ar[r] & Y\langle -i \rangle \ar[r] & 0 \ar[r] & Y[-i+1]}. \]
As the right morphism is zero, we have an isomorphism $Y \langle -i \rangle \simeq Y[-i]$.
\end{proof}

In this section, we will need a proper definition of "cutting along". All throughout the section, the letter $S$ denotes the whole geometric realization, containing all the polygons, disks, edges and vertices.

\begin{defi}
Let $P$ be a polygon of the geometric realization of type $A$, $D$, $\tilde{A}$ or $\tilde{D}$. Let $\alpha$ be an $m$-diagonal in it. The action of cutting the figure along $\alpha$ consists in separating the figure into two, where $\alpha$ becomes an edge of each figure. See figure \ref{fig:cut} for an illustration.
\end{defi}

\begin{figure}[h!]
\centering
\begin{tikzpicture}[scale=0.65]
\fill[fill=black,fill opacity=0.1] (0.,4.) -- (-1.78,3.42) -- (-2.87913480366,1.90451239418) -- (-2.87757227419,0.0324019384391) -- (-1.77590924476,-1.48124880383) -- (0.00505645156165,-2.05827669617) -- (1.78505645156,-1.47827669617) -- (2.88419125522,0.0372109096461) -- (2.88262872576,1.90932136539) -- (1.78096569632,3.42297210766) -- cycle;
\draw [fill=black,fill opacity=1.0] (-1.35763963744,0.969726407807) circle (0.4cm);
\draw [fill=black,fill opacity=1.0] (1.362696089,0.971996896021) circle (0.4cm);
\draw (0.,4.)-- (-1.78,3.42);
\draw (-1.78,3.42)-- (-2.87913480366,1.90451239418);
\draw (-2.87913480366,1.90451239418)-- (-2.87757227419,0.0324019384391);
\draw (-2.87757227419,0.0324019384391)-- (-1.77590924476,-1.48124880383);
\draw (-1.77590924476,-1.48124880383)-- (0.00505645156165,-2.05827669617);
\draw (0.00505645156165,-2.05827669617)-- (1.78505645156,-1.47827669617);
\draw (1.78505645156,-1.47827669617)-- (2.88419125522,0.0372109096461);
\draw (2.88419125522,0.0372109096461)-- (2.88262872576,1.90932136539);
\draw (2.88262872576,1.90932136539)-- (1.78096569632,3.42297210766);
\draw (1.78096569632,3.42297210766)-- (0.,4.);
\draw [shift={(5.67552365823,-0.867172962917)}] plot[domain=2.3078995791:3.34863646756,variable=\t]({1.*5.7942149119*cos(\t r)+0.*5.7942149119*sin(\t r)},{0.*5.7942149119*cos(\t r)+1.*5.7942149119*sin(\t r)});
\end{tikzpicture}
\includegraphics[scale=0.35]{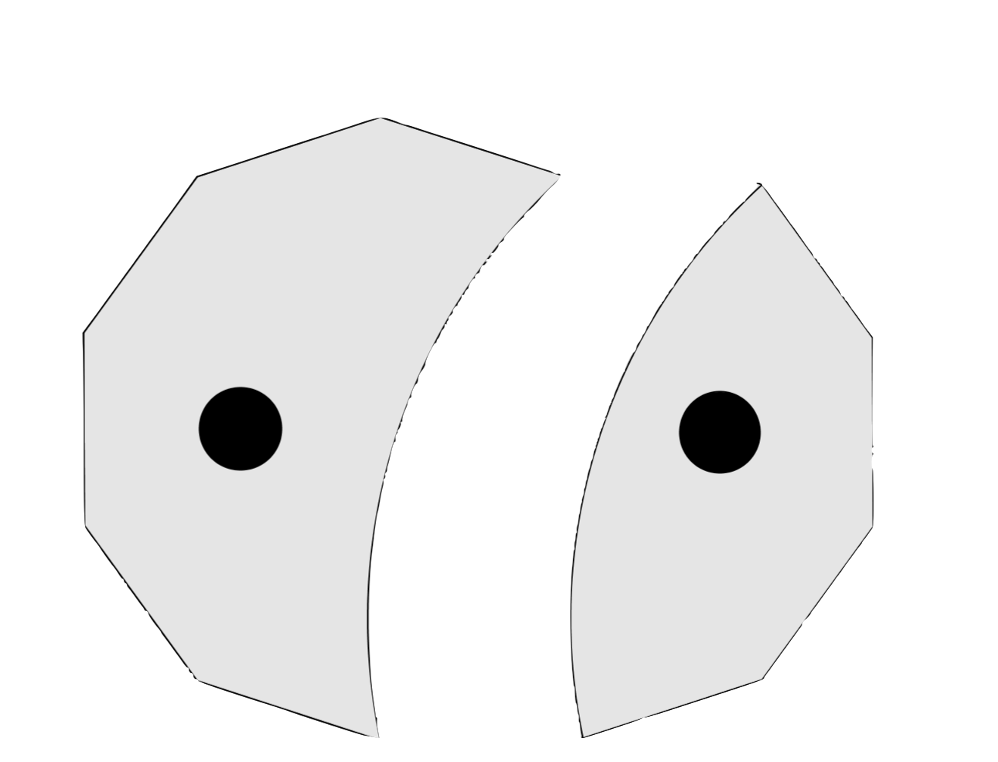}
\caption{Cutting along an $m$-diagonal in type $\tilde{D}$.}
\label{fig:cut}
\end{figure}

\cor{We observe that, cutting along this arc leads to two figures of type $D$. note that this $m$-diagonal corresponds to a vertex at the center of the quiver of type $\tilde{D}$. Strangely, dividing a quiver of type $\tilde{D}$ at this corresponding vertex gives two vertices of type $D$. We will show that this is not a coincidence.}

\subsection{Cases $A$ and $D$}

\cor{In order to set our strategy, we start with the study of $m$-ears by giving a useful lemma.}

\begin{defi}
We call an $m$-diagonal $\alpha$ an $m$-ear, when $\alpha$ divides $P$ into an $m+2$-gon and an $(n-1)m+2$-gon for the $A$ case (respectively $(n-1)m-m+1$-gon containing the interior polygon for case $D$).
\end{defi}

Note that any $(m+2)$-angulation of a figure of type $A$ or $D$, necessarily contains an $m$-ear. We can draw the associated colored quiver. For example, in the $(m+2)$-angulations in the figure \ref{fig:AD}, the associated colored quivers are: $1 \to 2 \to \cdots \to n$ or \[\xymatrix{ & & & & n-1 \\ 1 \ar[r] & 2 & \cdots \ar[l] & n-2 \ar[l] \ar[ur] \ar[dr] & \\ & & & & n}.\]

\begin{figure}[h!]
\centering
\begin{tikzpicture}[scale=0.7]

	\fill[line width=2pt,fill=black,fill opacity=0.1] (-5.81,3.19) -- (-3.71,3.19) -- (-2.225075759508251,4.674924240491749) -- (-2.225075759508251,6.774924240491748) -- (-3.71,8.259848480983498) -- (-5.81,8.259848480983498) -- (-7.294924240491749,6.77492424049175) -- (-7.294924240491749,4.67492424049175) -- cycle;

	\draw [line width=2pt] (-5.81,3.19)-- (-3.71,3.19);

	\draw [line width=2pt] (-3.71,3.19)-- (-2.225075759508251,4.674924240491749);

	\draw [line width=2pt] (-2.225075759508251,4.674924240491749)-- (-2.225075759508251,6.774924240491748);

	\draw [line width=2pt] (-2.225075759508251,6.774924240491748)-- (-3.71,8.259848480983498);

	\draw [line width=2pt] (-3.71,8.259848480983498)-- (-5.81,8.259848480983498);

	\draw [line width=2pt] (-5.81,8.259848480983498)-- (-7.294924240491749,6.77492424049175);

	\draw [line width=2pt] (-7.294924240491749,6.77492424049175)-- (-7.294924240491749,4.67492424049175);

	\draw [line width=2pt] (-7.294924240491749,4.67492424049175)-- (-5.81,3.19);

	\draw [shift={(-7.203422514538918,6.7370229846215395)},line width=2pt]  plot[domain=-0.7930112376757146:0.007613074278266029,variable=\t]({1*4.978491028265584*cos(\t r)+0*4.978491028265584*sin(\t r)},{0*4.978491028265584*cos(\t r)+1*4.978491028265584*sin(\t r)});

	\draw [shift={(-9.164664493606315,5.72492424049175)},line width=2pt]  plot[domain=-0.4350323362210098:0.43503233622100984,variable=\t]({1*6.014915264809814*cos(\t r)+0*6.014915264809814*sin(\t r)},{0*6.014915264809814*cos(\t r)+1*6.014915264809814*sin(\t r)});

	\draw [shift={(-9.05742223886937,3.944873665908306)},line width=2pt]  plot[domain=-0.1402392442906919:0.9256374076881401,variable=\t]({1*5.400440616490154*cos(\t r)+0*5.400440616490154*sin(\t r)},{0*5.400440616490154*cos(\t r)+1*5.400440616490154*sin(\t r)});

	\draw [shift={(-7.304340905032673,3.1805833354590765)},line width=2pt]  plot[domain=0.0026198524974553107:1.5681764742974413,variable=\t]({1*3.594353240175786*cos(\t r)+0*3.594353240175786*sin(\t r)},{0*3.594353240175786*cos(\t r)+1*3.594353240175786*sin(\t r)});

	\draw [shift={(-6.895085126611702,0.570372771004706)},line width=2pt]  plot[domain=0.6882909584210836:1.6679035317712605,variable=\t]({1*4.123980368849571*cos(\t r)+0*4.123980368849571*sin(\t r)},{0*4.123980368849571*cos(\t r)+1*4.123980368849571*sin(\t r)});

	\begin{scriptsize}

		\draw [fill=black] (-5.81,3.19) circle (2.5pt);

		\draw [fill=black] (-3.71,3.19) circle (2.5pt);

		\draw [fill=black] (-2.225075759508251,4.674924240491749) circle (2.5pt);

		\draw [fill=black] (-2.225075759508251,6.774924240491748) circle (2.5pt);

		\draw [fill=black] (-3.71,8.259848480983498) circle (2.5pt);

		\draw [fill=black] (-5.81,8.259848480983498) circle (2.5pt);

		\draw [fill=black] (-7.294924240491749,6.77492424049175) circle (2.5pt);

		\draw [fill=black] (-7.294924240491749,4.67492424049175) circle (2.5pt);

	\end{scriptsize}

\end{tikzpicture}
\hspace{10pt}
\begin{tikzpicture}[scale=0.7]

	\fill[line width=2pt,fill=black,fill opacity=0.1] (-5.81,3.19) -- (-3.71,3.19) -- (-2.225075759508251,4.674924240491749) -- (-2.225075759508251,6.774924240491748) -- (-3.71,8.259848480983498) -- (-5.81,8.259848480983498) -- (-7.294924240491749,6.77492424049175) -- (-7.294924240491749,4.67492424049175) -- cycle;

	\draw [line width=2pt,fill=black,fill opacity=1] (-4.79,5.73) circle (0.5cm);

	\draw [line width=2pt] (-5.81,3.19)-- (-3.71,3.19);

	\draw [line width=2pt] (-3.71,3.19)-- (-2.225075759508251,4.674924240491749);

	\draw [line width=2pt] (-2.225075759508251,4.674924240491749)-- (-2.225075759508251,6.774924240491748);

	\draw [line width=2pt] (-2.225075759508251,6.774924240491748)-- (-3.71,8.259848480983498);

	\draw [line width=2pt] (-3.71,8.259848480983498)-- (-5.81,8.259848480983498);

	\draw [line width=2pt] (-5.81,8.259848480983498)-- (-7.294924240491749,6.77492424049175);

	\draw [line width=2pt] (-7.294924240491749,6.77492424049175)-- (-7.294924240491749,4.67492424049175);

	\draw [line width=2pt] (-7.294924240491749,4.67492424049175)-- (-5.81,3.19);

	\draw [shift={(-7.203422514538918,6.7370229846215395)},line width=2pt]  plot[domain=-0.7930112376757146:0.007613074278266029,variable=\t]({1*4.978491028265584*cos(\t r)+0*4.978491028265584*sin(\t r)},{0*4.978491028265584*cos(\t r)+1*4.978491028265584*sin(\t r)});

	\draw [shift={(-9.164664493606315,5.72492424049175)},line width=2pt]  plot[domain=-0.4350323362210098:0.43503233622100984,variable=\t]({1*6.014915264809814*cos(\t r)+0*6.014915264809814*sin(\t r)},{0*6.014915264809814*cos(\t r)+1*6.014915264809814*sin(\t r)});

	\draw [shift={(-3.157759174983143,9.593075770435298)},line width=2pt]  plot[domain=4.013024101644992:4.626355695726939,variable=\t]({1*6.426845980008459*cos(\t r)+0*6.426845980008459*sin(\t r)},{0*6.426845980008459*cos(\t r)+1*6.426845980008459*sin(\t r)});

	\draw [shift={(-3.4067294056737296,7.07819483481802)},line width=2pt]  plot[domain=3.2194328453656786:4.634548788608804,variable=\t]({1*3.9000041188296892*cos(\t r)+0*3.9000041188296892*sin(\t r)},{0*3.9000041188296892*cos(\t r)+1*3.9000041188296892*sin(\t r)});

	\draw [shift={(-3.312834733878982,6.324359720714343)},line width=2pt]  plot[domain=2.4822361326244184:4.586347337952615,variable=\t]({1*3.159422590830399*cos(\t r)+0*3.159422590830399*sin(\t r)},{0*3.159422590830399*cos(\t r)+1*3.159422590830399*sin(\t r)});

	\draw [shift={(-6.3677624602332985,4.494453870625662)},line width=2pt]  plot[domain=-0.4562678524137693:0.8661611661912132,variable=\t]({1*2.960625135949438*cos(\t r)+0*2.960625135949438*sin(\t r)},{0*2.960625135949438*cos(\t r)+1*2.960625135949438*sin(\t r)});

	\draw [shift={(-4.264026035504306,5.033372606238494)},line width=2pt]  plot[domain=-0.6563511523444188:2.7741134367816302,variable=\t]({-0.47092809546053704*2.022079637248662*cos(\t r)+-0.88217159833329*1.0559082511117115*sin(\t r)},{0.88217159833329*2.022079637248662*cos(\t r)+-0.47092809546053704*1.0559082511117115*sin(\t r)});

	\draw [shift={(-3.2884614864731807,5.724924240491749)},line width=2pt]  plot[domain=1.7355807838461155:4.5476045233334705,variable=\t]({1*2.5697345433758465*cos(\t r)+0*2.5697345433758465*sin(\t r)},{0*2.5697345433758465*cos(\t r)+1*2.5697345433758465*sin(\t r)});

	\begin{scriptsize}

		\draw [fill=black] (-5.81,3.19) circle (2.5pt);

		\draw [fill=black] (-3.71,3.19) circle (2.5pt);

		\draw [fill=black] (-2.225075759508251,4.674924240491749) circle (2.5pt);

		\draw [fill=black] (-2.225075759508251,6.774924240491748) circle (2.5pt);

		\draw [fill=black] (-3.71,8.259848480983498) circle (2.5pt);

		\draw [fill=black] (-5.81,8.259848480983498) circle (2.5pt);

		\draw [fill=black] (-7.294924240491749,6.77492424049175) circle (2.5pt);

		\draw [fill=black] (-7.294924240491749,4.67492424049175) circle (2.5pt);

	\end{scriptsize}

\end{tikzpicture}
\caption{Two $(m+2)$-angulations for type $A$ and $D$.}
\label{fig:AD}
\end{figure}

\cor{Notice that an $m$-ear in cases $A$ and $D$ ending at vertex $1$ of the polygon $P$ corresponds to vertex $1$ in the quiver $Q$. This is morally obvious. Indeed, this is compatible with the fact that the new figure obtained by cutting along $\alpha$ realizes geometrically the new quiver without vertex $1$.}

\begin{lem}\label{lem:cutalong}
\cor{Let $\mathcal{C}$ be the $m$-cluster category of type $A_n$ (respectively $D_n$). Let $\alpha$ be an $m$-ear in the corresponding polygon. Let $X_\alpha$ be the $m$-rigid object associated with $\alpha$ (we can take this object from Theorem \ref{th:ra}). Let \[\mathcal{U}=\{ Y \in \mathcal{C},~\forall i \in \{1,\cdots,m\},~\mathrm{Ext}_\mathcal{C}^i(X_\alpha,Y)=0\}\] be a subcategory of $\mathcal{C}$. Let $\mathcal{C'}$ be the Iyama-Yoshino reduction of $\mathcal{C}$ defined in Theorem \ref{th:iy}: $\mathcal{C'}=\mathcal{U}/(X_\alpha)$. Then, we have an equivalence of categories :
\[ \mathcal{C'} \simeq \mathcal{C}^{m}_{Q/\alpha} \]
where $Q/\alpha$ is the quiver obtained from $Q$ by removing $\alpha$ and all incident arrows.}
\end{lem}

\begin{proof}
\cor{This lemma is a consequence of theorem \ref{th:kr} of Keller and Reiten, applied to $\mathcal{C}'$. We first notice that this category is Hom-finite, algebraic and $(m+1)$-Calabi-Yau.}

\cor{Let us find an $m$-cluster-tilting object $T$ in $\mathcal{C'}$ such that for any $i \in \{ 1,\cdots,m \}$, we have \[\mathrm{Ext}_\mathcal{C}^{-i}(T,T)=0,\] and $\mathrm{End}(T) \simeq KA_{n-1}$ (respectively $\mathrm{End}(T) \simeq KD_{n-1}$).} We recall that we chose the clockwise convention, it means that we draw the arrows of the quiver of an $(m+2)$-angulation clockwise. For each $i \in \{1,\cdots,n\}$, we introduce $P_i$ the projective object at vertex $i$ of the quiver $Q$.

\cor{As there exists an $m$-ear $\alpha_0$ ending at vertex $1$ of the polygon $P$, correpsonding to the first projective object $P_1$ (since it corresponds to vertex $1$ of the quiver $Q$), we obtain $\alpha$ from $\alpha_0$ by applying a rotation (since they are both $m$-ears). We may thus assume that $X_\alpha=P_1$.} Let \cor{\[T=\bigoplus_{i=2}^nP_i,\]} viewed both as an object in $\mathcal{C}^m_Q$ and of $\mathcal{C'}$. We have that $\mathrm{End}_{\mathcal{C'}}(T) \simeq KA_{n-1}$. Indeed, first, we have that $\mathrm{End}_{\mathrm{mod(KQ)}}(T)=KA_{n-1}$ (respectively $KD_{n-1}$) because $X_\alpha=P_1$ and $1$ is a source in $Q$. \cor{Moreover, the object $T$ is an $m$-cluster-tilting object in $\mathcal{C}'$. Indeed, it has $n-1$ indecomposable summands, and is $m$-rigid (since the vertex corresponding to $\alpha$ has no incident arrows).}

Moreover, from the Lemma of Keller and Reiten in \cite[Section 4]{KR}, we have that \[\forall i \in \{ 1,\cdots,m \}~\mathrm{Ext}_\mathcal{C}^{-i}(T,T)=0.\] From lemma \ref{lem:shift} and the fact that $1$ is a source in $Q$, we then have \[\forall i \in \{ 1,\cdots,m \}~\mathrm{Ext}_\mathcal{C'}^{-i}(T,T)=0.\] Then we have shown the lemma.
\end{proof}

For case $A$ (respectively for case $D$), if $\beta$ is an arc which does not cross $\alpha$, then we can cut along $\alpha$ in order to have two new figures of type $A$ (respectively one of type $D$ and one of type $A$) and one of these contains the same arc as $\beta$, which we still call as $\beta$.

\begin{figure}[h!]
\centering
\begin{tikzpicture}[scale=0.5]
\fill[fill=black,fill opacity=0.1] (0.,4.) -- (-1.78,3.42) -- (-2.87913480366,1.90451239418) -- (-2.87757227419,0.0324019384391) -- (-1.77590924476,-1.48124880383) -- (0.00505645156165,-2.05827669617) -- (1.78505645156,-1.47827669617) -- (2.88419125522,0.0372109096461) -- (2.88262872576,1.90932136539) -- (1.78096569632,3.42297210766) -- cycle;
\draw [fill=black,fill opacity=1.0] (-0.0239151014275,1.18372937641) circle (0.4cm);
\draw (0.,4.)-- (-1.78,3.42);
\draw (-1.78,3.42)-- (-2.87913480366,1.90451239418);
\draw (-2.87913480366,1.90451239418)-- (-2.87757227419,0.0324019384391);
\draw (-2.87757227419,0.0324019384391)-- (-1.77590924476,-1.48124880383);
\draw (-1.77590924476,-1.48124880383)-- (0.00505645156165,-2.05827669617);
\draw (0.00505645156165,-2.05827669617)-- (1.78505645156,-1.47827669617);
\draw (1.78505645156,-1.47827669617)-- (2.88419125522,0.0372109096461);
\draw (2.88419125522,0.0372109096461)-- (2.88262872576,1.90932136539);
\draw (2.88262872576,1.90932136539)-- (1.78096569632,3.42297210766);
\draw (1.78096569632,3.42297210766)-- (0.,4.);
\draw [shift={(-3.91535404058,2.24024452172)}] plot[domain=-0.831370599592:0.204721339077,variable=\t]({1.*5.81780913124*cos(\t r)+0.*5.81780913124*sin(\t r)},{0.*5.81780913124*cos(\t r)+1.*5.81780913124*sin(\t r)});
\end{tikzpicture}
$\to$
\begin{tikzpicture}[scale=0.7]
\draw [fill=black,fill opacity=1.0] (-0.0239151014275,1.18372937641) circle (0.4cm);
\fill[fill=black,fill opacity=0.1] (1.82431930879,2.14541232156) -- (0.321689706987,3.19725304282) -- (-1.43754671613,2.67826187653) -- (-2.12864905208,0.979249756396) -- (-1.23120314596,-0.620392537592) -- (0.578995930349,-0.916101731302) -- (1.93883135106,0.314797229498) -- cycle;
\draw (1.82431930879,2.14541232156)-- (0.321689706987,3.19725304282);
\draw (0.321689706987,3.19725304282)-- (-1.43754671613,2.67826187653);
\draw (-1.43754671613,2.67826187653)-- (-2.12864905208,0.979249756396);
\draw (-2.12864905208,0.979249756396)-- (-1.23120314596,-0.620392537592);
\draw (-1.23120314596,-0.620392537592)-- (0.578995930349,-0.916101731302);
\draw (0.578995930349,-0.916101731302)-- (1.93883135106,0.314797229498);
\draw (1.93883135106,0.314797229498)-- (1.82431930879,2.14541232156);
\end{tikzpicture}
and
\begin{tikzpicture}[scale=0.9]
\fill[fill=black,fill opacity=0.1] (-0.955545454545,3.90348895567) -- (-2.65351690458,2.65630638618) -- (-1.99207782888,0.656038965072) -- (0.114685451432,0.666988281737) -- (0.755297689209,2.67402275269) -- cycle;
\draw (-0.955545454545,3.90348895567)-- (-2.65351690458,2.65630638618);
\draw (-2.65351690458,2.65630638618)-- (-1.99207782888,0.656038965072);
\draw (-1.99207782888,0.656038965072)-- (0.114685451432,0.666988281737);
\draw (0.114685451432,0.666988281737)-- (0.755297689209,2.67402275269);
\draw (0.755297689209,2.67402275269)-- (-0.955545454545,3.90348895567);
\end{tikzpicture}
\end{figure}

\begin{rmk}
\cor{Let $P$ be an $nm+2$-gon (respectively an $mn-m+1$-gon) associated with a quiver $Q$ of type $A_n$ (respectively $D_n$). Let $\alpha$ be an $m$-ear from $i$ to $j$. Then cutting along $\alpha$ corresponds to applying the Iyama-Yoshino reduction of $\mathcal{C}^m_{A_n}$ (respectively $\mathcal{C}^m_{D_n}$) on $X_\alpha$. To be precise, we call by $S/\alpha$ the whole figure where $P$ is replaced by $P'$, with the same sides as $P$ except that the boundary from $i$ to $j$ is replaced by $\alpha$ (then $P'$ becomes a $(n-1)m+2$-gon - respectively an $m(n-1)-m+1$-gon). Then the Iyama Yoshino reduction of $\mathcal{C}^m_{A_n}$ (respectively $\mathcal{C}^m_{D_n}$) on $X_\alpha$ is equivalent to the higher cluster category $\mathcal{C}^m_{A_{n-1}}$ (respectively $\mathcal{C}^m_{D_{n-1}}$).}
\end{rmk}

\begin{lem}\label{lem:cross}
\cor{Let $P$ be an $nm+2$-gon (respectively an $mn-m+1$-gon) associated with a quiver $Q$ of type $A_n$ (respectively $D_n$). Let $\alpha$ be an $m$-ear. Let $\beta$ be an $m$-diagonal which cuts $\alpha$ (see figure \ref{fig:mear}). Let $X_\beta$ be the associated $m$-rigid object. Then there exists $k \in \{1,\cdots,m\}$ such that $\mathrm{Ext}_\mathcal{C}^k(X_\alpha,X_\beta) \neq 0$.}
\end{lem}

\begin{figure}[!h]
\centering
\begin{tikzpicture}[scale=0.9]
\fill[fill=black,fill opacity=0.1] (3.8,4.86) -- (2.62,4.2) -- (1.93,3.04) -- (1.91,1.69) -- (2.57,0.51) -- (3.73,-0.19) -- (5.08,-0.2) -- (6.26,0.46) -- (6.96,1.62) -- (6.97,2.97) -- (6.31,4.15) -- (5.15,4.84) -- cycle;
\draw (3.8,4.86)-- (2.62,4.2);
\draw (2.62,4.2)-- (1.93,3.04);
\draw (1.93,3.04)-- (1.91,1.69);
\draw (1.91,1.69)-- (2.57,0.51);
\draw (2.57,0.51)-- (3.73,-0.19);
\draw (3.73,-0.19)-- (5.08,-0.2);
\draw (5.08,-0.2)-- (6.26,0.46);
\draw (6.26,0.46)-- (6.96,1.62);
\draw (6.96,1.62)-- (6.97,2.97);
\draw (6.97,2.97)-- (6.31,4.15);
\draw (6.31,4.15)-- (5.15,4.84);
\draw (5.15,4.84)-- (3.8,4.86);
\draw (5.08,-0.2)-- (6.97,2.97);
\draw [dash pattern=on 5pt off 5pt] (6.96,1.62)-- (2.64,4.18);
\begin{scriptsize}
\draw[color=black] (5.88,1.14) node {$\alpha$};
\draw[color=black] (5,3.36) node {$\beta$};
\end{scriptsize}
\end{tikzpicture}
\caption{Example of an $m$-ear in case $A$}
\label{fig:mear}
\end{figure}

\begin{proof}
Let $i$ be one end of $\alpha$, the other being $i+m+1$ (as $\alpha$ is an $m$-ear). If $\beta$ crosses $\alpha$, it means that there exists $k \in \{1,\cdots,m-1\}$ such that $k+i$ is an extremity of $\beta$. Then, we can shift $\beta$ $k$ times in order to have $\alpha$ and $\beta$ sharing an extremity. Let us show that there is a morphism from $X_\alpha$ to $X_\beta[k]$.
	
If we take $\beta$, an arc which does not cross $\alpha$. As there is a bijection between the Auslander-Reiten quiver of $Q$ and the translation quiver built in \cite{BM01} for case $A$ and \cite{BM02} for case $D$, the arc $\beta$ corresponds to a unique object $X_\beta$ situated on the Auslander-Reiten quiver of $Q$.
\begin{enumerate}
\item In case $A$: we assume that $\alpha$ is the arc from $1$ to $m+2$ with no loss of generality.

In the Auslander-Reiten quiver of $\mathcal{C}$, the $m$-rigid $X_\alpha$ is situated at the bottom as we can see in the next picture where we identify an arc with the associated object in the higher cluster category. We give the name of the arcs by $D_{1j}$, where the arcs links $1$ to $j$. Moreover, we draw the Hom-hammock in red. \cor{Recall that, in a triangulated category, the Hom-Hammock of an object $A$ is the class of objects $X$ such that $\mathcal{C}(A,X) \neq 0$.}

\[ \scalebox{0.6}{\xymatrix{
~ \ar[dr] & & ~ \ar[dr] & & \color{red} D_{1~5m+2} \ar[dr] & & ~ \ar[dr] & & \\
& ~ \ar[dr] \ar[ur] & & \color{red} D_{1~4m+2} \ar[ur] \ar[dr] & & D_{m+1~5m+2} \ar[dr] \ar[ur] & & ~ \ar[dr] \ar[ur] & \\
\cdots \ar[dr] \ar[ur] & & \color{red} D_{1~3m+2} \ar[ur] \ar[dr] & & D_{m+1~4m+2} \ar[ur] \ar[dr] & & D_{2m+1~5m+2} \ar[dr] \ar[ur] & & \cdots \\
& \color{red} D_{1~2m+2} \ar[ur] \ar[dr] & & D_{m+1~3m+2} \ar[ur] \ar[dr] & & D_{2m+1~4m+2} \ar[ur] \ar[dr] & & D_{3m+1~4m+2} \ar[dr] \ar[ur] & \\
\color{red} X_\alpha=D_{1~m+2} \ar[ur] & & D_{m+1~2m+2} \ar[ur] & & D_{2m+1~3m+2} \ar[ur] & & D_{3m+1~4m+2} \ar[ur] & & D_{4m+1~5m+2}
}}
\]

If we draw the corresponding arcs on the Auslander-Reiten quiver, we realize that the ones on the slice arising from $X_\alpha$ (on the figure, $P_2$, $P_3$, $P_4$, $P_5$) have an extremity equal to $1$. We note moreover that those are all arcs having $1$ as an end. Then $\beta[k]$ belongs to one of them.

It is also known that these objects exactly correspond to the ones which have a nonzero morphism from $X_\alpha$. Then $\mathrm{Ext}_{\mathcal{C}}^k(X_\alpha,X_\beta) \neq 0$.

\item In case $D$:

In the Auslander-Reiten quiver of $\mathcal{C}$, the $m$-rigid $X_\alpha$ is situated at the bottom as we can see in the next picture. We name the diagonals by $D_{ij}$ in the same way as in $A_n$ case. Both particular diagonals are called $B_1^l$ or $B_1^r$. We draw the figure for the case $i=1$ for the sake of simplicity.

\[ \scalebox{0.6}{\xymatrix{
~ \ar[dr] & & ~ \ar[dr] & & \color{red} B_1^l \ar[dr] & & B_3^r \ar[dr] & & \\
~ \ar[r] & ~ \ar[dr] \ar[ur] \ar[r] & ~\ar[r] & \color{red} D_{1~4m+2} \ar[ur] \ar[dr] \ar[r] & \color{red} B_1^r \ar[r] & \color{red} D_{m+1~1} \ar[dr] \ar[ur] \ar[r] & B_3^l \ar[r] & ~ \ar[dr] \ar[ur] \ar[r] & \\
\cdots \ar[dr] \ar[ur] & & \color{red} D_{1~3m+2} \ar[ur] \ar[dr] & & D_{m+1~4m+2} \ar[ur] \ar[dr] & & \color{red} D_{2m+1~1} \ar[dr] \ar[ur] & & \cdots \\
& \color{red} D_{1~2m+2} \ar[ur] \ar[dr] & & D_{m+1~3m+2} \ar[ur] \ar[dr] & & D_{2m+1~4m+2} \ar[ur] \ar[dr] & & \color{red} D_{3m+1~1} \ar[dr] \ar[ur] & \\
\color{red} X_\alpha=D_{1~m+2} \ar[ur] & & D_{m+1~2m+2} \ar[ur] & & D_{2m+1~3m+2} \ar[ur] & & D_{3m+1~4m+2} \ar[ur] & & \color{red}D_{4m+1~1}
}}
\]

The Hom-hammock starting at $X_\alpha$ contains precisely those $X_\gamma$'s for which $\gamma$ contains vertex $1$. Then $\beta[k]$ belongs to one of them.

It is also known that these objects exactly correspond to the ones which have a nonzero morphism from $X_\alpha$. Then $\mathrm{Ext}_{\mathcal{C}}^k(X_\alpha,X_\beta) \neq 0$.
\end{enumerate}
\end{proof}

\begin{lem}
\cor{Let $P$ be an $nm+2$-gon (respectively an $mn-m+1$-gon) associated with a quiver $Q$ of type $A_n$ (respectively $D_n$). Let $\alpha$ be an $m$-ear. Let $\Gamma$ be the set of $m$-diagonals of $P$. Let $S/\alpha$ be the whole figure where $P$ is replaced by $P'$, with the same sides as $P$ except that the boundary from $i$ to $j$ is replaced by $\alpha$ (then $P'$ becomes a $(n-1)m+2$-gon - respectively an $m(n-1)-m+1$-gon).}
	
Then, the following diagram is commutative:
	
\[\xymatrix{\{\beta \in \Gamma\setminus\{\alpha\}, \beta \text{ does not cross } \alpha\} \ar[r] \ar@{<->}[d] & \{X \in \mathcal{U}; X \ncong X_\alpha\}/\simeq \ar@{<->}[d]\\
	S/\alpha \ar[r] & \mathcal{C}'/\simeq}\]
where $\mathcal{C}'$ is the Iyama-Yoshino reduction of $\mathcal{C}_Q^m$. The horizontal arrows are maps sending $\beta$ to $X_\beta$.
\end{lem}

\begin{proof}
\cor{Let $\beta$ be an $m$-diagonal of $P$ which does not cross $\alpha$. In $S/\alpha$, $\beta$ is a simple $m$-diagonal. As seen before, the higher cluster category corresponding to $S/\alpha$ is the Iyama Yoshino reduction of $\mathcal{C}_Q^m$, which is $\mathcal{C}'$.}
	
\cor{Now, let $X_\beta$ be the $m$-rigid indecomposable object associated with $\beta$. From the previous Lemma, the object $X_\beta$ lies in $\{X \in \mathcal{U}; X \ncong X_\alpha\}$. Now, it suffices to take the quotient under $X_\alpha$, and this is exactly the Iyama Yoshino reduction of $\mathcal{C}_Q^m$, which is $\mathcal{C}'$.}

This shows that the diagram is commutative.
\end{proof}

\begin{rmk}\label{rmk:shift}
\cor{We need to note that the cases are symmetric. Indeed, since the higher cluster category is $(m+1)$-Calabi-Yau, we know that $\mathcal{C}(\beta[k-(m+1)],\alpha) \simeq D\mathcal{C}(\alpha,\beta[k])$. This means that a morphism from $\alpha$ to $\beta[k]$ is in $1-1$-correspondence with a morphism from $\beta[k-(m+1)]$ to $\alpha$. Thus, shifting $\beta$, $k$ times is the same as shifting $\alpha$, $k$ times. This means no matter which vertex we shift.}
\end{rmk}

Let us now prove theorem \ref{th:cross}.

\begin{proof}[Proof of theorem \ref{th:cross}]
\cor{Let us suppose that $\alpha$ and $\beta$ cross each other. If $\alpha$ is an $m$-ear, then the result is already shown in Lemma \ref{lem:cross}.}

\cor{Else, let $i_\alpha$ and $j_\alpha$ (respectively $i_\beta$ and $j_\beta$) be the extremities of $\alpha$ (respectively $\beta$). For sake of clarity, we suppose that we choose $i_\alpha$ and $i_\beta$ such that $|i_\alpha-i_\beta|<|i_\alpha-j_\beta|$ and $|i_\alpha-i_\beta|<|j_\alpha-i_\beta|$ and $|i_\alpha-i_\beta|<|j_\alpha-j_\beta|$.}

\begin{itemize}
\item If $k=|i_\alpha-i_\beta| \leq m$, then the $m$-diagonal $\beta[k]$ ends in $i_\alpha$. We have seen in proof of Lemma \ref{lem:cross} that this means that $\beta[k]$ is situated on the Hom-hammock of $X_\alpha$. Then $\mathrm{Ext}_{\mathcal{C}}^k(X_\alpha,X_\beta) \neq 0$.

Note that this case happens whatever the arcs when $n=4$.
\item Else, we have $|i_\alpha-i_\beta| \geq m+1$. We proceed by induction. The case $n=4$ is already treated.

\cor{Suppose that the result is shown for a given $n$. Then we can draw $\gamma$, an $m$-ear from $i_\alpha$ to $i_\beta$ which does not cross neither $\alpha$ nor $\beta$. We decide to cut along this arc $\gamma$, as we did in lemma \ref{lem:cutalong}. We are now, from the previous Lemma, in a case of size $n-1$, and we can apply the induction hypothesis: there exists some $k \in \{1,\cdots,m\}$ such that $\mathrm{Ext}_{\mathcal{C'}}^k(X_\alpha,X_\beta) \neq 0$. From Iyama-Yoshino, we have that $\mathcal{C}(X,Y[i])$ and $\mathcal{C'}(X,Y\langle i \rangle)$ are isomorphic. This finishes the proof of the theorem.}
\end{itemize}
\end{proof}

\subsection{Case $\tilde{A}$}

We set the figure for the geometric realization of type $\tilde{A}$. We let $p$ and $q$ be two integers. We are going to show the result, which will be available for any $p$ and $q$. Let $O$ be the outer polygon with $mp$ sides, and $I$ be the inner polygon with $mq$ sides. In this subsection, we will use the same sketch of proof. Let us now define the notion of an $m$-ear:

\begin{defi}
Let $\alpha$ be an $m$-diagonal. Then $\alpha$ is an $m$-ear if it lies in the outer or inner polygon, and links a vertex $i$ to $i+m+1$, and is homotopic to the boundary path (see figure \ref{fig:mearatilde} for an example of $m$-ear).
\end{defi}

\begin{figure}[!h]
\centering
\begin{tikzpicture}[scale=0.9]
\fill[fill=black,fill opacity=0.1] (3.8,4.86) -- (2.62,4.2) -- (1.93,3.04) -- (1.91,1.69) -- (2.57,0.51) -- (3.73,-0.19) -- (5.08,-0.2) -- (6.26,0.46) -- (6.96,1.62) -- (6.97,2.97) -- (6.31,4.15) -- (5.15,4.84) -- cycle;
\fill[fill=black,fill opacity=0.1] (4.96,3.18) -- (3.9,3.18) -- (3.37,2.26) -- (3.9,1.34) -- (4.96,1.34) -- (5.49,2.26) -- cycle;
\draw (3.8,4.86)-- (2.62,4.2);
\draw (2.62,4.2)-- (1.93,3.04);
\draw (1.93,3.04)-- (1.91,1.69);
\draw (1.91,1.69)-- (2.57,0.51);
\draw (2.57,0.51)-- (3.73,-0.19);
\draw (3.73,-0.19)-- (5.08,-0.2);
\draw (5.08,-0.2)-- (6.26,0.46);
\draw (6.26,0.46)-- (6.96,1.62);
\draw (6.96,1.62)-- (6.97,2.97);
\draw (6.97,2.97)-- (6.31,4.15);
\draw (6.31,4.15)-- (5.15,4.84);
\draw (5.15,4.84)-- (3.8,4.86);
\draw (5.08,-0.2)-- (6.97,2.97);
\draw (4.96,3.18)-- (3.9,3.18);
\draw (3.9,3.18)-- (3.37,2.26);
\draw (3.37,2.26)-- (3.9,1.34);
\draw (3.9,1.34)-- (4.96,1.34);
\draw (4.96,1.34)-- (5.49,2.26);
\draw (5.49,2.26)-- (4.96,3.18);
\begin{scriptsize}
\draw[color=black] (6,1) node {$\alpha$};
\end{scriptsize}
\end{tikzpicture}
\caption{Example of an $m$-ear in case $\tilde{A}$}
\label{fig:mearatilde}
\end{figure}

\cor{Before setting the next Lemma, we recall that the transjective component of the Auslander-Reiten quiver of a higher cluster category exactly contains the injective and projective components of the Auslander-Reiten quiver, which means, every object except those from the tubes.}

\begin{lem}\label{lem:atilde}
\cor{Let $\mathcal{C}$ be the $m$-cluster category of type $\tilde{A}_n$. Let $\alpha$ be an $m$-diagonal which is either an $m$-ear or links a vertex of the outer polygon to a vertex of the inner polygon (this means that the $m$-rigid indecomposable object $X_\alpha$ lies in the transjective component of the Auslander-Reiten quiver of $\mathcal{C}$). Let \[\mathcal{U}=\{ Y \in \mathcal{C},~\forall i \in \{1,\cdots,m\},~\mathrm{Ext}_\mathcal{C}^i(X_\alpha,Y)=0\}\] be a subcategory of $\mathcal{C}$. Let $\mathcal{C'}$ be the Iyama-Yoshino reduction of $\mathcal{C}$ defined in Theorem \ref{th:iy}: $\mathcal{C'}=\mathcal{U}/(X_\alpha)$. Then, we have an equivalence of categories :
\[ \mathcal{C'} \simeq \mathcal{C}^{m}_{Q/\alpha} \]
where $Q/\alpha$ is the quiver obtained from $Q$ by removing $\alpha$ and all incident arrows.} 
\end{lem}

\begin{rmk}
\cor{We can illustrate by an example our strategy:}

We focus on Gabriel quivers. The mutation at vertex $1$ leads to the following quiver:

\[ \xymatrix@1{
	& 1 \ar[dl] & 3 \ar[r] \ar[l] & 5 \ar[dr] & \\
	0 \ar[urr] \ar[dr] & & & & 7 \\
	& 2 \ar[r] & 4 \ar[r] & 6 \ar[ur] & } \]

Using the Iyama-Yoshino reduction at vertex $1$ corresponds to forgetting this vertex and all incident arrows. By doing this, we are ensured to be reduced to a quiver of type $\tilde{A}_{n-1}$:

\[ \xymatrix@1{
	& & 3 \ar[r] & 5 \ar[dr] & \\
	0 \ar[urr] \ar[dr] & & & & 7 \\
	& 2 \ar[r] & 4 \ar[r] & 6 \ar[ur] & } \]
\end{rmk}

\begin{proof}
Let us start by recalling the distribution of simple module in the Auslander-Reiten quiver of a higher cluster category of type $\tilde{A}$.

\cor{The Auslander Reiten quiver of $\mathcal{C}_{\tilde{A}_n}$ is made of a transjective part (which contains the injective part and projective part), plus homogeneous tubes (whoch do not focus on) and two tubes, one of size $p$, one of size $q$.}

\cor{We note that the $m$-diagonals linking the outer polygon to the inner polygon correspond to objects in the transjective part of the Auslander-Reiten quiver.}

\cor{Moreover, the arcs linking two vertices of the outer polygon correspond to objects in the tube $\mathcal{T}_p$ of size $p$, whereas the arcs linking to vertices of the inner polygon correspond to objects in the tube $\mathcal{T}_q$ of size $q$.}

\cor{The simple objects figure at the bottom of the tubes $\mathcal{T}_p$ and $\mathcal{T}_q$.}

We can now start the proof of Lemma \ref{lem:atilde}.

There are two different cases.

First, if $\alpha$ is an $m$-ear:

Here again, we use the theorem of Keller and Reiten in \cite{KR}, as in type $A$ and $D$. We have to find an $m$-cluster-tilting object $T$ such that $\mathrm{End}_{\mathcal{C'}}(T) \simeq K \tilde{A}_{n-1}$, and $\forall i \in \{1,\cdots,m\} \mathrm{Ext}_\mathcal{C'}^{-i}(T,T)=0$.

\cor{Let $T=\oplus P_i$. We know from Torkildsen (see figure 8 in \cite{Tor}), that $T$ corresponds summand by summand, to the initial $(m+2)$-angulation. This one is made of $q+1$ $m$-diagonals linking the vertex $1$ of the external polygon to the vertices $1+km$, for $k \in \{0,\cdots,q-1\}$ plus an $m$-diagonal linking the vertex $1$ of the outer polygon, and the vertex $1$ of the inner polygon, but going around it. To complete the $(m+2)$-angulation, we add $p-1$ $m$-diagonals linking the vertex $1$ of the inner polygon $T$ to vertices of the outer polygon $P$, numbered $1+km$, $k \in \{1,\cdots,p-1\}$.}

\begin{figure}[!h]
\centering
\begin{tikzpicture}[scale=0.9]
\fill[fill=black,fill opacity=0.1] (3.8,4.86) -- (2.62,4.2) -- (1.93,3.04) -- (1.91,1.69) -- (2.57,0.51) -- (3.73,-0.19) -- (5.08,-0.2) -- (6.26,0.46) -- (6.96,1.62) -- (6.97,2.97) -- (6.31,4.15) -- (5.15,4.84) -- cycle;
\fill[fill=black,fill opacity=0.1] (4.96,3.18) -- (3.9,3.18) -- (3.37,2.26) -- (3.9,1.34) -- (4.96,1.34) -- (5.49,2.26) -- cycle;
\draw (3.8,4.86)-- (2.62,4.2);
\draw (2.62,4.2)-- (1.93,3.04);
\draw (1.93,3.04)-- (1.91,1.69);
\draw (1.91,1.69)-- (2.57,0.51);
\draw (2.57,0.51)-- (3.73,-0.19);
\draw (3.73,-0.19)-- (5.08,-0.2);
\draw (5.08,-0.2)-- (6.26,0.46);
\draw (6.26,0.46)-- (6.96,1.62);
\draw (6.96,1.62)-- (6.97,2.97);
\draw (6.97,2.97)-- (6.31,4.15);
\draw (6.31,4.15)-- (5.15,4.84);
\draw (5.15,4.84)-- (3.8,4.86);
\draw (4.96,3.18)-- (3.9,3.18);
\draw (3.9,3.18)-- (3.37,2.26);
\draw (3.37,2.26)-- (3.9,1.34);
\draw (3.9,1.34)-- (4.96,1.34);
\draw (4.96,1.34)-- (5.49,2.26);
\draw (5.49,2.26)-- (4.96,3.18);
\draw (6.26,0.46)-- (4.96,1.34);
\draw (4.96,1.34)-- (3.73,-0.19);
\draw [shift={(3.54,2.43)}] plot[domain=3.57:5.63,variable=\t]({1*1.79*cos(\t r)+0*1.79*sin(\t r)},{0*1.79*cos(\t r)+1*1.79*sin(\t r)});
\draw [shift={(5.14,4.44)}] plot[domain=-1.16:0.84,variable=\t]({-0.56*4*cos(\t r)+0.83*2.13*sin(\t r)},{-0.83*4*cos(\t r)+-0.56*2.13*sin(\t r)});
\draw [shift={(4.41,2.97)}] plot[domain=-2.59:1.24,variable=\t]({-0.73*2.13*cos(\t r)+0.68*1.66*sin(\t r)},{-0.68*2.13*cos(\t r)+-0.73*1.66*sin(\t r)});
\draw [shift={(4.31,2.66)}] plot[domain=-2.73:1.13,variable=\t]({-0.59*1.58*cos(\t r)+0.8*1.44*sin(\t r)},{-0.8*1.58*cos(\t r)+-0.59*1.44*sin(\t r)});
\draw [shift={(3.58,-0.75)}] plot[domain=0.83:1.35,variable=\t]({1*5.03*cos(\t r)+0*5.03*sin(\t r)},{0*5.03*cos(\t r)+1*5.03*sin(\t r)});
\draw [shift={(4.51,2.57)}] plot[domain=-2.46:1.54,variable=\t]({-0.93*1.48*cos(\t r)+0.37*1.31*sin(\t r)},{-0.37*1.48*cos(\t r)+-0.93*1.31*sin(\t r)});
\draw [shift={(4.64,1.77)}] plot[domain=-0.68:1.26,variable=\t]({1*2.09*cos(\t r)+0*2.09*sin(\t r)},{0*2.09*cos(\t r)+1*2.09*sin(\t r)});
\draw [shift={(2.99,0.56)}] plot[domain=-0.03:0.93,variable=\t]({1*3.27*cos(\t r)+0*3.27*sin(\t r)},{0*3.27*cos(\t r)+1*3.27*sin(\t r)});
\draw [shift={(4.88,1.57)}] plot[domain=-2.27:1.1,variable=\t]({-0.17*2.07*cos(\t r)+-0.99*1.54*sin(\t r)},{0.99*2.07*cos(\t r)+-0.17*1.54*sin(\t r)});
\end{tikzpicture}
\caption{The "initial" $(m+2)$-angulation of type $\tilde{A}$, for $m=2$, $p=6$, $q=3$.}
\label{fig:manginitatilde}
\end{figure}

Let $T'$ be the mutation of $T$ at $P_1$ the first preprojective module. Then \[T'=\oplus_{i \neq 1} P_i \oplus X\] is also an $m$-cluster-tilting object. Let us show that $\tau X$ corresponds to the simple module at the base of the tube of size $p$ (see figure \ref{fig:mcto} to visualize the mutation in terms of arcs). However, we do not know that the mutation of $m$-cluster tilting objects corresponds to the flip of $(m+2)$-angulations yet.

\begin{figure}[!h]
\centering
\begin{tikzpicture}[scale=0.7]
\fill[fill=black,fill opacity=0.1] (3.8,4.86) -- (2.62,4.2) -- (1.93,3.04) -- (1.91,1.69) -- (2.57,0.51) -- (3.73,-0.19) -- (5.08,-0.2) -- (6.26,0.46) -- (6.96,1.62) -- (6.97,2.97) -- (6.31,4.15) -- (5.15,4.84) -- cycle;
\fill[fill=black,fill opacity=0.1] (4.96,3.18) -- (3.9,3.18) -- (3.37,2.26) -- (3.9,1.34) -- (4.96,1.34) -- (5.49,2.26) -- cycle;
\draw (3.8,4.86)-- (2.62,4.2);
\draw (2.62,4.2)-- (1.93,3.04);
\draw (1.93,3.04)-- (1.91,1.69);
\draw (1.91,1.69)-- (2.57,0.51);
\draw (2.57,0.51)-- (3.73,-0.19);
\draw (3.73,-0.19)-- (5.08,-0.2);
\draw (5.08,-0.2)-- (6.26,0.46);
\draw (6.26,0.46)-- (6.96,1.62);
\draw (6.96,1.62)-- (6.97,2.97);
\draw (6.97,2.97)-- (6.31,4.15);
\draw (6.31,4.15)-- (5.15,4.84);
\draw (5.15,4.84)-- (3.8,4.86);
\draw (4.96,3.18)-- (3.9,3.18);
\draw (3.9,3.18)-- (3.37,2.26);
\draw (3.37,2.26)-- (3.9,1.34);
\draw (3.9,1.34)-- (4.96,1.34);
\draw (4.96,1.34)-- (5.49,2.26);
\draw (5.49,2.26)-- (4.96,3.18);
\draw (6.26,0.46)-- (4.96,1.34);
\draw (4.96,1.34)-- (3.73,-0.19);
\draw [shift={(3.54,2.43)}, color=red] plot[domain=3.57:5.63,variable=\t]({1*1.79*cos(\t r)+0*1.79*sin(\t r)},{0*1.79*cos(\t r)+1*1.79*sin(\t r)});
\draw [shift={(5.14,4.44)}] plot[domain=-1.16:0.84,variable=\t]({-0.56*4*cos(\t r)+0.83*2.13*sin(\t r)},{-0.83*4*cos(\t r)+-0.56*2.13*sin(\t r)});
\draw [shift={(4.41,2.97)}] plot[domain=-2.59:1.24,variable=\t]({-0.73*2.13*cos(\t r)+0.68*1.66*sin(\t r)},{-0.68*2.13*cos(\t r)+-0.73*1.66*sin(\t r)});
\draw [shift={(4.31,2.66)}] plot[domain=-2.73:1.13,variable=\t]({-0.59*1.58*cos(\t r)+0.8*1.44*sin(\t r)},{-0.8*1.58*cos(\t r)+-0.59*1.44*sin(\t r)});
\draw [shift={(3.58,-0.75)}] plot[domain=0.83:1.35,variable=\t]({1*5.03*cos(\t r)+0*5.03*sin(\t r)},{0*5.03*cos(\t r)+1*5.03*sin(\t r)});
\draw [shift={(4.51,2.57)}] plot[domain=-2.46:1.54,variable=\t]({-0.93*1.48*cos(\t r)+0.37*1.31*sin(\t r)},{-0.37*1.48*cos(\t r)+-0.93*1.31*sin(\t r)});
\draw [shift={(4.64,1.77)}] plot[domain=-0.68:1.26,variable=\t]({1*2.09*cos(\t r)+0*2.09*sin(\t r)},{0*2.09*cos(\t r)+1*2.09*sin(\t r)});
\draw [shift={(2.99,0.56)}] plot[domain=-0.03:0.93,variable=\t]({1*3.27*cos(\t r)+0*3.27*sin(\t r)},{0*3.27*cos(\t r)+1*3.27*sin(\t r)});
\draw [shift={(4.88,1.57)}] plot[domain=-2.27:1.1,variable=\t]({-0.17*2.07*cos(\t r)+-0.99*1.54*sin(\t r)},{0.99*2.07*cos(\t r)+-0.17*1.54*sin(\t r)});
\end{tikzpicture}
$\longrightarrow$
\begin{tikzpicture}[scale=0.7]
\fill[fill=black,fill opacity=0.1] (3.8,4.86) -- (2.62,4.2) -- (1.93,3.04) -- (1.91,1.69) -- (2.57,0.51) -- (3.73,-0.19) -- (5.08,-0.2) -- (6.26,0.46) -- (6.96,1.62) -- (6.97,2.97) -- (6.31,4.15) -- (5.15,4.84) -- cycle;
\fill[fill=black,fill opacity=0.1] (4.96,3.18) -- (3.9,3.18) -- (3.37,2.26) -- (3.9,1.34) -- (4.96,1.34) -- (5.49,2.26) -- cycle;
\draw (3.8,4.86)-- (2.62,4.2);
\draw (2.62,4.2)-- (1.93,3.04);
\draw (1.93,3.04)-- (1.91,1.69);
\draw (1.91,1.69)-- (2.57,0.51);
\draw (2.57,0.51)-- (3.73,-0.19);
\draw (3.73,-0.19)-- (5.08,-0.2);
\draw (5.08,-0.2)-- (6.26,0.46);
\draw (6.26,0.46)-- (6.96,1.62);
\draw (6.96,1.62)-- (6.97,2.97);
\draw (6.97,2.97)-- (6.31,4.15);
\draw (6.31,4.15)-- (5.15,4.84);
\draw (5.15,4.84)-- (3.8,4.86);
\draw (4.96,3.18)-- (3.9,3.18);
\draw (3.9,3.18)-- (3.37,2.26);
\draw (3.37,2.26)-- (3.9,1.34);
\draw (3.9,1.34)-- (4.96,1.34);
\draw (4.96,1.34)-- (5.49,2.26);
\draw (5.49,2.26)-- (4.96,3.18);
\draw (6.26,0.46)-- (4.96,1.34);
\draw (4.96,1.34)-- (3.73,-0.19);
\draw [shift={(5.14,4.44)}] plot[domain=-1.16:0.84,variable=\t]({-0.56*4*cos(\t r)+0.83*2.13*sin(\t r)},{-0.83*4*cos(\t r)+-0.56*2.13*sin(\t r)});
\draw [shift={(4.41,2.97)}] plot[domain=-2.59:1.24,variable=\t]({-0.73*2.13*cos(\t r)+0.68*1.66*sin(\t r)},{-0.68*2.13*cos(\t r)+-0.73*1.66*sin(\t r)});
\draw [shift={(4.31,2.66)}] plot[domain=-2.73:1.13,variable=\t]({-0.59*1.58*cos(\t r)+0.8*1.44*sin(\t r)},{-0.8*1.58*cos(\t r)+-0.59*1.44*sin(\t r)});
\draw [shift={(3.58,-0.75)}] plot[domain=0.83:1.35,variable=\t]({1*5.03*cos(\t r)+0*5.03*sin(\t r)},{0*5.03*cos(\t r)+1*5.03*sin(\t r)});
\draw [shift={(4.51,2.57)}] plot[domain=-2.46:1.54,variable=\t]({-0.93*1.48*cos(\t r)+0.37*1.31*sin(\t r)},{-0.37*1.48*cos(\t r)+-0.93*1.31*sin(\t r)});
\draw [shift={(4.64,1.77)}] plot[domain=-0.68:1.26,variable=\t]({1*2.09*cos(\t r)+0*2.09*sin(\t r)},{0*2.09*cos(\t r)+1*2.09*sin(\t r)});
\draw [shift={(2.99,0.56)}] plot[domain=-0.03:0.93,variable=\t]({1*3.27*cos(\t r)+0*3.27*sin(\t r)},{0*3.27*cos(\t r)+1*3.27*sin(\t r)});
\draw [shift={(4.88,1.57)}] plot[domain=-2.27:1.1,variable=\t]({-0.17*2.07*cos(\t r)+-0.99*1.54*sin(\t r)},{0.99*2.07*cos(\t r)+-0.17*1.54*sin(\t r)});
\draw [shift={(5.13,2.72)},color=red]  plot[domain=3.04:4.26,variable=\t]({1*3.22*cos(\t r)+0*3.22*sin(\t r)},{0*3.22*cos(\t r)+1*3.22*sin(\t r)});
\end{tikzpicture}
\caption{Geometric visualization of $T'$}
\label{fig:mcto}
\end{figure}

We have to show that $\tau X=S_1$, the simple module at vertex $1$, which is situated at the bottom of the tube of size $n-2$.

Let us find $\mathcal{C}(X[-1],T)$. For all $i \neq 1$, $\mathcal{C}(X[-1],P_i) \simeq \mathcal{C}(X,P_i[1])=0$ since $T'$ is an $m$-cluster-tilting object.

Now we focus on $\mathcal{C}(X[-1],P_1)$.
We have that $\mathcal{C}(X[-1],P_1) \simeq \mathcal{C}(X,P_1[1]) \simeq K$ from Iyama and Yoshino at \cor{Corollary 6.4} in \cite{IY}.

\cor{From the Auslander-Reiten duality, we have:}

\[ \mathcal{C}(X[-1],T) \simeq \mathcal{C}(X,T[1]) \simeq D\mathcal{C}(T,\tau X) \]

Then $\tau X$ is an object at the bottom of the tube $\mathcal{T}_p$, situated on the left. It exactly corresponds to $\simeq S_1$. Moreover, $X=X_\alpha$ corresponds to the $m$-ear starting at vertex $1$ (which corresponds to the red arc), named $\alpha$. Indeed, if we apply a translation to $\alpha$, which corresponds to shift it $m$ times, as the higher cluster category is $(m+1)$-Calabi-Yau, we find back the arc corresponding to the simple object $S_1$.

Furthermore, from the paper of Baur and Torkildsen \cite{BauTor}, we can easily visualize the morphisms in the module category of type $\tilde{A}$.

We have $\mathrm{End}_{\mathcal{C'}}(T') \simeq K\tilde{A}_{n-1}$. Indeed, in the module category, we have $\mathrm{End}_\mathrm{mod}(T') \simeq K\tilde{A}_{n-1}$, because, the objects of $T'$ (apart from $X_\alpha$) are on the projective slice of the Auslander-Reiten quiver of $Q$.

Now we show this for the higher cluster category. We have the following decomposition of morphisms ($G$ is the functor $\tau^{-1}[m]$):
\[ \mathcal{C}(M,N) \simeq \bigoplus_{i \in \mathbb{Z}} \mathcal{D}^b(G^iM,N). \]

If $m=1$, the result is already known, because if $X$ is a preprojective object and $Y$ a regular one, then $\mathrm{Ext}^1_\mathcal{C}(X,Y)=0$.

If not, we use the decomposition just above. We have that \[\mathcal{D}^b(\tau^{-1}[m]T,T)=D\mathcal{D}^b(T,T[m+1])\] thanks to the Auslander-Reiten duality. From the book \cite{ASS}, the algebra of a quiver of type $\tilde{A}$ is hereditary and then the extension $\mathrm{Ext}^2_{\mathcal{D}^b}(T,T)=0$. Then, for $i \geq 1$, all the terms $\mathcal{D}^b(G^iM,N)$ of the sum are zero. Then \[\mathrm{End}_{\mathcal{C}}(T') \simeq \mathrm{End}_{\mathcal{D}^b}(T').\]

\cor{As no morphism is incident to $\alpha$, we can apply the Iyama-Yoshino reduction at object $X_\alpha$ and the results still hold.}

It finally remains to prove that $\mathrm{Ext}_\mathcal{C'}^{-i}(T',T')=0$ for all $i \in \{1,m-1 \}$.

Let us first show that $\mathrm{Ext}_\mathcal{C}^{-i}(T',T')=0$, using the shift in $\mathcal{C}$. Then we will use lemma \ref{lem:shift} in order to conclude.

\cor{We claim that $\mathcal{C}(T',T'[-i])=0$. Indeed, from Keller and Reiten at Lemma 4.1 in \cite{KR}, we know that $\mathcal{C}(\oplus_{j \neq k}P_j,T'[-i])=0$. Moreover, $\mathcal{C}(X_\alpha,X_\alpha[-i])=0$ in the module category (which from Wraalsen \cite{W} or Zhou-Zhu \cite{ZZ}, immediately translates to the higher cluster category). In addition, there cannot be any morphism from a tubular component to a preprojective one. Then $\mathcal{C}(X_\alpha[-i],P_j)=0$ in the module category (which from Wraalsen \cite{W} or Zhou-Zhu \cite{ZZ}, immediately translates to the higher cluster category). It remains to show that $\mathcal{C}(P_j,X_\alpha[-i])=0$ for any $j \neq k$ and $i \in \{ 1,\cdots,m-1 \}$.}
	


For $i \neq 1$, there is no morphism from $P_j$ to $P_k[-i+1]$.

Then, there exists $g:P_j \to U[-i]$. As $U$ is only composed with projectives which are not $P_k$, this shows that $g=0$. Then $f=0$.


For $i=1$, the composition $P_j \to X_\alpha[-1] \to P_k$ is zero because there is no morphism from tubular objects to preprojective objects. Then there exists $g$ such as previously, but the composition with $U[-i] \to X_\alpha[-i]$ is zero for the same reason. Then $f=0$.

This shows that $\mathrm{Ext}_\mathcal{C}^{-i}(T',T')=0$.

%
%
%
%
%

From lemma \ref{lem:shift}, we have that $T'_l<-i> \simeq T'_l[-i]$. Finally,
\[ \mathrm{Ext}_\mathcal{C'}^{-i}(T',T')=0. \]

We now have gathered all the information in order to apply the theorem of Keller and Reiten, and we have that \[ \mathcal{C'} \simeq \mathcal{C}^{m}_{Q/\alpha}. \]

Now, if $\alpha$ corresponds to a transjective module, say that $\alpha$ links vertex $1$ of $P$ to vertex $1$ of $T$. We proceed the same way, we have that $\mathrm{Ext}^{-i}_\mathcal{C}(T',T')=0$ from Keller and Reiten \cite[Lemma 4.1]{KR}, and we can apply Keller-Reiten theorem. In details, let $T$ be the $m$-cluster-tilting object corresponding to a slice of the Auslander-Reiten quiver of $Q$ (see the article of Baur and Torkildsen \cite{BauTor} for details). Furthermore, Baur and Torkildsen in \cite[Proposition 3.7]{BauTor} showed that there was an isomorphism between the Auslander-Reiten quiver of $Q$ (except the homogeneous tubes) and the translation quiver $\Gamma$ built by themselves at paragraph 3.5. In this quiver Gamma, the $m$-diagonals compose the vertices and they are linked by elementary moves. Now it suffices to take the same object $T$ as before (the $m$-cluster-tilting object corresponding to the initial $(m+2)$-angulation), to change the object $X_\alpha \in T$ into its mutation exactly the same way we did before, and we turn to have \[\mathrm{End}_{\mathrm{mod}}(T) \simeq KA_{n-1}.\]
\end{proof}

\begin{rmk}
\cor{We set the annulus with an outer polygon with $mp$ sides, and an inner polygon with $mq$ sides, associated with a quiver $Q$ of type $\tilde{A}_n$. Let $\alpha$ be an $m$-ear from $i$ to $j$. Then cutting along $\alpha$ corresponds to applying the Iyama-Yoshino reduction of $\mathcal{C}^m_{\tilde{A}_n}$ on $X_\alpha$. To be precise, we call by $S/\alpha$ the whole figure where $P$ is replaced by $P'$, with the same sides as $P$ except that the boundary from $i$ to $j$ is replaced by $\alpha$ (then $P'$ becomes a $(n-1)m+2$-gon - respectively an $m(n-1)-m+1$-gon). Then the Iyama Yoshino reduction of $\mathcal{C}^m_{\tilde{A}_n}$ on $X_\alpha$ is equivalent to the higher cluster category $\mathcal{C}^m_{\tilde{A}_{n-1}}$.}
\end{rmk}

\begin{lem}\label{lem:crossatilde}
\cor{Let $P$ be an annulus associated with a quiver $Q$ of type $\tilde{A}_n$. Let $\alpha$ be an $m$-ear. Let $\beta$ be an $m$-diagonal which cuts $\alpha$ (see figure \ref{fig:mear}). Let $X_\beta$ be the associated $m$-rigid object. Then there exists $k \in \{1,\cdots,m\}$ such that $\mathrm{Ext}_\mathcal{C}^k(X_\alpha,X_\beta) \neq 0$.}
\end{lem}


\begin{rmk}
\cor{We need to note that the cases are symmetric. Indeed, since the higher cluster category is $(m+1)$-Calabi-Yau, we know that $\mathcal{C}(\beta[k-(m+1)],\alpha) \simeq D\mathcal{C}(\alpha,\beta[k])$. This means that a morphism from $\alpha$ to $\beta[k]$ is in $1-1$-correspondence with a morphism from $\beta[k-(m+1)]$ to $\alpha$. Thus, shifting $\beta$, $k$ times is the same as shifting $\alpha$, $k$ times. This means no matter which vertex we shift.}
\end{rmk}

\begin{proof}
\cor{Let $X_\alpha$ be the $m$-rigid indecomposable object associated with $\alpha$. By the geometric realization of Torkildsen in \cite{Tor}, the $m$-diagonal $\alpha$ is situated at the bottom of one tube. Without any loss of generality, we suppose that $\alpha$ links $i$ and $i+m$ in the outer polygon, and that its corresponding object appears in one tube of size $p$.}

\cor{Let $\beta$ be an $m$-diagonal crossing $\alpha$. It means that there exists $k \in \{1,\cdots,m-1\}$, such that an extremity of $\beta$ is the vertex $i+k$ in the outer polygon. Therefore, it suffices to shift $\beta$ $k<m$ times, so that one extremity of $\beta[k]$ is $i$, common to $\alpha$. There are two cases:}

First case: $\beta$ corresponds to an $m$-rigid indecomposable object $X_\beta$ in a tube. Then, by the proof of Proposition 7.2 in \cite{Tor}, there exists a nonzero morphism from $X_\alpha$ to $X_\beta[k]$ (see figure 13 of the article for a clear picture of this map).

Second case: The object $X_\beta$ is a preinjective indecomposable object. Then, by the paragraph 4.1 of the article written by Baur and Torkildsen \cite{BauTor}, as $\alpha$ and $\beta[k]$ share an oriented angle, there is a so-called "long move", hence a nonzero morphism in the module category from $X_\alpha$ to $X_\beta[k]$.

\cor{In any case, we have \[\mathrm{Ext}_\mathcal{C}^k(X_\alpha,X_\beta) \neq 0.\]
Indeed, we have found a nonzero morphism in the module category, then in the higher cluster category from $X_\alpha$ to $X_\beta[k]$.}
\end{proof}

Then the arcs which cross $\alpha$ exactly correspond to the $m$-rigid indecomposable objects which do not lie in \[\mathcal{U}=\{ Y \in \mathcal{C},~\forall i \in \{1,\cdots,m\},~\mathrm{Ext}_\mathcal{C}^i(X_\alpha,Y)=0\}.\]

We are now able to prove theorem \ref{th:cross}:

\begin{proof}[Proof of theorem \ref{th:cross}]
If $\alpha$ and $\beta$ are two crossing $m$-diagonals in the geometric realization of a quiver of type $\tilde{A}$ (an external polygon $P$ with $p$ sides together with an internal polygon $T$ with $q$ sides). There are two cases:

\begin{enumerate}
\item First case: The $m$-diagonal $\alpha$ links two vertices $i$ and $j$ of $P$ (or $T$), and is homotopic to the boundary path. If $\alpha$ is an $m$-ear, then the result is shown. \cor{Else, let $i_\alpha$ and $j_\alpha$ (respectively $i_\beta$ and $j_\beta$) be the extremities of $\alpha$ (respectively $\beta$). For sake of clarity, we suppose that we choose $i_\alpha$ and $i_\beta$ such that $|i_\alpha-i_\beta|<|i_\alpha-j_\beta|$ and $|i_\alpha-i_\beta|<|j_\alpha-i_\beta|$ and $|i_\alpha-i_\beta|<|j_\alpha-j_\beta|$.}
\begin{itemize}
\item If $k=|i_\alpha-i_\beta| \leq m$, then the $m$-diagonal $\beta[k]$ ends in $i_\alpha$. We have seen in proof of Lemma \ref{lem:crossatilde} that this means that $\mathrm{Ext}_{\mathcal{C}}^k(X_\alpha,X_\beta) \neq 0$.

\item Else, we have $|i_\alpha-i_\beta| \geq m+1$. We proceed by induction. The case $n=4$ is easy to treat.

\cor{Suppose that the result is shown for a given $n$. Then we can draw $\gamma$, an $m$-ear from $i_\alpha$ to $i_\beta$ which does not cross neither $\alpha$ nor $\beta$. We decide to cut along this arc $\gamma$, as we did in lemma \ref{lem:cutalong}. We are now, from the previous Lemma, in a case of size $n-1$, and we can apply the induction hypothesis: there exists some $k \in \{1,\cdots,m\}$ such that $\mathrm{Ext}_{\mathcal{C'}}^k(X_\alpha,X_\beta) \neq 0$. From Iyama-Yoshino, we have that $\mathcal{C}(X,Y[i])$ and $\mathcal{C'}(X,Y\langle i \rangle)$ are isomorphic. This finishes the proof of the theorem.}

Else we can draw an $m$-ear $\gamma$ between an extremity of $\alpha$ and an extremity of $\beta$, then it suffices to cut along $\gamma$ and repeat the operation as many times as necessary, in order to reduce to the previous case.
\end{itemize}
See figure \ref{fig:ext1} for an illustration.

\begin{figure}[!h]
\centering
\begin{tikzpicture}[scale=0.9]
\fill[fill=black,fill opacity=0.1] (3.8,4.86) -- (2.62,4.2) -- (1.93,3.04) -- (1.91,1.69) -- (2.57,0.51) -- (3.73,-0.19) -- (5.08,-0.2) -- (6.26,0.46) -- (6.96,1.62) -- (6.97,2.97) -- (6.31,4.15) -- (5.15,4.84) -- cycle;
\fill[fill=black,fill opacity=0.1] (4.96,3.18) -- (3.9,3.18) -- (3.37,2.26) -- (3.9,1.34) -- (4.96,1.34) -- (5.49,2.26) -- cycle;
\draw (3.8,4.86)-- (2.62,4.2);
\draw (2.62,4.2)-- (1.93,3.04);
\draw (1.93,3.04)-- (1.91,1.69);
\draw (1.91,1.69)-- (2.57,0.51);
\draw (2.57,0.51)-- (3.73,-0.19);
\draw (3.73,-0.19)-- (5.08,-0.2);
\draw (5.08,-0.2)-- (6.26,0.46);
\draw (6.26,0.46)-- (6.96,1.62);
\draw (6.96,1.62)-- (6.97,2.97);
\draw (6.97,2.97)-- (6.31,4.15);
\draw (6.31,4.15)-- (5.15,4.84);
\draw (5.15,4.84)-- (3.8,4.86);
\draw (4.96,3.18)-- (3.9,3.18);
\draw (3.9,3.18)-- (3.37,2.26);
\draw (3.37,2.26)-- (3.9,1.34);
\draw (3.9,1.34)-- (4.96,1.34);
\draw (4.96,1.34)-- (5.49,2.26);
\draw (5.49,2.26)-- (4.96,3.18);
\draw [shift={(3.94,1.48)}] plot[domain=0.05:2.02,variable=\t]({1*3.02*cos(\t r)+0*3.02*sin(\t r)},{0*3.02*cos(\t r)+1*3.02*sin(\t r)});
\draw [shift={(4.66,3.75)}] plot[domain=-1.06:0.24,variable=\t]({1*1.71*cos(\t r)+0*1.71*sin(\t r)},{0*1.71*cos(\t r)+1*1.71*sin(\t r)});
\draw [shift={(4.42,1.05)},dash pattern=on 5pt off 5pt,color=red]  plot[domain=1.02:2.09,variable=\t]({1*3.63*cos(\t r)+0*3.63*sin(\t r)},{0*3.63*cos(\t r)+1*3.63*sin(\t r)});
\begin{scriptsize}
\draw[color=black] (5.76,3.5) node {$\alpha$};
\draw[color=black] (6.26,2.54) node {$\beta$};
\draw[color=red] (5.48,4.3) node {$\gamma$};
\end{scriptsize}
\end{tikzpicture}
\caption{Illustration of the first case for $p=6$, $q=3$ and $m=2$}
\label{fig:ext1}
\end{figure}

\item Second case: The $m$-diagonal $\alpha$ corresponds to an object $X_\alpha$ which lies in the transjective part of the Auslander-Reiten quiver of $\mathcal{C}^{m}_{\tilde{A_n}}$.

\begin{enumerate}
\item If $\beta$ is homotopic to the boundary path of one of the polygons (let us say for instance that $\beta$ is homotopic to the boundary path of the external polygon). Then, we use the same type of argument.
\begin{itemize}
\item \cor{If $k=|i_\alpha-i_\beta| \leq m$, then it suffices to shift $\beta$ $k<m$ times in order to hang one extremity of $\beta$ to one extremity of $\alpha$. This corresponds to a long move, then to a morphism in the module category in the sense of Baur and Torkildsen in \cite{BauTor}.}
\item \cor{Else, there exists an $m$-ear which dos note cross neither $\alpha$ nor $\beta$. We cut along this $m$-ear, and repeat the operation as many times as necessary to reduce to the first case.}
\end{itemize}

\item If $\beta$ is an $m$-diagonal corresponding to the object $X_\beta$ in the transjective part of the Auslander-Reiten quiver of $\mathcal{C}^{m}_{\tilde{A_n}}$.
\begin{itemize}
\item \cor{If $k=|i_\alpha-i_\beta| \leq m$, then we can show that there exists a morphism in the module category from $X_\alpha$ to $X_\beta[k]$ with the article of Baur and Torkildsen \cite[Paragraphs 3.3 and 3.4]{BauTor}.}
\item \cor{Else, there exists an $m$-ear $\gamma$ which does not cross $\alpha$ nor $\beta$. It suffices to cut along $\gamma$ and repeat as many times as necessary in order to reduce to the previous case.}
\end{itemize}
\end{enumerate}
\end{enumerate}
\end{proof}

\subsection{Case $\tilde{D}$}

\begin{defi}
Let $P$ be a polygon with $(n-2)m$ sides with two $m-1$-gons inside of it, associated with a quiver of type $\tilde{D_n}$. Then, an $m$-ear is an $m$-diagonal linking a vertex $i$ to the vertex $i+m+1$ homotopic to the boundary of $P$.
\end{defi}

\begin{lem}\label{lem:dtilde}
Let $P$ be a polygon with $(n-2)m$ sides with two $m-1$-gons inside of it, associated with a quiver $Q$ of type $\tilde{D_n}$ and let $\alpha$ be an $m$-ear. Then the Iyama-Yoshino reduction of $\mathcal{C}^{m}_{\tilde{D_n}}$ applied on $X_\alpha$ corresponds to cutting along $\alpha$. More precisely, let $\mathcal{C}$ be the $m$-cluster category associated with a quiver of type $\tilde{D_n}$, and let \[\mathcal{C}'=\mathcal{U}/X_\alpha,\] where $X_\alpha$ is the $m$-rigid object associated with $\alpha$, and \[\mathcal{U}= \{ Y, \mathrm{Ext}_\mathcal{C}^l(X_\alpha,Y)=0~\forall l \in \{1,\cdots,m \} \}.\]
Let $Q/\alpha$ be the quiver $Q$ where the vertex corresponding to $\alpha$ and all the incident arrows have been removed. Then we have the following result:
\[ \mathcal{C'} \simeq \mathcal{C}^{m}_{Q/\alpha}. \]
\end{lem}

\begin{rmk}
Let us begin by illustrating this fact with the Gabriel quivers. We focus on the following $(m+2)$-angulation:

\begin{figure}[h!]
	\centering
\begin{tikzpicture}[scale=0.65]
	\fill[fill=black,fill opacity=0.1] (0,4) -- (-1.66,3.46) -- (-2.69,2.05) -- (-2.68,0.3) -- (-1.66,-1.11) -- (0,-1.65) -- (1.66,-1.11) -- (2.69,0.3) -- (2.69,2.05) -- (1.66,3.46) -- cycle;
	\draw [fill=black,fill opacity=1.0] (-1.66,1.57) circle (0.2cm);
	\draw [fill=black,fill opacity=1.0] (1.66,1.58) circle (0.2cm);
	\draw (0,4)-- (-1.66,3.46);
	\draw (-1.66,3.46)-- (-2.69,2.05);
	\draw (-2.69,2.05)-- (-2.68,0.3);
	\draw (-2.68,0.3)-- (-1.66,-1.11);
	\draw (-1.66,-1.11)-- (0,-1.65);
	\draw (0,-1.65)-- (1.66,-1.11);
	\draw (1.66,-1.11)-- (2.69,0.3);
	\draw (2.69,0.3)-- (2.69,2.05);
	\draw (2.69,2.05)-- (1.66,3.46);
	\draw (1.66,3.46)-- (0,4);
	\draw(-1.66,1.17) circle (0.4cm);
	\draw(1.66,1.18) circle (0.4cm);
	\draw [shift={(-7.95,-1.41)}] plot[domain=-0.03:0.66,variable=\t]({1*7.96*cos(\t r)+0*7.96*sin(\t r)},{0*7.96*cos(\t r)+1*7.96*sin(\t r)});
	\draw [shift={(8.42,-1.56)},color=black]  plot[domain=2.5:3.15,variable=\t]({1*8.42*cos(\t r)+0*8.42*sin(\t r)},{0*8.42*cos(\t r)+1*8.42*sin(\t r)});
	\draw [shift={(-1.34,-1.17)}] plot[domain=-1.77:1.07,variable=\t]({-0.15*3.42*cos(\t r)+-0.99*1.28*sin(\t r)},{0.99*3.42*cos(\t r)+-0.15*1.28*sin(\t r)});
	\draw [shift={(1.34,-1.17)}] plot[domain=-1.77:1.07,variable=\t]({0.15*3.42*cos(\t r)+0.99*1.28*sin(\t r)},{0.99*3.42*cos(\t r)+-0.15*1.28*sin(\t r)});
	\draw [shift={(1.47,-0.36)}] plot[domain=-0.32:1.37,variable=\t]({-0.91*3.99*cos(\t r)+-0.42*1.82*sin(\t r)},{0.42*3.99*cos(\t r)+-0.91*1.82*sin(\t r)});
	\draw [shift={(-1.47,-0.36)}] plot[domain=-0.32:1.37,variable=\t]({0.9*3.99*cos(\t r)+0.43*1.82*sin(\t r)},{0.43*3.99*cos(\t r)+-0.9*1.82*sin(\t r)});
	\draw [shift={(-0.37,0.1)}] plot[domain=2.86:5.31,variable=\t]({0.78*2.44*cos(\t r)+0.63*1.36*sin(\t r)},{-0.63*2.44*cos(\t r)+0.78*1.36*sin(\t r)});
	\draw [shift={(-1.26,-1.78)}] plot[domain=-1.59:0.07,variable=\t]({-0.16*3.73*cos(\t r)+-0.99*1.26*sin(\t r)},{0.99*3.73*cos(\t r)+-0.16*1.26*sin(\t r)});
	\draw [shift={(1.26,-1.78)}] plot[domain=-1.59:0.07,variable=\t]({0.16*3.73*cos(\t r)+0.99*1.26*sin(\t r)},{0.99*3.73*cos(\t r)+-0.16*1.26*sin(\t r)});
	\draw [shift={(0.37,0.1)}] plot[domain=2.86:5.31,variable=\t]({-0.78*2.44*cos(\t r)+-0.63*1.36*sin(\t r)},{-0.63*2.44*cos(\t r)+0.78*1.36*sin(\t r)});
\end{tikzpicture}
\caption{$(m+2)$-angulation of type $\tilde{D}$.}
\label{fig:tr}
\end{figure}

Let $Q$ be the associated quiver:
\[ \xymatrix@1{
	1 \ar[dr] & & & & & n-1 \\
	& 3 \ar[r] & \cdots \ar[r] & k-1 \ar[r] & k \ar[ur]\ar[dr] & \\
	2 \ar[ur] & & & & & n} \]
The mutation at vertex $k$ leads to the following quiver:

\[ \xymatrix@1{
1 \ar[dr] & & & & & n-1 \ar[dl] \\
& 3 \ar[r] & \cdots \ar[r] & k-1 \ar[urr]\ar[drr] & k & \\
2 \ar[ur] & & & & & n \ar[ul]} \]

Using the Iyama-Yoshino reduction at vertex $k$ corresponds to forget this vertex and all incident arrows. By doing this, we are ensured to be reduced to a quiver of type $\tilde{D}_{n-1}$:

\[ \xymatrix@1{
1 \ar[dr] & & & & & n-1 \\
& 3 \ar[r] & \cdots \ar[r] & k-1 \ar[urr]\ar[drr] & & \\
2 \ar[ur] & & & & & n} \]

Another illustration is given in figure \ref{fig:tubes}.
\end{rmk}

\begin{proof}
Here again, we use theorem \ref{th:kr} of Keller and Reiten. The biggest difficulty in this proof, is to build an $m$-ear. In fact, the $m$-ear is made from the flip of a special $(m+2)$-angulation. We are going to build an $m$-ear in this way, but any $m$-ear can be built in this way by rotation. From this $m$-ear, we build a new $m$-cluster-tilting object respecting the hypotheses of the Theorem of Keller and Reiten.

\cor{Let \[T=\bigoplus_{i=1}^n P_i\] be the sum of all projective objects. We know that $T$ is an $m$-cluster-tilting object. This object is naturally associated with the initial $(m+2)$-angulation defined the be author in \cite[Definition 2.15]{JM}.}

Let $P_{n-2}$ be the following projective module, which we can see as an object in $\mathcal{C}^{m}_{\tilde{D_n}}$:
\[ \xymatrix@1{
K \ar[dr] & & & & & 0 \\
& K \ar[r] & \cdots \ar[r] & K \ar[r] & K \ar[ur]\ar[dr] & \\
K \ar[ur] & & & & & 0} \]
From Iyama and Yoshino in \cite{IY}, we have an exchange triangle:

\begin{equation*}
P_{n-2} \to Y \to X \to P_{n-2}[1]
\label{eq:tri}
\end{equation*}

where $Y \in \mathrm{add} \bigoplus_{j \neq n-2} P_j$.

\cor{Let us consider $T_{n-2}=T/P_{n-2}$ the almost $m$-cluster-tilting object (see Wraalsen and Zhou, Zhu in \cite{W} and \cite{ZZ} for any results on these almost $m$-cluster-tilting objects).}

Let
\[T'=\bigoplus_{j \neq n-2} P_j \oplus X=T_{n-2} \oplus X.\]

Let us first show that $X$ is in fact $X_\alpha$, corresponding to the $m$-diagonal, $\alpha$, which is the arc obtained by flipping the $m$-diagonal of type $1$ corresponding to the vertex $n-2$ of the quiver (see figure \ref{fig:tubes}).

\begin{figure}[!h]
\centering
\begin{tikzpicture}[scale=0.6]
\fill[fill=black,fill opacity=0.1] (0,4) -- (-1.66,3.46) -- (-2.69,2.05) -- (-2.68,0.3) -- (-1.66,-1.11) -- (0,-1.65) -- (1.66,-1.11) -- (2.69,0.3) -- (2.69,2.05) -- (1.66,3.46) -- cycle;
\draw [fill=black,fill opacity=1.0] (1.66,1.61) circle (0.2cm);
\draw [fill=black,fill opacity=1.0] (-1.66,1.61) circle (0.2cm);
\draw (0,4)-- (-1.66,3.46);
\draw (-1.66,3.46)-- (-2.69,2.05);
\draw (-2.69,2.05)-- (-2.68,0.3);
\draw (-2.68,0.3)-- (-1.66,-1.11);
\draw (-1.66,-1.11)-- (0,-1.65);
\draw (0,-1.65)-- (1.66,-1.11);
\draw (1.66,-1.11)-- (2.69,0.3);
\draw (2.69,0.3)-- (2.69,2.05);
\draw (2.69,2.05)-- (1.66,3.46);
\draw (1.66,3.46)-- (0,4);
\draw [shift={(-8.03,-1.44)}] plot[domain=-0.03:0.66,variable=\t]({1*8.03*cos(\t r)+0*8.03*sin(\t r)},{0*8.03*cos(\t r)+1*8.03*sin(\t r)});
\draw [shift={(-0.75,-11.66)}] plot[domain=-0.82:0.59,variable=\t]({-0.07*14.49*cos(\t r)+-1*1.99*sin(\t r)},{1*14.49*cos(\t r)+-0.07*1.99*sin(\t r)});
\draw [shift={(8.03,-1.43)}] plot[domain=2.49:3.17,variable=\t]({1*8.03*cos(\t r)+0*8.03*sin(\t r)},{0*8.03*cos(\t r)+1*8.03*sin(\t r)});
\draw [shift={(0.76,-11.66)},color=red]  plot[domain=-0.82:0.59,variable=\t]({0.07*14.49*cos(\t r)+1*1.99*sin(\t r)},{1*14.49*cos(\t r)+-0.07*1.99*sin(\t r)});
\draw(1.67,1.18) circle (0.44cm);
\draw(-1.67,1.17) circle (0.44cm);
\draw [shift={(-7.33,-2.91)}] plot[domain=0.17:0.66,variable=\t]({1*7.44*cos(\t r)+0*7.44*sin(\t r)},{0*7.44*cos(\t r)+1*7.44*sin(\t r)});
\draw [shift={(7.33,-2.9)}] plot[domain=2.48:2.97,variable=\t]({1*7.44*cos(\t r)+0*7.44*sin(\t r)},{0*7.44*cos(\t r)+1*7.44*sin(\t r)});
\draw [shift={(2.52,-0.41)}] plot[domain=-0.43:1.16,variable=\t]({-0.96*5.16*cos(\t r)+-0.29*2.09*sin(\t r)},{0.29*5.16*cos(\t r)+-0.96*2.09*sin(\t r)});
\draw [shift={(-2.52,-0.42)}] plot[domain=-0.43:1.16,variable=\t]({0.96*5.16*cos(\t r)+0.3*2.09*sin(\t r)},{0.3*5.16*cos(\t r)+-0.96*2.09*sin(\t r)});
\begin{scriptsize}
\draw[color=black] (-0.37,1.28) node {$4$};
\draw[color=black] (-1.48,2.56) node {$3$};
\draw[color=black] (0.48,1.19) node {$5$};
\draw[color=black] (1.62,2.85) node {$6$};
\draw[color=black] (-0.7,0.23) node {$2$};
\draw[color=black] (0.62,0.23) node {$7$};
\draw[color=black] (-2.12,0.21) node {$1$};
\draw[color=black] (2.41,0.43) node {$8$};
\end{scriptsize}
\end{tikzpicture}
\hspace{10pt}
$\xrightarrow{\text{flip at k=6}}$
\hspace{10pt}
\begin{tikzpicture}[scale=0.6]
\fill[fill=black,fill opacity=0.1] (0,4) -- (-1.66,3.46) -- (-2.69,2.05) -- (-2.68,0.3) -- (-1.66,-1.11) -- (0,-1.65) -- (1.66,-1.11) -- (2.69,0.3) -- (2.69,2.05) -- (1.66,3.46) -- cycle;
\draw [fill=black,fill opacity=1.0] (1.66,1.61) circle (0.2cm);
\draw [fill=black,fill opacity=1.0] (-1.66,1.61) circle (0.2cm);
\draw (0,4)-- (-1.66,3.46);
\draw (-1.66,3.46)-- (-2.69,2.05);
\draw (-2.69,2.05)-- (-2.68,0.3);
\draw (-2.68,0.3)-- (-1.66,-1.11);
\draw (-1.66,-1.11)-- (0,-1.65);
\draw (0,-1.65)-- (1.66,-1.11);
\draw (1.66,-1.11)-- (2.69,0.3);
\draw (2.69,0.3)-- (2.69,2.05);
\draw (2.69,2.05)-- (1.66,3.46);
\draw (1.66,3.46)-- (0,4);
\draw [shift={(-8.03,-1.44)}] plot[domain=-0.03:0.66,variable=\t]({1*8.03*cos(\t r)+0*8.03*sin(\t r)},{0*8.03*cos(\t r)+1*8.03*sin(\t r)});
\draw [shift={(-0.75,-11.66)}] plot[domain=-0.82:0.59,variable=\t]({-0.07*14.49*cos(\t r)+-1*1.99*sin(\t r)},{1*14.49*cos(\t r)+-0.07*1.99*sin(\t r)});
\draw [shift={(8.03,-1.43)}] plot[domain=2.49:3.17,variable=\t]({1*8.03*cos(\t r)+0*8.03*sin(\t r)},{0*8.03*cos(\t r)+1*8.03*sin(\t r)});
\draw(1.67,1.18) circle (0.44cm);
\draw(-1.67,1.17) circle (0.44cm);
\draw [shift={(0.6,-0.09)}] plot[domain=-0.36:1.61,variable=\t]({-0.9*3.18*cos(\t r)+-0.43*1.66*sin(\t r)},{0.43*3.18*cos(\t r)+-0.9*1.66*sin(\t r)});
\draw [shift={(-0.6,-0.09)}] plot[domain=-0.36:1.61,variable=\t]({0.9*3.18*cos(\t r)+0.43*1.66*sin(\t r)},{0.43*3.18*cos(\t r)+-0.9*1.66*sin(\t r)});
\draw [shift={(-1.22,1.18)},color=red]  plot[domain=-0.67:0.67,variable=\t]({1*3.67*cos(\t r)+0*3.67*sin(\t r)},{0*3.67*cos(\t r)+1*3.67*sin(\t r)});
\draw [shift={(-8.55,-3.43)}] plot[domain=0.21:0.64,variable=\t]({1*8.73*cos(\t r)+0*8.73*sin(\t r)},{0*8.73*cos(\t r)+1*8.73*sin(\t r)});
\draw [shift={(8.55,-3.42)}] plot[domain=2.51:2.94,variable=\t]({1*8.73*cos(\t r)+0*8.73*sin(\t r)},{0*8.73*cos(\t r)+1*8.73*sin(\t r)});
\begin{scriptsize}
\draw[color=black] (-0.37,1.28) node {$4$};
\draw[color=black] (-1.48,2.56) node {$3$};
\draw[color=black] (0.48,1.19) node {$5$};
\draw[color=black] (-1.92,-0.06) node {$2$};
\draw[color=black] (2.07,0.27) node {$8$};
\draw[color=black] (2.56,1.31) node {$6$};
\draw[color=black] (-0.76,0.28) node {$1$};
\draw[color=black] (0.8,0.44) node {$7$};
\end{scriptsize}
\end{tikzpicture}
\caption{We flip the arc corresponding to $P_{n-2}$. The new arc is called by $\alpha$.}
\label{fig:tubes}
\end{figure}

\cor{We have to show that $\tau X=S_{n-2}$, the simple module in $n-2$, which is situated at the bottom of the tube of size $n-2$ as we set in the previous section.}

Let us find $\mathcal{C}(X[-1],T)$. For all $i \neq k$, $\mathcal{C}(X[-1],P_i) \simeq \mathcal{C}(X,P_i[1])=0$ since $T'$ is an $m$-cluster-tilting object.

Now we focus on $\mathcal{C}(X[-1],P_{n-2})$.
We have that $\mathcal{C}(X[-1],P_{n-2}) \simeq \mathcal{C}(X,P_{n-2}[1]) \simeq K$ from Iyama and Yoshino \cor{at Corollary 6.4} in \cite{IY}.

\cor{From the Auslander-Reiten duality, we have:}

\[ \mathcal{C}(X[-1],T) \simeq \mathcal{C}(X,T[1]) \simeq D\mathcal{C}(T,\tau X) \]

Then $\tau X \simeq S_{n-2}$ and $X=X_\alpha$ corresponds to the arc $\alpha$.

\cor{It now remains to check the hypotheses of the Theorem of Keller-Reiten. First, $\mathcal{C}'$ is a Hom-finite algebraic $(m+1)$-Calabi-Yau category. The object $T'$ is our candidate. It is still an $m$-cluster-tilting object.
First of all, from \cite[Lemma 4.1]{KR}, we have \[\mathcal{C}(P_j,P_l[-i])=0\] for any $j$ and $l$. Moreover, we have $\mathcal{C}(X,X[-i])=0$. In addition, $\mathcal{C}(X[-i],P_j)=0$ as there is no morphism from a regular object to a preprojective object. Finally, we have that $\mathcal{C}(P_j,X[-i])=0$ for any $j \neq k$ and any $i \in \{1,\cdots,m\}$ as in type $\tilde{A}$.}

%
%

We can thus apply the Theorem of Keller and Reiten, and this finishes the proof.
\end{proof}

\begin{lem}\label{lem:cut}
\cor{Let $P$ be a polygon associated with a quiver $Q$ of type $\tilde{D}_n$. Let $\alpha$ be an $m$-ear. Let $\beta$ be an $m$-diagonal which cuts $\alpha$. Let $X_\beta$ be the associated $m$-rigid object. Then there exists $k \in \{1,\cdots,m\}$ such that $\mathrm{Ext}_\mathcal{C}^k(X_\alpha,X_\beta) \neq 0$.}
\end{lem}

\begin{rmk}
\cor{We need to note that the cases are symmetric. Indeed, since the higher cluster category is $(m+1)$-Calabi-Yau, we know that \[\mathcal{C}(\beta[k-(m+1)],\alpha) \simeq D\mathcal{C}(\alpha,\beta[k]).\] This means that a morphism from $\alpha$ to $\beta[k]$ is in $1-1$-correspondence with a morphism from $\beta[k-(m+1)]$ to $\alpha$. Thus, shifting $\beta$, $k$ times is the same as shifting $\alpha$, $k$ times. This means no matter which vertex we shift.}
\end{rmk}

\begin{proof}
Let $\beta$ be an arc crossing $\alpha$.

\cor{Let $X_\alpha$ be the $m$-rigid indecomposable object associated with $\alpha$. By the geometric realization of $\tilde{D}$, the $m$-diagonal $\alpha$ is situated at the bottom of a tube of size $n-1$. Without any loss of generality, we suppose that $\alpha$ links $i$ and $i+m$ in $P$.}

\cor{Let $\beta$ be an $m$-diagonal crossing $\alpha$. It means that there exists $k \in \{1,\cdots,m-1\}$, such that an extremity $i_\beta$ of $\beta$ is the vertex $i+k$ in the outer polygon. Therefore, it suffices to shift $\beta$ $k<m$ times, so that one extremity of $\beta[k]$ is $i$, common to $\alpha$.}

\cor{If $\beta[k]$ is in a tube of size $n-2$, then it is situated at the same tube as $\alpha$ (since they both end in $i$), higher than it, in the sense that we can draw a path from $\alpha$ to $\beta[k]$. As $\beta[k]$ and $\alpha$ share a common end $i$, the number of arrow (it means indecomposable morphisms) from $\alpha$ to $\beta[k]$ is the same as the aisle of the tube. Then there exists a morphism from $X_\alpha$ to $X_\beta[k]$.}

\cor{Else, $\beta[k]$ is situated in a slice $\mathcal{S}$ of the preinjective part of the Auslander-Reiten quiver. Then it suffices to show that there is a morphism from $X_\alpha$ which lies in the tube, to the only source of $\mathcal{S}$, it means in our orientation, to $A_1$ in the following quiver.}

\[ \scalebox{0.5}{
 \xymatrix{
 & & & & A_8 \\
 & & & A_6\ar[ur] \ar[r] & A_7 \\
 & & X_\beta[k] \ar[ur] & & \\
 & A_4 \ar[ur] & & & \\
 A_1 \ar[ur] \ar[r] \ar[dr] & A_2 & & & \\
 & A_3 & & & & \\
 }}
\]

We note that $\beta[k-1]$ is exactly the arc corresponding to $\tau^{-1}P_l$ for an $l \in \{1,\cdots,n+1 \}$. We can prove the existence of a morphism in the module category from the simple regular $X_\alpha$ to $\tau^{-1}P_l$ for any $l$.

To draw an example, in case $n=7$, let us give the dimension vectors of $\tau^{-1}P_l$, for each $l$: They are

$\scalebox{0.5}{\xymatrix{1 \ar[dr] & & & & 0 \ar[dl] \\ & 2 & 1 \ar[l] & 1 \ar[l]& \\ 1 \ar[ur] & & & & 1 \ar[ul]}}$ ; $\scalebox{0.5}{\xymatrix{1 \ar[dr] & & & & 1 \ar[dl] \\ & 2 & 1\ar[l] & 1\ar[l] & \\ 1 \ar[ur] & & & & 0 \ar[ul]}}$ ; $\scalebox{0.5}{\xymatrix{1 \ar[dr] & & & & 1 \ar[dl] \\ & 3 & 2\ar[l] & 2\ar[l] & \\ 1 \ar[ur] & & & & 1 \ar[ul]}}$ ; $\scalebox{0.5}{\xymatrix{1 \ar[dr] & & & & 0 \ar[dl] \\ & 2 & 1\ar[l] & 1\ar[l] & \\ 1 \ar[ur] & & & & 0 \ar[ul]}}$ ; $\scalebox{0.5}{\xymatrix{1 \ar[dr] & & & & 0 \ar[dl] \\ & 2 & 1\ar[l] & 0\ar[l] & \\ 1 \ar[ur] & & & & 0 \ar[ul]}}$ ; $\scalebox{0.5}{\xymatrix{1 \ar[dr] & & & & 0 \ar[dl] \\ & 1 & 1\ar[l] & 0\ar[l] & \\ 0 \ar[ur] & & & & 0 \ar[ul]}}$ ; $\scalebox{0.5}{\xymatrix{0 \ar[dr] & & & & 0 \ar[dl] \\ & 1 & 1\ar[l] & 0\ar[l] & .\\ 1 \ar[ur] & & & & 0 \ar[ul]}}$

In any case there is a morphism from the simple \[\scalebox{0.5}{\xymatrix{0 \ar[dr] & & & & 0 \ar[dl] \\ & 1 & 0\ar[l] & 0\ar[l] & \\ 0 \ar[ur] & & & & 0 \ar[ul]}}\] to any of the $\tau^{-1}P_l$. We can have a deeper analyze in the article of Dlab and Ringel in \cite{DR}.
\end{proof}

We now generalize this result to all arcs excepted the ones in the tubes of size $2$.

\cor{Let $P$ be a polygon with $(n-2)m$ sides. Let $\alpha$ from $i$ to $i+km+2$, for a $k \in \{1,\cdots,n-4\}$, be an $m$-diagonal homotopic to the boundary patch (which corresponds to a regular module in a tube of size $n-2$). Then $\alpha$ cuts the figure into a polygon $T$ with $km+2$ sides on the one hand, and another figure of type $\tilde{D_{n'}}$ for some $n'<n$ on the other hand. Let $\alpha_1,\cdots,\alpha_k$ be $m$-diagonals (in the sense of type $A$) lying in $T$, all ending at the same vertex $i$.}

\begin{lem}\label{lem:succ}
Under these hypotheses, for a fixed $j \in \{1,\cdots,k\}$, let $X_j$ be the object corresponding to the $m$-diagonal $\alpha_j$.
\cor{Let $\mathcal{C}$ be the $m$-cluster category of type $\tilde{D_n}$, associated with an $(m+2)$-angulation containing $\alpha_j$, and let \[\mathcal{C}'_j=\mathcal{U}/\bigoplus X_{\alpha_i},\] where \[\mathcal{U}= \{ Y, \mathrm{Ext}_\mathcal{C}^l(\bigoplus X_{\alpha_i},Y)=0~\forall l \in \{1,\cdots,m \} \}.\] Let $Q/\alpha_1,\cdots,\alpha_n$ be the quiver $Q$ where the vertices corresponding to $\alpha_1,\cdots,\alpha_n$ and all the incident arrows have been removed. Then we have the following result:
\[ \mathcal{C'} \simeq \mathcal{C}^{m}_{Q/\alpha_1,\cdots,\alpha_n}. \]}
\end{lem}

\begin{proof}
We have that $\alpha$ is an $m$-diagonal linking two different vertices $i$ and $i+km+1$, homotopic to the boundary path.

\cor{If $\alpha$ is an $m$-ear, this is exactly the previous lemma. Else, it means that $k >1$. Then there exists $\gamma$ an $m$-ear from $i$ to $i+m+1$ which does not cut $\alpha$. We use the Lemma \ref{lem:dtilde} applied to $\gamma$ in order to cut along this $m$-ear. Then, if $\mathcal{C}'$ is the Iyama-Yoshino reduction at object $X_\gamma$, we obtain that \[\mathcal{C}' \simeq \mathcal{C}^m_{Q/\gamma}.\] We repeat this operation as many times as necessary, to reduce $n$ until $\alpha$ becomes an $m$-ear. We are ensured that the process stops since $\alpha$ cuts the polygon into a $km$-gon with both $m-1$ gons inside of it on the first side, and into a $km+2$-gon of type $A$ on the other side. This shows the result if $\alpha$ is such an $m$-diagonal, corresponding to an $m$-rigid indecomposable object in a tube of size $n-2$.}
\end{proof}

\begin{lem}\label{lem:iydtilde}
\cor{Let $P$ be a polygon with $(n-2)m$ sides. Let $\alpha$ from $i$ to $j$ be an $m$-diagonal which is associated to an $m$-rigid object $X_\alpha$ lying in the transjective component of the Auslander-Reiten quiver of $\mathcal{C}^{m}_Q$. Let $\mathcal{C}$ be the $m$-cluster category associated with a quiver of type $\tilde{D_n}$, and let $\mathcal{C}'=\mathcal{U}/X_\alpha$, where $X_\alpha$ is the $m$-rigid object associated with $\alpha$, and \[\mathcal{U}= \{ Y, \mathrm{Ext}_\mathcal{C}^l(X_\alpha,Y)=0~\forall l \in \{1,\cdots,m \} \}.\]
Let $\Delta$ be an $(m+2)$-angulation containing $\alpha$, and let $Q$ be the associated suiver. Let $Q/\alpha$ be the quiver $Q$ where the vertex corresponding to $\alpha$ and all the incident arrows have been removed. Then we have the following result:
\[ \mathcal{C'} \simeq \mathcal{C}^{m}_{Q/\alpha}. \]}
\end{lem}

\begin{proof}
From the hypotheses, we have that $X_\alpha$ is situated in the preprojective (or preinjective) part of the Auslander-Reiten quiver of $\mathcal{C}^{m}_Q$.

\cor{For $i \in \{1,\cdots,n\}$, let $X_i$ be the objects of the slice in which $X_\alpha$ is. Let $\alpha_i$ be the $m$-diagonal associated with $X_i$.
The objects $X_i$ are in the transjective part of the Auslander-Reiten quiver of the higher cluster category, and, the $m$-diagonals $\alpha_i$ do not cross each other, neither $\alpha$. Indeed, from \cite[Theorem 4.3]{JM}, in the transjective part of the Auslander-Reiten quiver of $\mathcal{C}^{m}_Q$, a slice is composed of objects associated with a collection of $m$-diagonals which do not cross each other, since they form an $(m+2)$-angulation.}

We are now able to use the theorem of Keller and Reiten in \cite[Theorem 4.2]{KR}. Indeed, let \[T=X \oplus\bigoplus_{i=1}^n X_i.\] We know that $T$ is an $m$-cluster-tilting object (because there exist $k \in \mathbb{N}$ such that $T=\tau^k(\bigoplus_{i=1}^{n+1} P_i)$).

Let $\mathcal{C}'$ be the Iyama Yoshino reduction of the higher cluster category $\mathcal{C}^{m}_Q$ over $T$.

We also have that \[\mathrm{Ext}_\mathcal{C'}^{-i}(T,T)=0.\] Moreover, we can check that $\mathrm{End}_\mathcal{C'}(T)=KQ_{T}$ as in the previous lemma. Then we can apply the theorem of Keller and Reiten, and this shows the result.
\end{proof}

We now state a technical lemma which helps us to find morphisms between two $m$-rigid objects.

\begin{lem}\label{lem:morph}
Let $\alpha$ and $\beta$ be two $m$-diagonals. Suppose that there exists $\Delta$ containing $\alpha$, and not $\beta$, and such that $\beta=\kappa^i_\Delta(\alpha)$, for an $i \in \{1,\cdots,m-1\}$.

Let $X_\alpha$ (respectively $X_\beta$) be the $m$-rigid indecomposable object associated with $\alpha$ (respectively $\beta$).

Then \[\mathrm{Ext}_\mathcal{C}^i(X_\alpha,X_\beta) \neq 0.\]
\end{lem}

\begin{proof}
\cor{We number the arcs in $\Delta$ and consider that $\alpha$ corresponds to $k$. We use Iyama-Yoshino reduction in order to prove the statement. Let us introduce \[\mathcal{C}'=\mathcal{U}/\bigoplus_{j \neq k} X_j,\] where
\[ \mathcal{U}= \{ Y, \mathrm{Ext}_\mathcal{C}^l(\bigoplus_{j \neq k} X_j,Y)=0~\forall l \in \{1,\cdots,m \} \}. \]}

By Iyama and Yoshino in \cite[Theorem 4.2]{IY}, we know that $\mathcal{C}'$ is triangulated and $(m+1)$-Calabi-Yau. If $X \to Y \to Z \to X[1]$ is a triangle in $\mathcal{C}$, where $X \to Y$ is a $\oplus X_j$-left approximation, then $Z$ is isomorphic to the shift of $X$ in $\mathcal{C}'$, which we note $Z=X\langle 1 \rangle$.

\cor{From the Lemma \ref{lem:iydtilde}, we have an equivalence of categories between $\mathcal{C'}$ and $\mathcal{C}^{m}_{Q'}$ where $Q'$ is obtained from $Q$ by forgetting all arrows indicent to any of the $X_j, j\neq k$. This roughly corresponds to cut along all the arcs of the $(m+2)$-angulation excepted $\alpha$. Then, as $\beta$ is the $i$-th twist of $\alpha$, it becomes the $i$-th shift in the reduced category. From Iyama and Yoshino in \cite[Theorem 4.2]{IY}, we have $X_\beta=X_\alpha\langle i \rangle$. Then \[\mathrm{Ext}_{\mathcal{C}'}^i(X_\alpha,X_\beta) \neq 0.\] Thus, we can find a morphism in the higher cluster category $\mathcal{C}$.}
\end{proof}

\begin{coro}
Under these hypotheses, the following diagram is commutative:

\cor{\[\xymatrix{\{\beta \in \Gamma\setminus\{\alpha\}, \beta \text{ does not cross } \alpha\} \ar[r] \ar@{<->}[d] & \{X \in \mathcal{U}; X \ncong X_\alpha\}/\simeq \ar@{<->}[d]\\ S/\alpha \ar[r] & \mathcal{C}'/\simeq}\]
where the left vertical bijection is given in the way of Marsh and Palu in \cite{MP}, and the right one is given by the Iyama-Yoshino reduction. The horizontal arrows are maps sending $\beta$ to $X_\beta$.}
\end{coro}

\begin{proof}
Let $\beta$ be an arc which does not cross $\alpha$. Then, $X_\beta$ belongs to $\mathcal{U}$ from the Lemma \ref{lem:morph}. We can apply to $X_\beta$ the Iyama-Yoshino reduction.

On the other hand, if we cut along $\beta$, let us consider the surface left. We take the notations from Definition \ref{def:arcdtilde}:
\begin{itemize}
\item \cor{If $\beta$ is of type $1$, then we obtain two figures of type $D_l$. To be of type $1$ for $\beta$ means that $X_\beta$ is in the transjective part of the Auslander-Reiten quiver. Then, forgetting all arrows incident to $X_\beta$, we observe that the Iyama-Yoshino reduction are two categories of type $D_l$.}
\item \cor{If $\beta$ is of type $2$, then it cuts the polygon into two figures, one of type $A$, one of type $\tilde{D}_k$, for $k<n$. It happens the same in the higher cluster category, and, as we have seen in a previous Lemma, in this case, the Iyama-Yoshino reduction leads to a category of type $\tilde{D}_k$.}
\item \cor{If $\beta$ is of type $3$, then it cuts the figure into a type $D_n$, and the Iyama-Yoshino reduction leads to a category of type $D_n$.}
\end{itemize}
\end{proof}

Before showing the main lemma of this section, we show that we can reduce to the case $\tilde{D}_4$. The following lemma show that we can reduce to cases where $n \leq 6$, and the next remark treats cases $n=5$ and $n=6$.

\begin{lem}
Suppose that $n>6$. Let $\alpha$ and $\beta$ be two crossing $m$-diagonals in the $(n-2)m$-gon realizing $\tilde{D}_n$. Then there exist at least $(n-4)$ $m$-ears which do not cut $\alpha$ neither $\beta$.
\end{lem}

\begin{proof}
\cor{The cases where $\alpha$ or $\beta$ are of type different from $1$ are immediate, and left to the reader. The case where $\alpha$ and $\beta$ are of type one (it means, crossing the space between both inner polygons, see figure \ref{fig:cross}) is the most difficult. The arcs cut the polygon $P$ into $4$ parts. If we cannot draw an $m$-ear between one of the parts, it means that the number of vertices strictly contained in a part is at most $m-1$ in each part. Then the total number of vertices is at most $4(m-1)+4$. Then $(n-2)m \leq 4m$ this means $n \leq 6$.}
\end{proof}

\begin{figure}[!h]
\centering
\begin{tikzpicture}[scale=0.7]
\fill[fill=black,fill opacity=0.1] (0.,4.) -- (-1.78,3.42) -- (-2.87913480366,1.90451239418) -- (-2.87757227419,0.0324019384391) -- (-1.77590924476,-1.48124880383) -- (0.00505645156165,-2.05827669617) -- (1.78505645156,-1.47827669617) -- (2.88419125522,0.0372109096461) -- (2.88262872576,1.90932136539) -- (1.78096569632,3.42297210766) -- cycle;
\draw [fill=black,fill opacity=1.0] (-1.35763963744,0.969726407807) circle (0.4cm);
\draw [fill=black,fill opacity=1.0] (1.362696089,0.971996896021) circle (0.4cm);
\draw (0.,4.)-- (-1.78,3.42);
\draw (-1.78,3.42)-- (-2.87913480366,1.90451239418);
\draw (-2.87913480366,1.90451239418)-- (-2.87757227419,0.0324019384391);
\draw (-2.87757227419,0.0324019384391)-- (-1.77590924476,-1.48124880383);
\draw (-1.77590924476,-1.48124880383)-- (0.00505645156165,-2.05827669617);
\draw (0.00505645156165,-2.05827669617)-- (1.78505645156,-1.47827669617);
\draw (1.78505645156,-1.47827669617)-- (2.88419125522,0.0372109096461);
\draw (2.88419125522,0.0372109096461)-- (2.88262872576,1.90932136539);
\draw (2.88262872576,1.90932136539)-- (1.78096569632,3.42297210766);
\draw (1.78096569632,3.42297210766)-- (0.,4.);
\draw [shift={(4.20112235564,-0.389471074632)}] plot[domain=2.13641512908:3.52012091758,variable=\t]({1.*4.51573706872*cos(\t r)+0.*4.51573706872*sin(\t r)},{0.*4.51573706872*cos(\t r)+1.*4.51573706872*sin(\t r)});
\draw [shift={(-2.31999172307,-2.2314299915)}] plot[domain=0.18145201887:1.70517284349,variable=\t]({1.*4.17356686807*cos(\t r)+0.*4.17356686807*sin(\t r)},{0.*4.17356686807*cos(\t r)+1.*4.17356686807*sin(\t r)});
\begin{scriptsize}
\draw[color=black] (0.68232081142,1.93504417731) node {$\beta$};
\draw[color=black] (-0.579888054095,1.92001788129) node {$\alpha$};
\end{scriptsize}
\end{tikzpicture}
\caption{Two $m$-diagonals $\alpha$ and $\beta$ of type $1$, for $m=2$, $n=7$.}
\label{fig:cross}
\end{figure}

\begin{rmk}
If $n=5$ or $n=6$, then the only case where we cannot reduce to $\tilde{D}_4$ is when $\alpha$ and $\beta$ are of type $1$. But at this moment there exists $k < m$ such that $\alpha=\beta[k]$, then there exists a nonzero extension between $\alpha$ and $\beta$.
\end{rmk}

\begin{lem}\label{lem:arcs}
Let $\alpha$ and $\beta$ be two arcs in an $(m+2)$-angulation $\Delta$. Let $X_\alpha$ and $X_\beta$ be their associated $m$-rigid indecomposable objects. If for any $i \in \{ 1,\cdots,m \}$ we have  \[\mathrm{Ext}_\mathcal{C}^i(X_\alpha,X_\beta)=0,\] then $\alpha$ and $\beta$ do not cross.
\end{lem}

\begin{proof}
\cor{Taken into account both previous Lemmas, we only have to show the result for $\tilde{D}_4$. Indeed, if $n>4$, then it suffices to draw $m$-ears, and to aaply the Iyama-Yoshino reduction along the $m$-rigid indecomposable objects corresponding to these $m$-diagonals, to reduce to case $n=4$.}

Let now $P$ be a polygon with $2m$ vertices. The inner polygons which characterize case $\tilde{D}$ contain $m-1$ vertices.

We recall that we are in the case where $n=4$, it means that we study a $2m$-gon.
We show that if $\alpha$ crosses $\beta$, then \[{\mathrm{Ext}}_\mathcal{C}^i(X_\alpha,X_\beta) \neq 0, \text{ for some } i \in \{1,\cdots,n\}.\]

Suppose that $\alpha$ and $\beta$ cross. Then both arcs can be of different type. Let us sum up all the cases to treat in the following tabular:

\begin{center}
\begin{tabular}{|c|c|c|c|c|}
\hline
$\alpha$, $\beta$ & \begin{tikzpicture}[scale=0.3]
\fill[dash pattern=on 3pt off 3pt,fill=black,fill opacity=0.1] (0,4) -- (-2.48,2.98) -- (-3.51,0.51) -- (-2.49,-1.97) -- (-0.02,-3.01) -- (2.46,-1.99) -- (3.49,0.49) -- (2.47,2.97) -- cycle;
\fill[fill=black,fill opacity=1.0] (1.41,2.37) -- (1.21,2.14) -- (1.32,1.85) -- (1.61,1.8) -- (1.81,2.03) -- (1.71,2.32) -- cycle;
\fill[fill=black,fill opacity=1.0] (-1.89,-0.91) -- (-1.66,-0.72) -- (-1.37,-0.82) -- (-1.32,-1.12) -- (-1.55,-1.31) -- (-1.84,-1.21) -- cycle;
\draw (0,4)-- (-2.48,2.98);
\draw [dash pattern=on 3pt off 3pt] (-2.48,2.98)-- (-3.51,0.51);
\draw (-3.51,0.51)-- (-2.49,-1.97);
\draw [dash pattern=on 3pt off 3pt] (-2.49,-1.97)-- (-0.02,-3.01);
\draw (-0.02,-3.01)-- (2.46,-1.99);
\draw [dash pattern=on 3pt off 3pt] (2.46,-1.99)-- (3.49,0.49);
\draw (3.49,0.49)-- (2.47,2.97);
\draw [dash pattern=on 3pt off 3pt] (2.47,2.97)-- (0,4);
\draw (0.21,4.5)-- (2.68,3.46);
\draw (2.68,3.46)-- (2.61,3.62);
\draw (2.68,3.46)-- (2.52,3.4);
\draw (0.21,4.5)-- (0.28,4.34);
\draw (0.21,4.5)-- (0.37,4.56);
\draw (1.41,2.37)-- (1.21,2.14);
\draw (1.21,2.14)-- (1.32,1.85);
\draw (1.32,1.85)-- (1.61,1.8);
\draw (1.61,1.8)-- (1.81,2.03);
\draw (1.81,2.03)-- (1.71,2.32);
\draw (1.71,2.32)-- (1.41,2.37);
\draw (0,4)-- (-0.02,-3.01);
\begin{scriptsize}
\draw[color=black] (1.58,4.14) node {$m$};
\end{scriptsize}
\end{tikzpicture} & \begin{tikzpicture}[scale=0.3]
\fill[dash pattern=on 3pt off 3pt,fill=black,fill opacity=0.1] (0,4) -- (-2.48,2.98) -- (-3.51,0.51) -- (-2.49,-1.97) -- (-0.02,-3.01) -- (2.46,-1.99) -- (3.49,0.49) -- (2.47,2.97) -- cycle;
\fill[fill=black,fill opacity=1.0] (1.41,2.37) -- (1.21,2.14) -- (1.32,1.85) -- (1.61,1.8) -- (1.81,2.03) -- (1.71,2.32) -- cycle;
\fill[fill=black,fill opacity=1.0] (-1.89,-0.91) -- (-1.66,-0.72) -- (-1.37,-0.82) -- (-1.32,-1.12) -- (-1.55,-1.31) -- (-1.84,-1.21) -- cycle;
\draw (0,4)-- (-2.48,2.98);
\draw [dash pattern=on 3pt off 3pt] (-2.48,2.98)-- (-3.51,0.51);
\draw (-3.51,0.51)-- (-2.49,-1.97);
\draw [dash pattern=on 3pt off 3pt] (-2.49,-1.97)-- (-0.02,-3.01);
\draw (-0.02,-3.01)-- (2.46,-1.99);
\draw [dash pattern=on 3pt off 3pt] (2.46,-1.99)-- (3.49,0.49);
\draw (3.49,0.49)-- (2.47,2.97);
\draw [dash pattern=on 3pt off 3pt] (2.47,2.97)-- (0,4);
\draw (0.21,4.5)-- (2.68,3.46);
\draw (2.68,3.46)-- (2.61,3.62);
\draw (2.68,3.46)-- (2.52,3.4);
\draw (0.21,4.5)-- (0.28,4.34);
\draw (0.21,4.5)-- (0.37,4.56);
\draw (1.41,2.37)-- (1.21,2.14);
\draw (1.21,2.14)-- (1.32,1.85);
\draw (1.32,1.85)-- (1.61,1.8);
\draw (1.61,1.8)-- (1.81,2.03);
\draw (1.81,2.03)-- (1.71,2.32);
\draw (1.71,2.32)-- (1.41,2.37);
\draw [shift={(-4.93,1.62)}] plot[domain=-0.76:0.06,variable=\t]({1*6.75*cos(\t r)+0*6.75*sin(\t r)},{0*6.75*cos(\t r)+1*6.75*sin(\t r)});
\begin{scriptsize}
\draw[color=black] (1.58,4.14) node {$m$};
\end{scriptsize}
\end{tikzpicture} & \begin{tikzpicture}[scale=0.3]
\fill[dash pattern=on 3pt off 3pt,fill=black,fill opacity=0.1] (0,4) -- (-2.48,2.98) -- (-3.51,0.51) -- (-2.49,-1.97) -- (-0.02,-3.01) -- (2.46,-1.99) -- (3.49,0.49) -- (2.47,2.97) -- cycle;
\fill[fill=black,fill opacity=1.0] (1.41,2.37) -- (1.21,2.14) -- (1.32,1.85) -- (1.61,1.8) -- (1.81,2.03) -- (1.71,2.32) -- cycle;
\fill[fill=black,fill opacity=1.0] (-1.89,-0.91) -- (-1.66,-0.72) -- (-1.37,-0.82) -- (-1.32,-1.12) -- (-1.55,-1.31) -- (-1.84,-1.21) -- cycle;
\draw (0,4)-- (-2.48,2.98);
\draw [dash pattern=on 3pt off 3pt] (-2.48,2.98)-- (-3.51,0.51);
\draw (-3.51,0.51)-- (-2.49,-1.97);
\draw [dash pattern=on 3pt off 3pt] (-2.49,-1.97)-- (-0.02,-3.01);
\draw (-0.02,-3.01)-- (2.46,-1.99);
\draw [dash pattern=on 3pt off 3pt] (2.46,-1.99)-- (3.49,0.49);
\draw (3.49,0.49)-- (2.47,2.97);
\draw [dash pattern=on 3pt off 3pt] (2.47,2.97)-- (0,4);
\draw (0.21,4.5)-- (2.68,3.46);
\draw (2.68,3.46)-- (2.61,3.62);
\draw (2.68,3.46)-- (2.52,3.4);
\draw (0.21,4.5)-- (0.28,4.34);
\draw (0.21,4.5)-- (0.37,4.56);
\draw (1.41,2.37)-- (1.21,2.14);
\draw (1.21,2.14)-- (1.32,1.85);
\draw (1.32,1.85)-- (1.61,1.8);
\draw (1.61,1.8)-- (1.81,2.03);
\draw (1.81,2.03)-- (1.71,2.32);
\draw (1.71,2.32)-- (1.41,2.37);
\draw (-1.65,-0.75)-- (1.61,1.8);
\begin{scriptsize}
\draw[color=black] (1.58,4.14) node {$m$};
\end{scriptsize}
\end{tikzpicture} & \begin{tikzpicture}[scale=0.3]
\fill[dash pattern=on 3pt off 3pt,fill=black,fill opacity=0.1] (0,4) -- (-2.48,2.98) -- (-3.51,0.51) -- (-2.49,-1.97) -- (-0.02,-3.01) -- (2.46,-1.99) -- (3.49,0.49) -- (2.47,2.97) -- cycle;
\fill[fill=black,fill opacity=1.0] (1.41,2.37) -- (1.21,2.14) -- (1.32,1.85) -- (1.61,1.8) -- (1.81,2.03) -- (1.71,2.32) -- cycle;
\fill[fill=black,fill opacity=1.0] (-1.89,-0.91) -- (-1.66,-0.72) -- (-1.37,-0.82) -- (-1.32,-1.12) -- (-1.55,-1.31) -- (-1.84,-1.21) -- cycle;
\draw (0,4)-- (-2.48,2.98);
\draw [dash pattern=on 3pt off 3pt] (-2.48,2.98)-- (-3.51,0.51);
\draw (-3.51,0.51)-- (-2.49,-1.97);
\draw [dash pattern=on 3pt off 3pt] (-2.49,-1.97)-- (-0.02,-3.01);
\draw (-0.02,-3.01)-- (2.46,-1.99);
\draw [dash pattern=on 3pt off 3pt] (2.46,-1.99)-- (3.49,0.49);
\draw (3.49,0.49)-- (2.47,2.97);
\draw [dash pattern=on 3pt off 3pt] (2.47,2.97)-- (0,4);
\draw (0.21,4.5)-- (2.68,3.46);
\draw (2.68,3.46)-- (2.61,3.62);
\draw (2.68,3.46)-- (2.52,3.4);
\draw (0.21,4.5)-- (0.28,4.34);
\draw (0.21,4.5)-- (0.37,4.56);
\draw (1.41,2.37)-- (1.21,2.14);
\draw (1.21,2.14)-- (1.32,1.85);
\draw (1.32,1.85)-- (1.61,1.8);
\draw (1.61,1.8)-- (1.81,2.03);
\draw (1.81,2.03)-- (1.71,2.32);
\draw (1.71,2.32)-- (1.41,2.37);
\draw [shift={(2.12,-0.39)}] plot[domain=2.02:3.47,variable=\t]({1*4.88*cos(\t r)+0*4.88*sin(\t r)},{0*4.88*cos(\t r)+1*4.88*sin(\t r)});
\begin{scriptsize}
\draw[color=black] (1.58,4.14) node {$m$};
\end{scriptsize}
\end{tikzpicture} \\
\hline
\begin{tikzpicture}[scale=0.3]
\fill[dash pattern=on 3pt off 3pt,fill=black,fill opacity=0.1] (0,4) -- (-2.48,2.98) -- (-3.51,0.51) -- (-2.49,-1.97) -- (-0.02,-3.01) -- (2.46,-1.99) -- (3.49,0.49) -- (2.47,2.97) -- cycle;
\fill[fill=black,fill opacity=1.0] (1.41,2.37) -- (1.21,2.14) -- (1.32,1.85) -- (1.61,1.8) -- (1.81,2.03) -- (1.71,2.32) -- cycle;
\fill[fill=black,fill opacity=1.0] (-1.89,-0.91) -- (-1.66,-0.72) -- (-1.37,-0.82) -- (-1.32,-1.12) -- (-1.55,-1.31) -- (-1.84,-1.21) -- cycle;
\draw (0,4)-- (-2.48,2.98);
\draw [dash pattern=on 3pt off 3pt] (-2.48,2.98)-- (-3.51,0.51);
\draw (-3.51,0.51)-- (-2.49,-1.97);
\draw [dash pattern=on 3pt off 3pt] (-2.49,-1.97)-- (-0.02,-3.01);
\draw (-0.02,-3.01)-- (2.46,-1.99);
\draw [dash pattern=on 3pt off 3pt] (2.46,-1.99)-- (3.49,0.49);
\draw (3.49,0.49)-- (2.47,2.97);
\draw [dash pattern=on 3pt off 3pt] (2.47,2.97)-- (0,4);
\draw (0.21,4.5)-- (2.68,3.46);
\draw (2.68,3.46)-- (2.61,3.62);
\draw (2.68,3.46)-- (2.52,3.4);
\draw (0.21,4.5)-- (0.28,4.34);
\draw (0.21,4.5)-- (0.37,4.56);
\draw (1.41,2.37)-- (1.21,2.14);
\draw (1.21,2.14)-- (1.32,1.85);
\draw (1.32,1.85)-- (1.61,1.8);
\draw (1.61,1.8)-- (1.81,2.03);
\draw (1.81,2.03)-- (1.71,2.32);
\draw (1.71,2.32)-- (1.41,2.37);
\draw (0,4)-- (-0.02,-3.01);
\begin{scriptsize}
\draw[color=black] (1.58,4.14) node {$m$};
\end{scriptsize}
\end{tikzpicture} & Case $1$ & Case $2$ & Case $3$ & Case $4$ \\
\hline
\begin{tikzpicture}[scale=0.3]
\fill[dash pattern=on 3pt off 3pt,fill=black,fill opacity=0.1] (0,4) -- (-2.48,2.98) -- (-3.51,0.51) -- (-2.49,-1.97) -- (-0.02,-3.01) -- (2.46,-1.99) -- (3.49,0.49) -- (2.47,2.97) -- cycle;
\fill[fill=black,fill opacity=1.0] (1.41,2.37) -- (1.21,2.14) -- (1.32,1.85) -- (1.61,1.8) -- (1.81,2.03) -- (1.71,2.32) -- cycle;
\fill[fill=black,fill opacity=1.0] (-1.89,-0.91) -- (-1.66,-0.72) -- (-1.37,-0.82) -- (-1.32,-1.12) -- (-1.55,-1.31) -- (-1.84,-1.21) -- cycle;
\draw (0,4)-- (-2.48,2.98);
\draw [dash pattern=on 3pt off 3pt] (-2.48,2.98)-- (-3.51,0.51);
\draw (-3.51,0.51)-- (-2.49,-1.97);
\draw [dash pattern=on 3pt off 3pt] (-2.49,-1.97)-- (-0.02,-3.01);
\draw (-0.02,-3.01)-- (2.46,-1.99);
\draw [dash pattern=on 3pt off 3pt] (2.46,-1.99)-- (3.49,0.49);
\draw (3.49,0.49)-- (2.47,2.97);
\draw [dash pattern=on 3pt off 3pt] (2.47,2.97)-- (0,4);
\draw (0.21,4.5)-- (2.68,3.46);
\draw (2.68,3.46)-- (2.61,3.62);
\draw (2.68,3.46)-- (2.52,3.4);
\draw (0.21,4.5)-- (0.28,4.34);
\draw (0.21,4.5)-- (0.37,4.56);
\draw (1.41,2.37)-- (1.21,2.14);
\draw (1.21,2.14)-- (1.32,1.85);
\draw (1.32,1.85)-- (1.61,1.8);
\draw (1.61,1.8)-- (1.81,2.03);
\draw (1.81,2.03)-- (1.71,2.32);
\draw (1.71,2.32)-- (1.41,2.37);
\draw [shift={(-4.93,1.62)}] plot[domain=-0.76:0.06,variable=\t]({1*6.75*cos(\t r)+0*6.75*sin(\t r)},{0*6.75*cos(\t r)+1*6.75*sin(\t r)});
\begin{scriptsize}
\draw[color=black] (1.58,4.14) node {$m$};
\end{scriptsize}
\end{tikzpicture} & Case $2$ & Case $5$ & Case $6$ & Case $7$ \\
\hline
\begin{tikzpicture}[scale=0.3]
\fill[dash pattern=on 3pt off 3pt,fill=black,fill opacity=0.1] (0,4) -- (-2.48,2.98) -- (-3.51,0.51) -- (-2.49,-1.97) -- (-0.02,-3.01) -- (2.46,-1.99) -- (3.49,0.49) -- (2.47,2.97) -- cycle;
\fill[fill=black,fill opacity=1.0] (1.41,2.37) -- (1.21,2.14) -- (1.32,1.85) -- (1.61,1.8) -- (1.81,2.03) -- (1.71,2.32) -- cycle;
\fill[fill=black,fill opacity=1.0] (-1.89,-0.91) -- (-1.66,-0.72) -- (-1.37,-0.82) -- (-1.32,-1.12) -- (-1.55,-1.31) -- (-1.84,-1.21) -- cycle;
\draw (0,4)-- (-2.48,2.98);
\draw [dash pattern=on 3pt off 3pt] (-2.48,2.98)-- (-3.51,0.51);
\draw (-3.51,0.51)-- (-2.49,-1.97);
\draw [dash pattern=on 3pt off 3pt] (-2.49,-1.97)-- (-0.02,-3.01);
\draw (-0.02,-3.01)-- (2.46,-1.99);
\draw [dash pattern=on 3pt off 3pt] (2.46,-1.99)-- (3.49,0.49);
\draw (3.49,0.49)-- (2.47,2.97);
\draw [dash pattern=on 3pt off 3pt] (2.47,2.97)-- (0,4);
\draw (0.21,4.5)-- (2.68,3.46);
\draw (2.68,3.46)-- (2.61,3.62);
\draw (2.68,3.46)-- (2.52,3.4);
\draw (0.21,4.5)-- (0.28,4.34);
\draw (0.21,4.5)-- (0.37,4.56);
\draw (1.41,2.37)-- (1.21,2.14);
\draw (1.21,2.14)-- (1.32,1.85);
\draw (1.32,1.85)-- (1.61,1.8);
\draw (1.61,1.8)-- (1.81,2.03);
\draw (1.81,2.03)-- (1.71,2.32);
\draw (1.71,2.32)-- (1.41,2.37);
\draw (-1.65,-0.75)-- (1.61,1.8);
\begin{scriptsize}
\draw[color=black] (1.58,4.14) node {$m$};
\end{scriptsize}
\end{tikzpicture} & Case $3$ & Case $6$ & Case $8$ & Impossible \\
\hline
\begin{tikzpicture}[scale=0.3]
\fill[dash pattern=on 3pt off 3pt,fill=black,fill opacity=0.1] (0,4) -- (-2.48,2.98) -- (-3.51,0.51) -- (-2.49,-1.97) -- (-0.02,-3.01) -- (2.46,-1.99) -- (3.49,0.49) -- (2.47,2.97) -- cycle;
\fill[fill=black,fill opacity=1.0] (1.41,2.37) -- (1.21,2.14) -- (1.32,1.85) -- (1.61,1.8) -- (1.81,2.03) -- (1.71,2.32) -- cycle;
\fill[fill=black,fill opacity=1.0] (-1.89,-0.91) -- (-1.66,-0.72) -- (-1.37,-0.82) -- (-1.32,-1.12) -- (-1.55,-1.31) -- (-1.84,-1.21) -- cycle;
\draw (0,4)-- (-2.48,2.98);
\draw [dash pattern=on 3pt off 3pt] (-2.48,2.98)-- (-3.51,0.51);
\draw (-3.51,0.51)-- (-2.49,-1.97);
\draw [dash pattern=on 3pt off 3pt] (-2.49,-1.97)-- (-0.02,-3.01);
\draw (-0.02,-3.01)-- (2.46,-1.99);
\draw [dash pattern=on 3pt off 3pt] (2.46,-1.99)-- (3.49,0.49);
\draw (3.49,0.49)-- (2.47,2.97);
\draw [dash pattern=on 3pt off 3pt] (2.47,2.97)-- (0,4);
\draw (0.21,4.5)-- (2.68,3.46);
\draw (2.68,3.46)-- (2.61,3.62);
\draw (2.68,3.46)-- (2.52,3.4);
\draw (0.21,4.5)-- (0.28,4.34);
\draw (0.21,4.5)-- (0.37,4.56);
\draw (1.41,2.37)-- (1.21,2.14);
\draw (1.21,2.14)-- (1.32,1.85);
\draw (1.32,1.85)-- (1.61,1.8);
\draw (1.61,1.8)-- (1.81,2.03);
\draw (1.81,2.03)-- (1.71,2.32);
\draw (1.71,2.32)-- (1.41,2.37);
\draw [shift={(2.12,-0.39)}] plot[domain=2.02:3.47,variable=\t]({1*4.88*cos(\t r)+0*4.88*sin(\t r)},{0*4.88*cos(\t r)+1*4.88*sin(\t r)});
\begin{scriptsize}
\draw[color=black] (1.58,4.14) node {$m$};
\end{scriptsize}
\end{tikzpicture} & Case $4$ & Case $7$ & Impossible & Case $9$ \\
\hline
\end{tabular}
\end{center}

\vspace{10pt}

\cor{Here again, the cases are symmetric. We choose an orientation without loss of generality.}

First, we have to notice that cases $4$,$7$,$9$ are already treated from lemmas \ref{lem:cut} and \ref{lem:succ}.

Case $1$: $\alpha$ and $\beta$ are of type $1$ (cf figure \ref{fig:t1}).

\begin{figure}[!h]
\centering
\begin{tikzpicture}[scale=0.6]
\fill[fill=black,fill opacity=0.1] (0,4) -- (-1.66,3.46) -- (-2.69,2.05) -- (-2.68,0.3) -- (-1.66,-1.11) -- (0,-1.65) -- (1.66,-1.11) -- (2.69,0.3) -- (2.69,2.05) -- (1.66,3.46) -- cycle;
\draw [fill=black,fill opacity=1.0] (-1.66,1.57) circle (0.2cm);
\draw [fill=black,fill opacity=1.0] (1.66,1.58) circle (0.2cm);
\draw (0,4)-- (-1.66,3.46);
\draw (-1.66,3.46)-- (-2.69,2.05);
\draw (-2.69,2.05)-- (-2.68,0.3);
\draw (-2.68,0.3)-- (-1.66,-1.11);
\draw (-1.66,-1.11)-- (0,-1.65);
\draw (0,-1.65)-- (1.66,-1.11);
\draw (1.66,-1.11)-- (2.69,0.3);
\draw (2.69,0.3)-- (2.69,2.05);
\draw (2.69,2.05)-- (1.66,3.46);
\draw (1.66,3.46)-- (0,4);
\draw [shift={(-8.03,-1.44)}] plot[domain=-0.03:0.66,variable=\t]({1*8.03*cos(\t r)+0*8.03*sin(\t r)},{0*8.03*cos(\t r)+1*8.03*sin(\t r)});
\draw(-1.66,1.17) circle (0.4cm);
\draw(1.66,1.18) circle (0.4cm);
\draw [shift={(1.25,1.99)}] plot[domain=-2.33:0.04,variable=\t]({1*1.44*cos(\t r)+0*1.44*sin(\t r)},{0*1.44*cos(\t r)+1*1.44*sin(\t r)});
\draw [shift={(-2.6,-2.2)}] plot[domain=0.83:1.59,variable=\t]({1*4.25*cos(\t r)+0*4.25*sin(\t r)},{0*4.25*cos(\t r)+1*4.25*sin(\t r)});
\end{tikzpicture}
\caption{First case, $\alpha$ and $\beta$ are of type $1$}
\label{fig:t1}
\end{figure}

\cor{Let $i$ (respectively $j$) be the end of $\alpha$ (respectively $\beta$) such that $|j-i|$ is the minimal number of vertices of $P$ from an end of $\alpha$ to an end of $\beta$.}

\cor{We have $|j-i| \leq m+1$. Indeed, as $|j-i|$ is the minimal distance, and as $\alpha$ and $\beta$ cut $P$ into $4$ parts, the number of vertices of $P$ would be superior to $4m+4$. This is impossible since $P$ has $2m$ sides.}

\cor{Let $k=|j-i|$. Then one end of $\beta[k]$ is $i$, as $\alpha$. As they are of type $1$, and share an oriented angle, they are on the same slice of the Auslander-Reiten quiver. Then, this is a nonzero composition of arrows. In this way, we have found a nonzero morphism from $\alpha$ to $\beta[k]$. Then ${\mathrm{Ext}}_\mathcal{C}^k(X_\alpha,X_\beta) \neq 0$.}

\vspace{10pt}

Case $2$: $\alpha$ is of type $1$ and $\beta$ is of type $3$ (cf figure \ref{fig:t2}). Let $j$ be the end of $\beta$ on $P$, and let $i$ be the end of $\alpha$ which is closest to $i$ (in terms of vertices of $P$).

\begin{figure}[!h]
\centering
\begin{tikzpicture}[scale=0.6]
\fill[fill=black,fill opacity=0.1] (0,4) -- (-1.66,3.46) -- (-2.69,2.05) -- (-2.68,0.3) -- (-1.66,-1.11) -- (0,-1.65) -- (1.66,-1.11) -- (2.69,0.3) -- (2.69,2.05) -- (1.66,3.46) -- cycle;
\draw [fill=black,fill opacity=1.0] (-1.66,1.57) circle (0.2cm);
\draw [fill=black,fill opacity=1.0] (1.66,1.58) circle (0.2cm);
\draw (0,4)-- (-1.66,3.46);
\draw (-1.66,3.46)-- (-2.69,2.05);
\draw (-2.69,2.05)-- (-2.68,0.3);
\draw (-2.68,0.3)-- (-1.66,-1.11);
\draw (-1.66,-1.11)-- (0,-1.65);
\draw (0,-1.65)-- (1.66,-1.11);
\draw (1.66,-1.11)-- (2.69,0.3);
\draw (2.69,0.3)-- (2.69,2.05);
\draw (2.69,2.05)-- (1.66,3.46);
\draw (1.66,3.46)-- (0,4);
\draw [shift={(-8.03,-1.44)}] plot[domain=-0.03:0.66,variable=\t]({1*8.03*cos(\t r)+0*8.03*sin(\t r)},{0*8.03*cos(\t r)+1*8.03*sin(\t r)});
\draw(-1.66,1.17) circle (0.4cm);
\draw(1.66,1.18) circle (0.4cm);
\draw [shift={(1.98,-1.95)}] plot[domain=1.67:2.92,variable=\t]({1*3.74*cos(\t r)+0*3.74*sin(\t r)},{0*3.74*cos(\t r)+1*3.74*sin(\t r)});
\end{tikzpicture}
\caption{Case where $\alpha$ and $\beta$ are not of the same type}
\label{fig:t2}
\end{figure}

\cor{Here again, we have $|j-i| \leq m+1$. Indeed, if it was not the case, as $|j-i|$ is the minimal distance, the $m$-diagonal $\alpha$ would cut the polygon into a $l$-gon, the number $l$ begin strictly superior to $2m+2$. Then the number of vertices of $P$ would be superior to $4m+4$. This is impossible since $P$ has $2m$ sides.}

This case is similar to that of the first one. Let $k=|j-i|\leq m$. It suffices to shift $\beta$ $k<m$ times in order to hang both arcs to the same vertex. Consequently, they do not cross a mesh in the Auslander-Reiten quiver. Then, there is a Hom-hammock from one to another. Then there is a nonzero extension from $\alpha$ to $\beta$.

\begin{figure}[!h]
\centering
\begin{tikzpicture}[scale=0.3]
\fill[dash pattern=on 3pt off 3pt,fill=black,fill opacity=0.1] (0,4) -- (-2.48,2.98) -- (-3.51,0.51) -- (-2.49,-1.97) -- (-0.02,-3.01) -- (2.46,-1.99) -- (3.49,0.49) -- (2.47,2.97) -- cycle;
\fill[fill=black,fill opacity=1.0] (1.41,2.37) -- (1.21,2.14) -- (1.32,1.85) -- (1.61,1.8) -- (1.81,2.03) -- (1.71,2.32) -- cycle;
\fill[fill=black,fill opacity=1.0] (-1.89,-0.91) -- (-1.66,-0.72) -- (-1.37,-0.82) -- (-1.32,-1.12) -- (-1.55,-1.31) -- (-1.84,-1.21) -- cycle;
\draw (0,4)-- (-2.48,2.98);
\draw [dash pattern=on 3pt off 3pt] (-2.48,2.98)-- (-3.51,0.51);
\draw (-3.51,0.51)-- (-2.49,-1.97);
\draw [dash pattern=on 3pt off 3pt] (-2.49,-1.97)-- (-0.02,-3.01);
\draw (-0.02,-3.01)-- (2.46,-1.99);
\draw [dash pattern=on 3pt off 3pt] (2.46,-1.99)-- (3.49,0.49);
\draw (3.49,0.49)-- (2.47,2.97);
\draw [dash pattern=on 3pt off 3pt] (2.47,2.97)-- (0,4);
\draw (0.21,4.5)-- (2.68,3.46);
\draw (2.68,3.46)-- (2.61,3.62);
\draw (2.68,3.46)-- (2.52,3.4);
\draw (0.21,4.5)-- (0.28,4.34);
\draw (0.21,4.5)-- (0.37,4.56);
\draw (1.41,2.37)-- (1.21,2.14);
\draw (1.21,2.14)-- (1.32,1.85);
\draw (1.32,1.85)-- (1.61,1.8);
\draw (1.61,1.8)-- (1.81,2.03);
\draw (1.81,2.03)-- (1.71,2.32);
\draw (1.71,2.32)-- (1.41,2.37);
\draw [shift={(7.63,1.41)}]  plot[domain=-0.78:1,variable=\t]({-1*7.43*cos(\t r)+-0.09*1.54*sin(\t r)},{0.09*7.43*cos(\t r)+-1*1.54*sin(\t r)});
\draw [shift={(2.85,1.03)}] plot[domain=1.76:2.55,variable=\t]({1*1.98*cos(\t r)+0*1.98*sin(\t r)},{0*1.98*cos(\t r)+1*1.98*sin(\t r)});
\draw [shift={(1.91,0.49)}] plot[domain=1.35:4.93,variable=\t]({1*2.54*cos(\t r)+0*2.54*sin(\t r)},{0*2.54*cos(\t r)+1*2.54*sin(\t r)});
\draw [->] (-0.3,1.71) -- (0.42,1.69);
\draw [->] (1.4,2.79) -- (1.58,2.54);
\begin{scriptsize}
\draw[color=black] (1.58,4.14) node {$m$};
\draw[color=black] (2.16,2.63) node {$\beta$};
\draw[color=black] (-0.3,0.68) node {$\alpha$};
\end{scriptsize}
\end{tikzpicture}
\caption{The arrows are elementary move. The nonzero extension corresponds to the composition of the arrows}
\label{fig:emm}
\end{figure}

\vspace{10pt}

Case $3$: We are in the situation of figure \ref{fig:t21}

\begin{figure}[!h]
\centering
\begin{minipage}[t]{7cm}
\centering
\begin{tikzpicture}[scale=0.5]
\fill[dash pattern=on 3pt off 3pt,fill=black,fill opacity=0.1] (0,4) -- (-2.48,2.98) -- (-3.51,0.51) -- (-2.49,-1.97) -- (-0.02,-3.01) -- (2.46,-1.99) -- (3.49,0.49) -- (2.47,2.97) -- cycle;
\fill[fill=black,fill opacity=1.0] (1.41,2.37) -- (1.21,2.14) -- (1.32,1.85) -- (1.61,1.8) -- (1.81,2.03) -- (1.71,2.32) -- cycle;
\fill[fill=black,fill opacity=1.0] (-1.89,-0.91) -- (-1.66,-0.72) -- (-1.37,-0.82) -- (-1.32,-1.12) -- (-1.55,-1.31) -- (-1.84,-1.21) -- cycle;
\draw (0,4)-- (-2.48,2.98);
\draw [dash pattern=on 3pt off 3pt] (-2.48,2.98)-- (-3.51,0.51);
\draw (-3.51,0.51)-- (-2.49,-1.97);
\draw [dash pattern=on 3pt off 3pt] (-2.49,-1.97)-- (-0.02,-3.01);
\draw (-0.02,-3.01)-- (2.46,-1.99);
\draw [dash pattern=on 3pt off 3pt] (2.46,-1.99)-- (3.49,0.49);
\draw (3.49,0.49)-- (2.47,2.97);
\draw [dash pattern=on 3pt off 3pt] (2.47,2.97)-- (0,4);
\draw (0.21,4.5)-- (2.68,3.46);
\draw (2.68,3.46)-- (2.61,3.62);
\draw (2.68,3.46)-- (2.52,3.4);
\draw (0.21,4.5)-- (0.28,4.34);
\draw (0.21,4.5)-- (0.37,4.56);
\draw (1.41,2.37)-- (1.21,2.14);
\draw (1.21,2.14)-- (1.32,1.85);
\draw (1.32,1.85)-- (1.61,1.8);
\draw (1.61,1.8)-- (1.81,2.03);
\draw (1.81,2.03)-- (1.71,2.32);
\draw (1.71,2.32)-- (1.41,2.37);
\draw (-1.36,-0.85)-- (1.61,1.8);
\draw (-2.48,2.98)-- (2.46,-1.99);
\begin{scriptsize}
\draw[color=black] (1.58,4.14) node {$m$};
\end{scriptsize}
\end{tikzpicture}
\caption{Case where $\alpha$ is of type $1$ and $\beta$ is in the tube}
\label{fig:t21}
\end{minipage}
\begin{minipage}[t]{7cm}
\centering
\begin{tikzpicture}[scale=0.5]
\fill[dash pattern=on 3pt off 3pt,fill=black,fill opacity=0.1] (0,4) -- (-2.48,2.98) -- (-3.51,0.51) -- (-2.49,-1.97) -- (-0.02,-3.01) -- (2.46,-1.99) -- (3.49,0.49) -- (2.47,2.97) -- cycle;
\fill[fill=black,fill opacity=1.0] (1.41,2.37) -- (1.21,2.14) -- (1.32,1.85) -- (1.61,1.8) -- (1.81,2.03) -- (1.71,2.32) -- cycle;
\fill[fill=black,fill opacity=1.0] (-1.89,-0.91) -- (-1.66,-0.72) -- (-1.37,-0.82) -- (-1.32,-1.12) -- (-1.55,-1.31) -- (-1.84,-1.21) -- cycle;
\draw (0,4)-- (-2.48,2.98);
\draw [dash pattern=on 3pt off 3pt] (-2.48,2.98)-- (-3.51,0.51);
\draw (-3.51,0.51)-- (-2.49,-1.97);
\draw [dash pattern=on 3pt off 3pt] (-2.49,-1.97)-- (-0.02,-3.01);
\draw (-0.02,-3.01)-- (2.46,-1.99);
\draw [dash pattern=on 3pt off 3pt] (2.46,-1.99)-- (3.49,0.49);
\draw (3.49,0.49)-- (2.47,2.97);
\draw [dash pattern=on 3pt off 3pt] (2.47,2.97)-- (0,4);
\draw (0.21,4.5)-- (2.68,3.46);
\draw (2.68,3.46)-- (2.61,3.62);
\draw (2.68,3.46)-- (2.52,3.4);
\draw (0.21,4.5)-- (0.28,4.34);
\draw (0.21,4.5)-- (0.37,4.56);
\draw (1.41,2.37)-- (1.21,2.14);
\draw (1.21,2.14)-- (1.32,1.85);
\draw (1.32,1.85)-- (1.61,1.8);
\draw (1.61,1.8)-- (1.81,2.03);
\draw (1.81,2.03)-- (1.71,2.32);
\draw (1.71,2.32)-- (1.41,2.37);
\draw [color=red] (-2.48,2.98)-- (2.46,-1.99);
\draw [shift={(0.84,6.8)}] plot[domain=4:4.87,variable=\t]({1*5.06*cos(\t r)+0*5.06*sin(\t r)},{0*5.06*cos(\t r)+1*5.06*sin(\t r)});
\draw [shift={(-0.51,-4.98)}] plot[domain=0.79:1.77,variable=\t]({1*4.21*cos(\t r)+0*4.21*sin(\t r)},{0*4.21*cos(\t r)+1*4.21*sin(\t r)});
\begin{scriptsize}
\draw[color=black] (1.58,4.14) node {$m$};
\end{scriptsize}
\end{tikzpicture}
\caption{}
\label{fig:case3}
\end{minipage}
\end{figure}

In this case, it is more difficult to see morphisms in the Auslander-Reiten quiver of $Q$ because one arc is in the transjective component and the other is in a tube. Nonetheless, if we can find an $(m+2)$-angulation where $\beta$ is the $i$-th twist of $\alpha$, then from lemma \ref{lem:morph}, there is an extension which is nonzero. We have to complete $\alpha$ into an $(m+2)$-angulation containing this arc (see figure \ref{fig:case3}):

As $\beta$ is the $i$-twist of $\alpha$, then there exists $i \in \{1,\cdots,m\}$ such that $\mathrm{Ext}_\mathcal{C}^i(X_\alpha,X_\beta) \neq 0$.

The case where $\alpha$ is in the tube and $\beta$ is of type $1$ is similar.

\vspace{10pt}

\cor{Case $4$: In the case of figure \ref{fig:case4}:}

\begin{figure}[!h]
\centering
\begin{minipage}[t]{7cm}
\centering
\begin{tikzpicture}[scale=0.5]
\fill[dash pattern=on 3pt off 3pt,fill=black,fill opacity=0.1] (0,4) -- (-2.48,2.98) -- (-3.51,0.51) -- (-2.49,-1.97) -- (-0.02,-3.01) -- (2.46,-1.99) -- (3.49,0.49) -- (2.47,2.97) -- cycle;
\fill[fill=black,fill opacity=1.0] (1.41,2.37) -- (1.21,2.14) -- (1.32,1.85) -- (1.61,1.8) -- (1.81,2.03) -- (1.71,2.32) -- cycle;
\fill[fill=black,fill opacity=1.0] (-1.89,-0.91) -- (-1.66,-0.72) -- (-1.37,-0.82) -- (-1.32,-1.12) -- (-1.55,-1.31) -- (-1.84,-1.21) -- cycle;
\draw (0,4)-- (-2.48,2.98);
\draw [dash pattern=on 3pt off 3pt] (-2.48,2.98)-- (-3.51,0.51);
\draw (-3.51,0.51)-- (-2.49,-1.97);
\draw [dash pattern=on 3pt off 3pt] (-2.49,-1.97)-- (-0.02,-3.01);
\draw (-0.02,-3.01)-- (2.46,-1.99);
\draw [dash pattern=on 3pt off 3pt] (2.46,-1.99)-- (3.49,0.49);
\draw (3.49,0.49)-- (2.47,2.97);
\draw [dash pattern=on 3pt off 3pt] (2.47,2.97)-- (0,4);
\draw (0.21,4.5)-- (2.68,3.46);
\draw (2.68,3.46)-- (2.61,3.62);
\draw (2.68,3.46)-- (2.52,3.4);
\draw (0.21,4.5)-- (0.28,4.34);
\draw (0.21,4.5)-- (0.37,4.56);
\draw (1.41,2.37)-- (1.21,2.14);
\draw (1.21,2.14)-- (1.32,1.85);
\draw (1.32,1.85)-- (1.61,1.8);
\draw (1.61,1.8)-- (1.81,2.03);
\draw (1.81,2.03)-- (1.71,2.32);
\draw (1.71,2.32)-- (1.41,2.37);
\draw (-2.48,2.98)-- (2.46,-1.99);
\draw [shift={(1.83,-0.27)}] plot[domain=1.98:3.52,variable=\t]({1*4.65*cos(\t r)+0*4.65*sin(\t r)},{0*4.65*cos(\t r)+1*4.65*sin(\t r)});
\begin{scriptsize}
\draw[color=black] (1.58,4.14) node {$m$};
\end{scriptsize}
\end{tikzpicture}
\caption{Case $4$, $\alpha$ is of type $1$ and $\beta$ is in the tube of size $m-1$.}
\label{fig:case4}
\end{minipage}
\begin{minipage}[t]{7cm}
\centering
\begin{tikzpicture}[scale=0.5]
\fill[dash pattern=on 3pt off 3pt,fill=black,fill opacity=0.1] (0,4) -- (-2.48,2.98) -- (-3.51,0.51) -- (-2.49,-1.97) -- (-0.02,-3.01) -- (2.46,-1.99) -- (3.49,0.49) -- (2.47,2.97) -- cycle;
\fill[fill=black,fill opacity=1.0] (1.41,2.37) -- (1.21,2.14) -- (1.32,1.85) -- (1.61,1.8) -- (1.81,2.03) -- (1.71,2.32) -- cycle;
\fill[fill=black,fill opacity=1.0] (-1.89,-0.91) -- (-1.66,-0.72) -- (-1.37,-0.82) -- (-1.32,-1.12) -- (-1.55,-1.31) -- (-1.84,-1.21) -- cycle;
\draw (0,4)-- (-2.48,2.98);
\draw [dash pattern=on 3pt off 3pt] (-2.48,2.98)-- (-3.51,0.51);
\draw (-3.51,0.51)-- (-2.49,-1.97);
\draw [dash pattern=on 3pt off 3pt] (-2.49,-1.97)-- (-0.02,-3.01);
\draw (-0.02,-3.01)-- (2.46,-1.99);
\draw [dash pattern=on 3pt off 3pt] (2.46,-1.99)-- (3.49,0.49);
\draw (3.49,0.49)-- (2.47,2.97);
\draw [dash pattern=on 3pt off 3pt] (2.47,2.97)-- (0,4);
\draw (0.21,4.5)-- (2.68,3.46);
\draw (2.68,3.46)-- (2.61,3.62);
\draw (2.68,3.46)-- (2.52,3.4);
\draw (0.21,4.5)-- (0.28,4.34);
\draw (0.21,4.5)-- (0.37,4.56);
\draw (1.41,2.37)-- (1.21,2.14);
\draw (1.21,2.14)-- (1.32,1.85);
\draw (1.32,1.85)-- (1.61,1.8);
\draw (1.61,1.8)-- (1.81,2.03);
\draw (1.81,2.03)-- (1.71,2.32);
\draw (1.71,2.32)-- (1.41,2.37);
\draw [color=red] (-2.48,2.98)-- (2.46,-1.99);
\draw [shift={(7.67,-12.74)}] plot[domain=-0.75:0.39,variable=\t]({-0.61*15.96*cos(\t r)+-0.79*3.66*sin(\t r)},{0.79*15.96*cos(\t r)+-0.61*3.66*sin(\t r)});
\begin{scriptsize}
\draw[color=black] (1.58,4.14) node {$m$};
\end{scriptsize}
\end{tikzpicture}
\caption{}
\label{fig:case4bis}
\end{minipage}
\end{figure}

\cor{Then we use the same argument as in case $3$, if we find an $(m+2)$-angulation where $\beta$ is the flip of $\alpha$, then there exists a morphism between them.}

\cor{In this case, we have to take an $(m+2)$-angulation containing this arc (see figure \ref{fig:case4bis}). Let $a$ and $b$ (respectively $c $and $d$) be the ends of $\alpha$ (respectively $\beta$).}

\cor{If $\beta$ is not minimal, then we cut along an $m$-ear in order to reduce to the case where $\beta$ is minimal. If the arc from $a$ to $c$ is admissible, then we draw it as in figure \ref{fig:case4bis}. Else, let $e$ be the smallest integer such that the arc from $a$ to $c+e$ is admissible. Then we flip $\alpha$ as many times as necessary in order to get $\beta$.}

Then there exists a nonzero extension between $X_\alpha$ and $X_\beta$. The inverse case is similar.

\vspace{10pt}

Case $5$: If both $\alpha$ and $\beta$ are of type $3$ (cf figure \ref{fig:t3}). Let $i$ be the end of $\alpha$ on $P$ and let $j$ be the end of $\beta$ on $P$.

\begin{figure}[!h]
\centering
\begin{tikzpicture}[scale=0.6]
\fill[fill=black,fill opacity=0.1] (0,4) -- (-1.66,3.46) -- (-2.69,2.05) -- (-2.68,0.3) -- (-1.66,-1.11) -- (0,-1.65) -- (1.66,-1.11) -- (2.69,0.3) -- (2.69,2.05) -- (1.66,3.46) -- cycle;
\draw [fill=black,fill opacity=1.0] (-1.66,1.57) circle (0.2cm);
\draw [fill=black,fill opacity=1.0] (1.66,1.58) circle (0.2cm);
\draw (0,4)-- (-1.66,3.46);
\draw (-1.66,3.46)-- (-2.69,2.05);
\draw (-2.69,2.05)-- (-2.68,0.3);
\draw (-2.68,0.3)-- (-1.66,-1.11);
\draw (-1.66,-1.11)-- (0,-1.65);
\draw (0,-1.65)-- (1.66,-1.11);
\draw (1.66,-1.11)-- (2.69,0.3);
\draw (2.69,0.3)-- (2.69,2.05);
\draw (2.69,2.05)-- (1.66,3.46);
\draw (1.66,3.46)-- (0,4);
\draw(-1.66,1.17) circle (0.4cm);
\draw(1.66,1.18) circle (0.4cm);
\draw [shift={(1.91,-1.86)}] plot[domain=1.64:2.93,variable=\t]({1*3.65*cos(\t r)+0*3.65*sin(\t r)},{0*3.65*cos(\t r)+1*3.65*sin(\t r)});
\draw [shift={(-1.95,-1.92)}] plot[domain=0.22:1.49,variable=\t]({1*3.7*cos(\t r)+0*3.7*sin(\t r)},{0*3.7*cos(\t r)+1*3.7*sin(\t r)});
\end{tikzpicture}
\caption{Case where both $\alpha$ and $\beta$ are of other type}
\label{fig:t3}
\end{figure}


\cor{Let $k=|j-i|$. We can move $\beta$ $k<m$ times in order to have that the end of $\beta[k]$ hung to $P$ is $i$. Then the composition of elementary moves in figure \ref{fig:type5} is not zero since it follows a slice of the Auslander-Reiten quiver (so do not cross a mesh).}

Then there is a nonzero extension between $\alpha$ and $\beta$.

\begin{figure}[!h]
\centering
\begin{tikzpicture}[scale=0.5]
\fill[dash pattern=on 3pt off 3pt,fill=black,fill opacity=0.1] (0,4) -- (-2.48,2.98) -- (-3.51,0.51) -- (-2.49,-1.97) -- (-0.02,-3.01) -- (2.46,-1.99) -- (3.49,0.49) -- (2.47,2.97) -- cycle;
\fill[fill=black,fill opacity=1.0] (1.41,2.37) -- (1.21,2.14) -- (1.32,1.85) -- (1.61,1.8) -- (1.81,2.03) -- (1.71,2.32) -- cycle;
\fill[fill=black,fill opacity=1.0] (-1.89,-0.91) -- (-1.66,-0.72) -- (-1.37,-0.82) -- (-1.32,-1.12) -- (-1.55,-1.31) -- (-1.84,-1.21) -- cycle;
\draw (0,4)-- (-2.48,2.98);
\draw [dash pattern=on 3pt off 3pt] (-2.48,2.98)-- (-3.51,0.51);
\draw (-3.51,0.51)-- (-2.49,-1.97);
\draw [dash pattern=on 3pt off 3pt] (-2.49,-1.97)-- (-0.02,-3.01);
\draw (-0.02,-3.01)-- (2.46,-1.99);
\draw [dash pattern=on 3pt off 3pt] (2.46,-1.99)-- (3.49,0.49);
\draw (3.49,0.49)-- (2.47,2.97);
\draw [dash pattern=on 3pt off 3pt] (2.47,2.97)-- (0,4);
\draw (0.21,4.5)-- (2.68,3.46);
\draw (2.68,3.46)-- (2.61,3.62);
\draw (2.68,3.46)-- (2.52,3.4);
\draw (0.21,4.5)-- (0.28,4.34);
\draw (0.21,4.5)-- (0.37,4.56);
\draw (1.41,2.37)-- (1.21,2.14);
\draw (1.21,2.14)-- (1.32,1.85);
\draw (1.32,1.85)-- (1.61,1.8);
\draw (1.61,1.8)-- (1.81,2.03);
\draw (1.81,2.03)-- (1.71,2.32);
\draw (1.71,2.32)-- (1.41,2.37);
\draw [shift={(10.32,2.65)}] plot[domain=3.2:3.67,variable=\t]({1*9.12*cos(\t r)+0*9.12*sin(\t r)},{0*9.12*cos(\t r)+1*9.12*sin(\t r)});
\draw [shift={(-1.05,-6.59)}] plot[domain=0.92:1.63,variable=\t]({1*5.79*cos(\t r)+0*5.79*sin(\t r)},{0*5.79*cos(\t r)+1*5.79*sin(\t r)});
\draw (2.46,-1.99)-- (-2.48,2.98);
\begin{scriptsize}
\draw[color=black] (1.57,4.14) node {$m$};
\end{scriptsize}
\end{tikzpicture}
\caption{}
\label{fig:type5}
\end{figure}

\vspace{10pt}

Case $6$: If we are in the situation of figure \ref{fig:case6}: It means that $\alpha$ is of type $3$ and $\beta$ is of type $4$.

\begin{figure}[!h]
\centering
\begin{minipage}[t]{7cm}
\centering
\begin{tikzpicture}[scale=0.5]
\fill[dash pattern=on 3pt off 3pt,fill=black,fill opacity=0.1] (0,4) -- (-2.48,2.98) -- (-3.51,0.51) -- (-2.49,-1.97) -- (-0.02,-3.01) -- (2.46,-1.99) -- (3.49,0.49) -- (2.47,2.97) -- cycle;
\fill[fill=black,fill opacity=1.0] (1.41,2.37) -- (1.21,2.14) -- (1.32,1.85) -- (1.61,1.8) -- (1.81,2.03) -- (1.71,2.32) -- cycle;
\fill[fill=black,fill opacity=1.0] (-1.89,-0.91) -- (-1.66,-0.72) -- (-1.37,-0.82) -- (-1.32,-1.12) -- (-1.55,-1.31) -- (-1.84,-1.21) -- cycle;
\draw (0,4)-- (-2.48,2.98);
\draw [dash pattern=on 3pt off 3pt] (-2.48,2.98)-- (-3.51,0.51);
\draw (-3.51,0.51)-- (-2.49,-1.97);
\draw [dash pattern=on 3pt off 3pt] (-2.49,-1.97)-- (-0.02,-3.01);
\draw (-0.02,-3.01)-- (2.46,-1.99);
\draw [dash pattern=on 3pt off 3pt] (2.46,-1.99)-- (3.49,0.49);
\draw (3.49,0.49)-- (2.47,2.97);
\draw [dash pattern=on 3pt off 3pt] (2.47,2.97)-- (0,4);
\draw (0.21,4.5)-- (2.68,3.46);
\draw (2.68,3.46)-- (2.61,3.62);
\draw (2.68,3.46)-- (2.52,3.4);
\draw (0.21,4.5)-- (0.28,4.34);
\draw (0.21,4.5)-- (0.37,4.56);
\draw (1.41,2.37)-- (1.21,2.14);
\draw (1.21,2.14)-- (1.32,1.85);
\draw (1.32,1.85)-- (1.61,1.8);
\draw (1.61,1.8)-- (1.81,2.03);
\draw (1.81,2.03)-- (1.71,2.32);
\draw (1.71,2.32)-- (1.41,2.37);
\draw [shift={(-2.07,5.59)}] plot[domain=4.44:5.48,variable=\t]({1*5.28*cos(\t r)+0*5.28*sin(\t r)},{0*5.28*cos(\t r)+1*5.28*sin(\t r)});
\draw (1.32,1.85)-- (-1.35,-0.83);
\begin{scriptsize}
\draw[color=black] (1.58,4.14) node {$m$};
\end{scriptsize}
\end{tikzpicture}
\caption{Case where $\alpha$ is of type $3$ and $X_\beta$ is in a tube of size $2$}
\label{fig:case6}
\end{minipage}
\begin{minipage}[t]{7cm}
\centering
\begin{tikzpicture}[scale=0.5]
\fill[dash pattern=on 3pt off 3pt,fill=black,fill opacity=0.1] (0,4) -- (-2.48,2.98) -- (-3.51,0.51) -- (-2.49,-1.97) -- (-0.02,-3.01) -- (2.46,-1.99) -- (3.49,0.49) -- (2.47,2.97) -- cycle;
\fill[fill=black,fill opacity=1.0] (1.41,2.37) -- (1.21,2.14) -- (1.32,1.85) -- (1.61,1.8) -- (1.81,2.03) -- (1.71,2.32) -- cycle;
\fill[fill=black,fill opacity=1.0] (-1.89,-0.91) -- (-1.66,-0.72) -- (-1.37,-0.82) -- (-1.32,-1.12) -- (-1.55,-1.31) -- (-1.84,-1.21) -- cycle;
\draw (0,4)-- (-2.48,2.98);
\draw [dash pattern=on 3pt off 3pt] (-2.48,2.98)-- (-3.51,0.51);
\draw (-3.51,0.51)-- (-2.49,-1.97);
\draw [dash pattern=on 3pt off 3pt] (-2.49,-1.97)-- (-0.02,-3.01);
\draw (-0.02,-3.01)-- (2.46,-1.99);
\draw [dash pattern=on 3pt off 3pt] (2.46,-1.99)-- (3.49,0.49);
\draw (3.49,0.49)-- (2.47,2.97);
\draw [dash pattern=on 3pt off 3pt] (2.47,2.97)-- (0,4);
\draw (0.21,4.5)-- (2.68,3.46);
\draw (2.68,3.46)-- (2.61,3.62);
\draw (2.68,3.46)-- (2.52,3.4);
\draw (0.21,4.5)-- (0.28,4.34);
\draw (0.21,4.5)-- (0.37,4.56);
\draw (1.41,2.37)-- (1.21,2.14);
\draw (1.21,2.14)-- (1.32,1.85);
\draw (1.32,1.85)-- (1.61,1.8);
\draw (1.61,1.8)-- (1.81,2.03);
\draw (1.81,2.03)-- (1.71,2.32);
\draw (1.71,2.32)-- (1.41,2.37);
\draw [shift={(-2.07,5.59)},color=red]  plot[domain=4.44:5.48,variable=\t]({1*5.28*cos(\t r)+0*5.28*sin(\t r)},{0*5.28*cos(\t r)+1*5.28*sin(\t r)});
\draw [shift={(-5.25,-4.64)}] plot[domain=0.77:1.25,variable=\t]({1*5.43*cos(\t r)+0*5.43*sin(\t r)},{0*5.43*cos(\t r)+1*5.43*sin(\t r)});
\begin{scriptsize}
\draw[color=black] (1.58,4.14) node {$m$};
\end{scriptsize}
\end{tikzpicture}
\caption{}
\label{fig:case6bis}
\end{minipage}
\end{figure}

The same arguments as in case $3$ lead to find an $(m+2)$-angulation containing these arcs in figure \ref{fig:case6bis}.

Here again, there exists a nonzero extension between $X_\alpha$ and $X_\beta$. The inverse case is similar.

\vspace{10pt}

\cor{Case $7$: If we are in the situation of figure \ref{fig:case7}: It means that $\alpha$ is of type $3$ and $\beta$ is of type $2$.}

\begin{figure}[!ht]
\centering
\begin{minipage}[t]{7cm}
\centering
\begin{tikzpicture}[scale=0.5]
\fill[dash pattern=on 3pt off 3pt,fill=black,fill opacity=0.1] (0,4) -- (-2.48,2.98) -- (-3.51,0.51) -- (-2.49,-1.97) -- (-0.02,-3.01) -- (2.46,-1.99) -- (3.49,0.49) -- (2.47,2.97) -- cycle;
\fill[fill=black,fill opacity=1.0] (1.41,2.37) -- (1.21,2.14) -- (1.32,1.85) -- (1.61,1.8) -- (1.81,2.03) -- (1.71,2.32) -- cycle;
\fill[fill=black,fill opacity=1.0] (-1.89,-0.91) -- (-1.66,-0.72) -- (-1.37,-0.82) -- (-1.32,-1.12) -- (-1.55,-1.31) -- (-1.84,-1.21) -- cycle;
\draw (0,4)-- (-2.48,2.98);
\draw [dash pattern=on 3pt off 3pt] (-2.48,2.98)-- (-3.51,0.51);
\draw (-3.51,0.51)-- (-2.49,-1.97);
\draw [dash pattern=on 3pt off 3pt] (-2.49,-1.97)-- (-0.02,-3.01);
\draw (-0.02,-3.01)-- (2.46,-1.99);
\draw [dash pattern=on 3pt off 3pt] (2.46,-1.99)-- (3.49,0.49);
\draw (3.49,0.49)-- (2.47,2.97);
\draw [dash pattern=on 3pt off 3pt] (2.47,2.97)-- (0,4);
\draw (0.21,4.5)-- (2.68,3.46);
\draw (2.68,3.46)-- (2.61,3.62);
\draw (2.68,3.46)-- (2.52,3.4);
\draw (0.21,4.5)-- (0.28,4.34);
\draw (0.21,4.5)-- (0.37,4.56);
\draw (1.41,2.37)-- (1.21,2.14);
\draw (1.21,2.14)-- (1.32,1.85);
\draw (1.32,1.85)-- (1.61,1.8);
\draw (1.61,1.8)-- (1.81,2.03);
\draw (1.81,2.03)-- (1.71,2.32);
\draw (1.71,2.32)-- (1.41,2.37);
\draw [shift={(-4.02,1.29)}] plot[domain=-0.82:0.13,variable=\t]({1*5.88*cos(\t r)+0*5.88*sin(\t r)},{0*5.88*cos(\t r)+1*5.88*sin(\t r)});
\draw [shift={(0.64,2.04)}] plot[domain=3.5:5.14,variable=\t]({1*4.42*cos(\t r)+0*4.42*sin(\t r)},{0*4.42*cos(\t r)+1*4.42*sin(\t r)});
\begin{scriptsize}
\draw[color=black] (1.58,4.14) node {$m$};
\end{scriptsize}
\end{tikzpicture}
\caption{Case where $\alpha$ is of type $3$ and $X_\beta$ is in a tube of size $m-1$}
\label{fig:case7}
\end{minipage}
\begin{minipage}[t]{7cm}
\centering
\begin{tikzpicture}[scale=0.5]
\fill[dash pattern=on 3pt off 3pt,fill=black,fill opacity=0.1] (0,4) -- (-2.48,2.98) -- (-3.51,0.51) -- (-2.49,-1.97) -- (-0.02,-3.01) -- (2.46,-1.99) -- (3.49,0.49) -- (2.47,2.97) -- cycle;
\fill[fill=black,fill opacity=1.0] (1.41,2.37) -- (1.21,2.14) -- (1.32,1.85) -- (1.61,1.8) -- (1.81,2.03) -- (1.71,2.32) -- cycle;
\fill[fill=black,fill opacity=1.0] (-1.89,-0.91) -- (-1.66,-0.72) -- (-1.37,-0.82) -- (-1.32,-1.12) -- (-1.55,-1.31) -- (-1.84,-1.21) -- cycle;
\draw (0,4)-- (-2.48,2.98);
\draw [dash pattern=on 3pt off 3pt] (-2.48,2.98)-- (-3.51,0.51);
\draw (-3.51,0.51)-- (-2.49,-1.97);
\draw [dash pattern=on 3pt off 3pt] (-2.49,-1.97)-- (-0.02,-3.01);
\draw (-0.02,-3.01)-- (2.46,-1.99);
\draw [dash pattern=on 3pt off 3pt] (2.46,-1.99)-- (3.49,0.49);
\draw (3.49,0.49)-- (2.47,2.97);
\draw [dash pattern=on 3pt off 3pt] (2.47,2.97)-- (0,4);
\draw (0.21,4.5)-- (2.68,3.46);
\draw (2.68,3.46)-- (2.61,3.62);
\draw (2.68,3.46)-- (2.52,3.4);
\draw (0.21,4.5)-- (0.28,4.34);
\draw (0.21,4.5)-- (0.37,4.56);
\draw (1.41,2.37)-- (1.21,2.14);
\draw (1.21,2.14)-- (1.32,1.85);
\draw (1.32,1.85)-- (1.61,1.8);
\draw (1.61,1.8)-- (1.81,2.03);
\draw (1.81,2.03)-- (1.71,2.32);
\draw (1.71,2.32)-- (1.41,2.37);
\draw [shift={(-4.02,1.29)},color=red]  plot[domain=-0.82:0.13,variable=\t]({1*5.88*cos(\t r)+0*5.88*sin(\t r)},{0*5.88*cos(\t r)+1*5.88*sin(\t r)});
\draw [shift={(-2.88,-0.79)}] plot[domain=-0.22:0.54,variable=\t]({1*5.47*cos(\t r)+0*5.47*sin(\t r)},{0*5.47*cos(\t r)+1*5.47*sin(\t r)});
\draw [shift={(2.92,3.41)}] plot[domain=3.57:4.28,variable=\t]({1*7.05*cos(\t r)+0*7.05*sin(\t r)},{0*7.05*cos(\t r)+1*7.05*sin(\t r)});
\begin{scriptsize}
\draw[color=black] (1.58,4.14) node {$m$};
\end{scriptsize}
\end{tikzpicture}
\caption{}
\label{fig:case7bis}
\end{minipage}
\end{figure}

\cor{The same arguments as case $5$ lead to find an $(m+2)$-angulation containing these arcs in figure \ref{fig:case7bis}.}

\cor{Here again, there exists a morphism between $\alpha$ and $\beta[k]$ for some $k \in \{1,\cdots,m\}$. The inverse case is similar.}

\vspace{10pt}

Case $8$: If both $m$-diagonals are of type $4$, it means that we focus on a tube of size $2$ (cf figure \ref{fig:t5}).

\begin{figure}[!h]
\centering
\begin{tikzpicture}[scale=0.2]
\fill[fill=black,fill opacity=0.15] (0,4) -- (0,-4) -- (6.93,-8) -- (13.86,-4) -- (13.86,4) -- (6.93,8) -- cycle;
\draw [fill=black,fill opacity=1.0] (3.46,2) circle (0.4cm);
\draw [fill=black,fill opacity=1.0] (3.46,-2) circle (0.4cm);
\draw [fill=black,fill opacity=1.0] (10.39,2) circle (0.4cm);
\draw [fill=black,fill opacity=1.0] (10.39,-2) circle (0.4cm);
\draw (0,4)-- (0,-4);
\draw (0,-4)-- (6.93,-8);
\draw (6.93,-8)-- (13.86,-4);
\draw (13.86,-4)-- (13.86,4);
\draw (13.86,4)-- (6.93,8);
\draw (6.93,8)-- (0,4);
\draw [shift={(1.86,0)}] plot[domain=-0.89:0.89,variable=\t]({1*2.57*cos(\t r)+0*2.57*sin(\t r)},{0*2.57*cos(\t r)+1*2.57*sin(\t r)});
\draw [shift={(5.07,0)}] plot[domain=2.25:4.04,variable=\t]({1*2.57*cos(\t r)+0*2.57*sin(\t r)},{0*2.57*cos(\t r)+1*2.57*sin(\t r)});
\draw [shift={(12,0)}] plot[domain=2.25:4.04,variable=\t]({1*2.57*cos(\t r)+0*2.57*sin(\t r)},{0*2.57*cos(\t r)+1*2.57*sin(\t r)});
\draw [shift={(8.79,0)}] plot[domain=-0.89:0.89,variable=\t]({1*2.57*cos(\t r)+0*2.57*sin(\t r)},{0*2.57*cos(\t r)+1*2.57*sin(\t r)});
\draw [shift={(-0.01,-9.55)}] plot[domain=0.62:1.28,variable=\t]({1*12.45*cos(\t r)+0*12.45*sin(\t r)},{0*12.45*cos(\t r)+1*12.45*sin(\t r)});
\draw [shift={(13.17,-8.59)}] plot[domain=1.82:2.55,variable=\t]({1*11.33*cos(\t r)+0*11.33*sin(\t r)},{0*11.33*cos(\t r)+1*11.33*sin(\t r)});
\end{tikzpicture}
\caption{Case of a tube of size $2$, for example for $m=3$}
\label{fig:t5}
\end{figure}

\cor{If $\alpha$ and $\beta$ cross each other, it means that they are in the same tube. Then one is situated higher than the other and there exists a Hom-hammock between them, without crossing a mesh.}

\vspace{10pt}

Case $9$: If both $m$-diagonals are of type $2$, it means we are in the situation of figure \ref{fig:t4} in the tube:

\begin{figure}[!h]
\centering
\begin{tikzpicture}[scale=0.6]
\fill[fill=black,fill opacity=0.1] (0,4) -- (-1.66,3.46) -- (-2.69,2.05) -- (-2.68,0.3) -- (-1.66,-1.11) -- (0,-1.65) -- (1.66,-1.11) -- (2.69,0.3) -- (2.69,2.05) -- (1.66,3.46) -- cycle;
\draw [fill=black,fill opacity=1.0] (-1.66,1.57) circle (0.2cm);
\draw [fill=black,fill opacity=1.0] (1.66,1.58) circle (0.2cm);
\draw (0,4)-- (-1.66,3.46);
\draw (-1.66,3.46)-- (-2.69,2.05);
\draw (-2.69,2.05)-- (-2.68,0.3);
\draw (-2.68,0.3)-- (-1.66,-1.11);
\draw (-1.66,-1.11)-- (0,-1.65);
\draw (0,-1.65)-- (1.66,-1.11);
\draw (1.66,-1.11)-- (2.69,0.3);
\draw (2.69,0.3)-- (2.69,2.05);
\draw (2.69,2.05)-- (1.66,3.46);
\draw (1.66,3.46)-- (0,4);
\draw(-1.66,1.17) circle (0.4cm);
\draw(1.66,1.18) circle (0.4cm);
\draw [shift={(-1.61,1.17)}] plot[domain=-0.61:0.61,variable=\t]({1*3.99*cos(\t r)+0*3.99*sin(\t r)},{0*3.99*cos(\t r)+1*3.99*sin(\t r)});
\draw [shift={(0.29,0.28)}] plot[domain=0.64:3.13,variable=\t]({1*2.98*cos(\t r)+0*2.98*sin(\t r)},{0*2.98*cos(\t r)+1*2.98*sin(\t r)});
\end{tikzpicture}
\caption{Case of the "$m$-ears"}
\label{fig:t4}
\end{figure}

\cor{If one of the arc is an $m$-ear, we only need to move the other to conclude. If not, we use Lemma \ref{lem:succ} to cut along an new $m$-ear $\gamma$ which does not cross any of the arcs.}

\vspace{10pt}

In any case, we have shown that if $\alpha$ crosses $\beta$, then there exists $k$ such that \[{\mathrm{Ext}}_\mathcal{C}^k(X_\alpha,X_\beta) \neq 0.\]
\end{proof}

\section{Compatibility with the flip and bijection between $m$-cluster-tilting objects and $(m+2)$-angulations}

With theorem \ref{th:cross}, we are able to define an $(m+2)$-angulation from an $m$-cluster-tilting object.

\cor{Let $T=\bigoplus_{i=1}^{n+1} T_i$ be an $m$-cluster-tilting object, and $T_i$ its $m$-rigid indecomposable summands. From Theorem \ref{th:ra}, for each $i \in \{1,\cdots,n\}$, we can associate with $T_i$, the $m$-diagonal $\alpha_i$. From Theorem \ref{th:cross}, we know that the $\alpha_i$ do not cross each other. Then the set $\{\alpha_i, i \in \{1,\cdots,n+1\}\}$ form a maximal set of noncrossing $m$-diagonals, which is an $(m+2)$-angulation.}

\begin{defi}\label{def:mcto}
\cor{We define this $(m+2)$-angulation as the $(m+2)$-angulation $\Delta_T$, associated with the $m$-cluster-tilting object $T$.}
\end{defi}

We first show the theorem of compatibility between the flip of an $(m+2)$-angulation, and the mutation of an $m$-cluster-tilting object.

\cor{Let \[\Delta=\{\alpha_i, i \in \{1,\cdots,n+1\}\}\] be an $(m+2)$-angulation. For any $i \in \{1,\cdots,n+1\}$, let $X_i$ be the $m$-rigid indecomposable object associated with $\alpha_i$. Let $X=\bigoplus_{i=1}^{n+1}X_i$ be its associated object. This object is $m$-cluster-tilting since it is the sum of $n+1$ $m$-rigid indecomposable objects.}

\begin{theo}\label{th:comp}
Let $\Delta$ be an $(m+2)$-angulation, and let $X$ be its associated object as we introduce just before.

\cor{Let $i \in \,\cdots,n+1\}$. Let $\mu_i$ be the flip of $\Delta$ at the $m$-diagonal $\alpha_i$. Let $\tilde{\mu}_i$ be the mutation of the $m$-cluster-tilting object $X$ at summand $X_i$. Then we have: \[ \mu_i(\Delta) = \Delta_{\mu_i(X)} \]}
\end{theo}

\begin{rmk}
In fact, the $(m+2)$-angulation $\Delta$ and $\mu_i(\Delta)$ (respectively $X$ and $\mu_i(X)$) are identical in all their components, except the $i^{th}$ component.
\end{rmk}

\begin{proof}
By Buan and Thomas in \cite{BT}, we know that there is a triangle \[X_i \to B_i^{(0)} \to X_i^{(1)} \to,\] where $B_i^{(0)} \in {\mathrm{add}}T$. The aim of the proof is to show that $X_i^{(1)} \simeq X_{\tilde{i}}$, where $X_{\tilde{i}}$ is the $m$-rigid corresponding to the arc $\overline{\alpha_i}=\kappa_i(\alpha)$ which is the twist of $\alpha_i$.

Let $\overline{X}=X/X_i$ be an almost $m$-cluster-tilting object. Then from Wraalsen (\cite{W}) and Zhou, Zhu (\cite{ZZ}), $\overline{X}$ has $m+1$ complements.

\cor{Let \[\overline{\Delta}=\{\alpha_j, j \in \{1,\cdots,n+1\}\setminus \{i\}\}\] be the "almost" $(m+2)$-angulation, containing all arcs of $\Delta$ except $\alpha_i$. Note that for each $j \in \{1,\cdots,n+1\}\setminus \{i\}\}$, the $m$-diagonal $\alpha_j$ corresponds to $X_j$. We can say that $\overline{\Delta}$ corresponds to $\overline{X}$.}

Let \[{\mathcal{U}}=\{ Y \in {\mathcal{C}}^{m}_Q, \forall k \in \{1,\cdots,m\}, {\mathrm{Ext}}_\mathcal{C}^k(\overline{X},Y)=0 \}.\]

\cor{Let $\mathcal{C}'=\mathcal{U}/(\overline{X})$ be the Iyama-Yoshino reduction of the higher cluster category $\mathcal{C}_Q^m$ at $\overline{X}$. We define the twist in this category $\langle 1 \rangle$.}

\cor{Then by Theorem \ref{th:cross}, an object in ${\mathcal{U}}$ corresponds to an arc which does not cross $\overline{\Delta}$. There are $m+1$ possibilities of remaining arcs in order to have an $(m+2)$-angulation. Indeed, removing one arc od $\Delta$ is the first step to the flip process.}

Then by Keller, and Iyama and Yoshino in \cite{Kel03} and \cite{IY}, ${\mathcal{C}}'$ is a triangulated, hom-finite, algebraic and $(m+1)$-Calabi-Yau category. Moreover, each arc which does not cross $\overline{\Delta}$ corresponds to an $m$-cluster-tilting object in ${\mathcal{C}}'$. In addition, for any $i\in \{1,\cdots,n+1\}$, if $\tilde{X}_i$ is the object corresponding to $\alpha_i$ in the category $\mathcal{C}'$, then we have: \[{\mathrm{Ext}}_\mathcal{C'}^k(\tilde{X}_i,\tilde{X}_i)={\mathrm{Ext}}_\mathcal{C'}^{-k}(\tilde{X}_i,\tilde{X}_i)=0\] for all $k \in \{1,\cdots,m \}$. Indeed, the $m$-diagonal $\alpha_i$ does not cross itself.

The algebra ${\mathrm{End}}(\alpha_i)=K$ is hereditary since it is of global dimension $0$. Then, by \cite[Theorem $4.2$]{KR}, we have an equivalence \[ {\mathcal{C}}' \simeq {\mathcal{C}}^{m}_{A_1}. \]
\cor{Therefore we have a distinguished triangle \[ \tilde{X}_i \to E_i \to \tilde{X}_i\langle 1 \rangle \to \tilde{X}_i[1], \] where $E_i$ is the set of arcs which follow $\alpha_i$ in the sense of its quiver.}

\cor{Note that $\langle 1 \rangle$ is the shift in the category ${\mathcal{C}}^{m}_{A_1}$, which means the shift in the remaining $(2m+2)$-gon.
Then it follows that \[\tilde{X}_i\langle 1 \rangle =X_{\overline{\alpha_i}}.\] Then we have two distinguished triangles:}
\cor{\[
\xymatrix{
X_i \ar[r] \ar[d]_*[@]{\sim} & B_i^{(0)} \ar[r] \ar[d]_*[@]{\sim} & X_i^{(1)} \ar[r] & X_i[1] \ar[d]_*[@]{\sim} \\
\tilde{X}_i \ar[r] & E_i \ar[r] & X_{\overline{\alpha_i}} \ar[r] & \tilde{X}_i[1]} \]}

By TR3, the third axiom of triangulated categories, we have a morphism $X_i^{(1)} \to X_{\overline{\alpha_i}}$

\[
\xymatrix{
X_i \ar[r] \ar[d]_*[@]{\sim} & B_i^{(0)} \ar[r] \ar[d]_*[@]{\sim} & X_i^{(1)} \ar[r] \ar[d] & \Sigma X_i \ar[d]_*[@]{\sim} \\
\tilde{X}_i \ar[r] & E_i \ar[r] & X_{\overline{\alpha_i}} \ar[r] & \tilde{X}_i[1]} \]

By the five lemma applied to triangulated categories, we have an isomorphism

\[
\xymatrix{
X_i \ar[r] \ar[d]_*[@]{\sim} & B_i^{(0)} \ar[r] \ar[d]_*[@]{\sim} & X_i^{(1)} \ar[r] \ar[d]_*[@]{\sim} & \Sigma X_i \ar[d]_*[@]{\sim} \\
\tilde{X}_i \ar[r] & E_i \ar[r] & X_{\overline{\alpha_i}} \ar[r] & \tilde{X}_i[1]} \]

Then we have shown that \[ X_i^{(1)} \simeq X_{\tilde{i}}. \]
\end{proof}

\begin{lem}\label{lem:arceq}
\cor{Let $Q$ be a quiver of type $A$, $D$, $\tilde{A}$, or $\tilde{D}$. We consider the higher cluster category $\mathcal{C}_Q^m$. Let $P$ be a polygon, realizing the higher cluster category. Let $\alpha$ and $\beta$ be two $m$-diagonals in $P$. Let $X_\alpha$ (respectively $X_\beta$) be the $m$-rigid indecomposable object associated with $\alpha$ (respectively with $\beta$). Then we have the following equivalence:}
\[ \forall i \in \{1,\cdots,m \}, {\mathrm{Ext}}_\mathcal{C}^i(X_\alpha,X_\beta)=0 \Leftrightarrow \alpha \text{ and } \beta \text{ do not cross}. \]
\end{lem}

\begin{proof}
The direct implication is exactly Theorem \ref{th:cross}.

\cor{Let us now suppose that $\alpha$ and $\beta$ do not cross each other. We complete this set of two $m$-diagonals into an $(m+2)$-angulation $\Delta$ containing the $m$-diagonals $\alpha_1,\cdots,\alpha_{n-1}$. For each $i \in \{1,\cdots,n-1\}$ let $X_i$ be the $m$-rigid indecomposable object associated with $\alpha_i$.}

\cor{Let $X$ be the object defined by \[X=X_\alpha \oplus X_\beta \bigoplus_{i=1}^{n-1} X_i.\] This object os the sum of $n+1$ $m$-rigid indecomposable objects, therefore, it is $m$-clustet-tilting. Then its summands are without self-extension, so:}

\[\forall i \in \{1,\cdots,m \}, {\mathrm{Ext}}_\mathcal{C}^i(X_\alpha,X_\beta)=0.\]

%
%
%
\end{proof}

Finally, we show that there exists a bijection between $(m+2)$-angulations and $m$-cluster-tilting objects.

\begin{theo}\label{th:bij}
\cor{Let $Q$ be a quiver of type $A$, $D$, $\tilde{A}$, or $\tilde{D}$. We consider the higher cluster category $\mathcal{C}_Q^m$. Let $P$ be a polygon, realizing the higher cluster category. Let $\mathcal{P}(P)$ be the set of all existing $(m+2)$-angulations of $P$. Consider the following application:
\[\begin{array}{cccc}
\Phi : & \mathcal{P}(P) & \to &\{\text{isoclasses of } m-\text{cluster-tilting objects}\}  \\
& \Delta & \mapsto & X_\Delta
\end{array}\]}

This application is a bijection.
\end{theo}

\begin{proof}
\cor{By lemma \ref{lem:arceq}, with an $(m+2)$-angulation, we associate a unique $m$-cluster-tilting object by taking the sum of the $m$-rigid indecomposable objects corresponding to each arc. Therefore the application is well-defined.}

\cor{We are going to show that any $m$-cluster-tilting object has a unique antecedent by $\Phi$. If we take $X$ an $m$-cluster-tilting object, we can associate a unique $(m+2)$-angulation. Indeed, $X= \bigoplus X_i$, where the $X_i$ are $m$-rigid indecomposable objects. With each summand $X_i$, we associate the corresponding arc $\alpha_i$ from Theorem \ref{th:ra}. We have:
\[\forall k \in \{1,\cdots,m \},\forall i,j \in \{1,\cdots,n+1 \}, {\mathrm{Ext}}_\mathcal{C}^k(X_i,X_j)=0.\]}

\cor{By Theorem \ref{th:cross}, we know that, for any $i,j \in \{1,\cdots,n+1 \}$ the $m$-diagonal $\alpha_i$ do not cross $\alpha_j$. There are $n+1$ such $m$-diagonals, so they form a maximal set of noncrossing $m$-diagonals, thus an $(m+2)$-angulation. It is uniquely defined. This achieves the proof.}

\end{proof}

\vspace{20pt}

We can summarize all the important properties between $m$-cluster-tilting objects, colored quivers, and $(m+2)$-angulations in the following diagram:

\[ \scalebox{1.2}{ \xymatrix {
& (m+2)\text{-angulation } \Delta \ar^[@]{\hbox to 0pt{\hss \text{\scriptsize{Theorem \ref{th:comp}}}\hss}}[ddr] \ar[ddl] & \\
& & \\
\text{Colored quiver } Q_\Delta \ar@{-}^[@]{\hbox to 0pt{\hss \text{\scriptsize{Theorem \ref{theo:corresp}}}\hss}}[uur] \ar@{-}^[@]{\hbox to 0pt{\hss \text{\scriptsize{Theorem \ref{th:mut}}}\hss}}[rr] & & m\text{-cluster-tilting object } X_\Delta \ar[uul] \ar[ll]
}} \]

We now finish this section with a direct consequence of this diagram.

\begin{theo}
\cor{Let $Q$ be a quiver of type $A$, $D$, $\tilde{A}$, or $\tilde{D}$. We consider the higher cluster category $\mathcal{C}_Q^m$. Let $P$ be a polygon, realizing the higher cluster category. Let $\Delta$ be an $(m+2)$-angulation. Let $Q_\Delta$ be the associated colored quiver. Let $X_\Delta$ be the $m$-cluster-tilting object associated with $\Delta$, and let $Q_{X_\Delta}$ be the quiver associated with $X_\Delta$ in the sense of Buan and Thomas in \cite{BT}. Then
\[ Q_\Delta=Q_{X_\Delta} \]}
\end{theo}

Note here that theorem \ref{theo:corresp} is a direct consequence of theorems \ref{th:comp} and \ref{th:mut}.

\bibliographystyle{alpha}
\bibliography{biblio}

\newcommand{\etalchar}[1]{$^{#1}$}
\def\ocirc#1{\ifmmode\setbox0=\hbox{$#1$}\dimen0=\ht0 \advance\dimen0
  by1pt\rlap{\hbox to\wd0{\hss\raise\dimen0
  \hbox{\hskip.2em$\scriptscriptstyle\circ$}\hss}}#1\else {\accent"17 #1}\fi}
  \def\ocirc#1{\ifmmode\setbox0=\hbox{$#1$}\dimen0=\ht0 \advance\dimen0
  by1pt\rlap{\hbox to\wd0{\hss\raise\dimen0
  \hbox{\hskip.2em$\scriptscriptstyle\circ$}\hss}}#1\else {\accent"17 #1}\fi}
  \def\ocirc#1{\ifmmode\setbox0=\hbox{$#1$}\dimen0=\ht0 \advance\dimen0
  by1pt\rlap{\hbox to\wd0{\hss\raise\dimen0
  \hbox{\hskip.2em$\scriptscriptstyle\circ$}\hss}}#1\else {\accent"17 #1}\fi}
\begin{thebibliography}{BMR{\etalchar{+}}06}

\bibitem[ASS06]{ASS}
Ibrahim Assem, Daniel Simson, and Andrzej Skowro{\'n}ski.
\newblock {\em Elements of the representation theory of associative algebras.
  {V}ol. 1}, volume~65 of {\em London Mathematical Society Student Texts}.
\newblock Cambridge University Press, Cambridge, 2006.
\newblock Techniques of representation theory.

\bibitem[BIRS09]{BIRS}
A.~B. Buan, O.~Iyama, I.~Reiten, and J.~Scott.
\newblock Cluster structures for 2-{C}alabi-{Y}au categories and unipotent
  groups.
\newblock {\em Compos. Math.}, 145(4):1035--1079, 2009.

\bibitem[BM07]{BM02}
Karin Baur and Robert~J. Marsh.
\newblock A geometric description of the {$m$}-cluster categories of type
  {$D_n$}.
\newblock {\em Int. Math. Res. Not. IMRN}, (4):Art. ID rnm011, 19, 2007.

\bibitem[BM08]{BM01}
Karin Baur and Robert~J. Marsh.
\newblock A geometric description of {$m$}-cluster categories.
\newblock {\em Trans. Amer. Math. Soc.}, 360(11):5789--5803, 2008.

\bibitem[BMR{\etalchar{+}}06]{BMRRT}
Aslak~Bakke Buan, Robert Marsh, Markus Reineke, Idun Reiten, and Gordana
  Todorov.
\newblock Tilting theory and cluster combinatorics.
\newblock {\em Adv. Math.}, 204(2):572--618, 2006.

\bibitem[BT09]{BT}
Aslak~Bakke Buan and Hugh Thomas.
\newblock Coloured quiver mutation for higher cluster categories.
\newblock {\em Adv. Math.}, 222(3):971--995, 2009.

\bibitem[BT15]{BauTor}
Karin Baur and Edmund~André Torkildsen.
\newblock A geometric realization of tame categories.
\newblock 2015.

\bibitem[CCS06]{CCS}
Philippe. Caldero, Frédéric. Chapoton, and Ralf. Schiffler.
\newblock Quivers with relations arising from clusters ({$A_n$} case).
\newblock {\em Trans. Amer. Math. Soc.}, 358(3):1347--1364, 2006.

\bibitem[DR74]{DR}
Vlastimil Dlab and Claus~Michael Ringel.
\newblock {\em Representations of graphs and algebras}.
\newblock Department of Mathematics, Carleton University, Ottawa, Ont., 1974.
\newblock Carleton Mathematical Lecture Notes, No. 8.

\bibitem[FR05]{FR}
Sergey Fomin and Nathan Reading.
\newblock Generalized cluster complexes and {C}oxeter combinatorics.
\newblock {\em Int. Math. Res. Not.}, (44):2709--2757, 2005.

\bibitem[FST08]{FST}
Sergey Fomin, Michael Shapiro, and Dylan Thurston.
\newblock Cluster algebras and triangulated surfaces. {I}. {C}luster complexes.
\newblock {\em Acta Math.}, 201(1):83--146, 2008.

\bibitem[FZ02]{FZ}
Sergey Fomin and Andrei Zelevinsky.
\newblock Cluster algebras. {I}. {F}oundations.
\newblock {\em J. Amer. Math. Soc.}, 15(2):497--529 (electronic), 2002.

\bibitem[IY08]{IY}
Osamu Iyama and Yuji Yoshino.
\newblock Mutation in triangulated categories and rigid {C}ohen-{M}acaulay
  modules.
\newblock {\em Invent. Math.}, 172(1):117--168, 2008.

\bibitem[JM]{JM}
Lucie Jacquet-Malo.
\newblock A geometric realization of the m-cluster categories of type
  $\tilde{D}_n$.
\newblock {\em To appear}.

\bibitem[Kel05]{Kel03}
Bernhard Keller.
\newblock On triangulated orbit categories.
\newblock {\em Doc. Math.}, 10:551--581, 2005.

\bibitem[Kel11]{Kel01}
B.~Keller.
\newblock Cluster algebras and cluster categories.
\newblock {\em Bull. Iranian Math. Soc.}, 37(2):187--234, 2011.

\bibitem[KR08]{KR}
Bernhard Keller and Idun Reiten.
\newblock Acyclic {C}alabi-{Y}au categories.
\newblock {\em Compos. Math.}, 144(5):1332--1348, 2008.
\newblock With an appendix by Michel Van den Bergh.

\bibitem[MP14]{MP}
Robert~J. Marsh and Yann Palu.
\newblock Coloured quivers for rigid objects and partial triangulations: the
  unpunctured case.
\newblock {\em Proc. Lond. Math. Soc. (3)}, 108(2):411--440, 2014.

\bibitem[Sch08]{Sch}
Ralf Schiffler.
\newblock A geometric model for cluster categories of type {$D_n$}.
\newblock {\em J. Algebraic Combin.}, 27(1):1--21, 2008.

\bibitem[Tho07]{Tho}
Hugh Thomas.
\newblock Defining an {$m$}-cluster category.
\newblock {\em J. Algebra}, 318(1):37--46, 2007.

\bibitem[Tor12a]{T}
Hedmund~André Torkildsen.
\newblock A geometric realization of the $m$ cluster category of type
  $\tilde{A}$.
\newblock 2012.

\bibitem[Tor12b]{Tor}
Hermund~André Torkildsen.
\newblock A geometric realization of the $m$-cluster category of type
  $\tilde{A}$, 2012.

\bibitem[Tza06]{Tz}
Eleni Tzanaki.
\newblock Polygon dissections and some generalizations of cluster complexes.
\newblock {\em J. Combin. Theory Ser. A}, 113(6):1189--1198, 2006.

\bibitem[Wr{\ocirc{a}}09]{W}
Anette Wr{\ocirc{a}}lsen.
\newblock Rigid objects in higher cluster categories.
\newblock {\em J. Algebra}, 321(2):532--547, 2009.

\bibitem[Zhu08]{Z}
Bin Zhu.
\newblock Generalized cluster complexes via quiver representations.
\newblock {\em J. Algebraic Combin.}, 27(1):35--54, 2008.

\bibitem[ZZ09]{ZZ}
Yu~Zhou and Bin Zhu.
\newblock Cluster combinatorics of {$d$}-cluster categories.
\newblock {\em J. Algebra}, 321(10):2898--2915, 2009.

\end{thebibliography}

\end{document}